\documentclass[11pt,reqno]{amsart}
\usepackage[T1]{fontenc}
\usepackage[utf8]{inputenc}
\usepackage{lmodern,amsmath,amssymb,amsthm,mathtools}
\usepackage[a4paper,margin=28mm]{geometry}
\usepackage{microtype,enumitem}
\usepackage[hidelinks]{hyperref}
\numberwithin{equation}{section}
\newtheorem{theorem}{Theorem}[section]
\newtheorem{proposition}[theorem]{Proposition}
\newtheorem{lemma}[theorem]{Lemma}
\newtheorem{corollary}[theorem]{Corollary}
\theoremstyle{definition}
\newtheorem{definition}[theorem]{Definition}
\newtheorem{remark}[theorem]{Remark}
\newtheorem*{remark*}{Remark}
\newtheorem{example}[theorem]{Example}
\newtheorem*{notation*}{Notation}
\newtheorem*{claim*}{Claim}

\newcommand{\R}{\mathbb R}
\newcommand{\N}{\mathbb N}
\newcommand{\Sig}{\boldsymbol\Sigma}
\newcommand{\PiB}{\boldsymbol\Pi}
\newcommand{\norm}[1]{\left\lVert#1\right\rVert}
\newcommand{\abs}[1]{\left\lvert#1\right\rvert}
\newcommand{\cl}[2]{\overline{#1}^{\,#2}}
\providecommand{\MR}[1]{\relax\ifhmode\unskip\space\fi
  \href{https://mathscinet.ams.org/mathscinet/article?mr=#1}{MR#1}}
\providecommand{\bysame}{\leavevmode\hbox to 3em{\hrulefill}\thinspace}
\setlist[enumerate]{label=\textup{(\arabic*)},leftmargin=*}
\title[Complexity classes of Polish subgroups]{Complexity classes of Polish subgroups}
\author{Martino Lupini}
\address{Dipartimento di Matematica, Universit\`a di Bologna,
Piazza di Porta S.~Donato 5, 40126 Bologna, Italy}
\email{martino.lupini@unibo.it}
\urladdr{https://www.lupini.org/}
\date{25 November 2026}
\begin{document}
\begin{abstract}
A \emph{Polish subgroup} $H$ of a Polish group $G$, also called Polishable, is the image of a continuous homomorphism from a Polish group into $G$;
it has \emph{complexity class} $\Gamma$ if $H\in\Gamma$ and $G\setminus H\notin\Gamma$.
For $\alpha<\omega_1$, the smallest $\boldsymbol\Pi^0_{1+\alpha+1}$ subgroup $s^H_\alpha(G)$ of $G$ containing $H$ is Polish; the Solecki rank of $H$ is the least $\alpha$ with $s^H_\alpha(G)=H$.
Refining a theorem of Farah and Solecki, we determine the possible complexity classes: for nonopen $H$ of rank $\rho$ they are exactly $\boldsymbol\Pi^0_{1+\rho}$ if $\rho$ is zero or a limit,
and $\boldsymbol\Sigma^0_{1+\rho}$, $D(\boldsymbol\Pi^0_{1+\rho})$, $\boldsymbol\Pi^0_{1+\rho+1}$ if $\rho$ is a successor, where $D(\Gamma)=\{A\setminus B:A,B\in\Gamma\}$.
For successor $\rho$, with each $s^H_\alpha(G)$ in its Polish topology and $S=s^H_{\rho-1}(G)$, $H$ has class $\boldsymbol\Pi^0_{1+\rho+1}$ unless $H\in D(\boldsymbol\Pi^0_2)(S)$,
and $\boldsymbol\Sigma^0_{1+\rho}$ if and only if $H\in\boldsymbol\Sigma^0_2(S)$ when $\rho-1$ is zero or a limit, and $H\in\boldsymbol\Sigma^0_3(s^H_{\rho-2}(G))$ when $\rho-1$ is a successor.
If the Polish group $H$ is non-Archimedean, or merely its coset relation is classifiable by countable structures, exactly $\boldsymbol\Sigma^0_{1+\rho}$ for successor $\rho-1$ and $D(\boldsymbol\Pi^0_{1+\rho})$ for zero or limit $\rho-1$ are excluded;
the non-Archimedean case is implicit in work of Hjorth, Kechris, and Louveau.
Ranges of operators between separable Fr\'echet spaces realize every class,
and between separable Banach spaces exactly those other than $\boldsymbol\Pi^0_\lambda$, $\lambda$ a limit, each by a compact injective operator into any separable infinite-dimensional Banach space.
Banakh, Dobrowolski, and Plichko stated most of the Banach list and published in 2000 the question whether an operator between Fr\'echet spaces can have range of class $D(\boldsymbol\Pi^0_{\lambda+1})$, $\lambda$ a limit; the compact operators above answer it affirmatively.
By a 2000 theorem of the same authors, a $\boldsymbol\Sigma^0_3$ and $\boldsymbol\Pi^0_3$ coefficient space of a Schauder basis is the intersection of an $F_\sigma$ and a $G_\delta$ set,
which answers negatively a question Ding raised in a 2015 preprint; we extend this to all Polish subgroups.
This is a corrected and extended version of the author's 2026 article, incorporating its corrigendum.
\end{abstract}

\maketitle
\enlargethispage{2pt}

\section{Introduction}\label{sec:introduction}

A continuous homomorphism between Polish groups can have a range that
is far from closed. The range nevertheless carries a natural Polish
group topology, inherited from the quotient by the kernel, and it is a
Borel subset of the target. We call a subgroup $H$ of a Polish group
$G$ a \emph{Polish subgroup} if it admits a Polish group topology with
the same Borel sets as the subspace topology inherited from $G$.
Equivalently, $H$ is the range of a continuous homomorphism from a
Polish group into $G$. In the descriptive set theory literature these
subgroups are called Polishable
\cite[Section~1.6]{becker_descriptive_1996},
\cite[Section~1.2]{farah_borel_2006}. The intrinsic Polish
topology is unique and can be strictly finer than the subspace
topology. Countable subgroups with the discrete topology are Polish
subgroups, and so are the spaces $\ell_p$, $1\leq p<\infty$, and $c_0$
inside $\mathbb R^{\mathbb N}$ \cite[Section~8.2]{gao_invariant_2009}.
More generally, the range of a continuous linear map between separable
Fr\'echet spaces is a Polish subgroup of the target, and its natural
complete topology can differ substantially from the inherited one.

Polish subgroups are the subgroups whose coset spaces behave like
orbit spaces of Polish group actions. We use left cosets and write
\[
 xE_H^Gy\quad\Longleftrightarrow\quad x^{-1}y\in H
\]
for the coset relation of a subgroup $H\leq G$. When $H$ is a Polish
subgroup, $E_H^G$ is the orbit relation of the continuous action
$h\cdot x=xh^{-1}$ of the Polish group $H$ on $G$. For abelian Borel
subgroups, Polishability is visible in the coset relation: either $H$
is a Polish subgroup or $E_1$ embeds continuously into $E_H^G$, by a
dichotomy of Solecki \cite[Theorem~1.1]{solecki_coset_2009}, and
$E_1$ is not Borel reducible to the orbit relation of a continuous
action of a Polish group, by a theorem of Kechris and Louveau recalled
in \cite[p.~572]{solecki_coset_2009}.

The Borel complexity of a Polish subgroup measures the complexity of
its coset relation. An equivalence relation is \emph{potentially}
in a Borel pointclass if some Polish topology with the same Borel sets
puts it in that pointclass; the notion is due to Louveau
\cite[Definition~1, p.~152]{louveau_reducibility_1994}. If $H$ is a
Polish subgroup of $G$, $1\leq\alpha<\omega_1$, and $\Gamma$ is one of
$\boldsymbol\Sigma^0_\alpha$, $\boldsymbol\Pi^0_\alpha$, and
$D(\boldsymbol\Pi^0_\alpha)$, then
\[
 E_H^G\text{ is potentially }\Gamma
 \quad\Longleftrightarrow\quad H\in\Gamma(G)
\]
(Theorem~\ref{Theorem:Polishable-complexity}); for $\alpha\geq3$ this
is due to Hjorth, Kechris, and Louveau
\cite[Propositions~5.1--5.2]{hjorth_borel_1998}. Determining the
possible complexity classes of Polish subgroups is therefore the same
problem as determining the possible potential classes of coset
relations. For the orbit relations of Borel actions of closed subgroups
of $S_\infty$, Hjorth, Kechris, and Louveau determined the complete list
of potential classes \cite[Theorem~1]{hjorth_borel_1998}. They also
observed that very little was known about the potential Wadge classes
of the orbit relations of actions of general Polish groups; for
instance, it was unknown whether $\boldsymbol\Sigma^0_3$ can be such a
class \cite[p.~68]{hjorth_borel_1998}. Coset relations of Polish
subgroups, which are the orbit relations of translation actions, are a
natural family in which to study these questions, and
$\boldsymbol\Sigma^0_3$ already occurs among them
(Subsection~\ref{sec:intro-questions}).

Quotients by Polish subgroups also arise outside descriptive set
theory. If $H$ is normal, then $G/H$ is a \emph{group with a Polish
cover}. The \v{C}ech cohomology groups of locally compact separable
metrizable spaces, for instance, are naturally groups with a Polish
cover \cite{BLP}, and abelian groups with a Polish cover form the left
heart of the quasi-abelian category of abelian Polish groups
\cite{Heart}. A Borel map inducing a bijection of coset spaces
preserves the complexity class of the subgroup
(Corollary~\ref{cor:coset-bijection}), so the complexity class of $H$
in $G$ is an invariant of the quotient $G/H$ rather than of its
presentation, and \cite[Section~4]{Heart} reformulates in this setting
the results of \cite{L}.

The possible complexities were first studied for linear subspaces.
Saint Raymond associated with every continuous injective linear map
between separable Fr\'echet spaces a canonical transfinite sequence of
separable Fr\'echet models. He used it to relate the Borel class of the
range to the Borel class of the inverse map and to the weak-star derived
sets of the range of the adjoint, and he constructed operator ranges of
arbitrarily high multiplicative class \cite{saint-raymond_espaces_1976};
see \cite[Theorem~1.1]{BDP} for the range--inverse theorem. His theory
also shows that, for $\lambda$ zero or a limit, an operator range in
$\boldsymbol\Pi^0_{1+\lambda+1}$ already lies in
$\boldsymbol\Pi^0_{1+\lambda}$, and that no range of an operator
with Banach domain has complexity class $\boldsymbol\Pi^0_\lambda$
at a positive limit $\lambda$ \cite[Theorem~1.2]{BDP}; compare
\cite[Th\'eor\`eme~38]{saint-raymond_espaces_1976}. Banakh,
Dobrowolski, and Plichko proved that an operator range in an ambiguous
class $\boldsymbol\Delta^0_{1+\alpha+1}$, $\alpha\geq1$, lies in
$D(\boldsymbol\Pi^0_{1+\alpha})$ \cite[Theorem~1.3]{BDP}, and they
stated realizations of most of the remaining classes by compact
operators between Banach spaces
\cite[Theorems~2.6--2.8 and Proposition~2.10]{BDP}.

For arbitrary Polish subgroups, Solecki introduced a canonical
approximation by Polish subgroups
\cite[Section~2]{solecki_polish_1999}, and Farah and Solecki related
its length to the least multiplicative class containing the subgroup
\cite[Theorem~3.1]{farah_borel_2006}. Subgroups of prescribed length
were constructed by Saint Raymond
\cite[Th\'eor\`eme~21]{saint-raymond_espaces_1976}, by Hjorth, Kechris,
and Louveau \cite[Lemmas~5.4--5.6]{hjorth_borel_1998}, and by Farah and
Solecki \cite[Section~5]{farah_borel_2006}; see
\cite[p.~502]{farah_borel_2006}. For non-Archimedean Polish subgroups,
the classification theorem of Hjorth, Kechris, and Louveau quoted
above, together with their analysis of Polish subgroups, gives the
complete list of classes,
and their examples realize it in $\mathbb Z_2^{\mathbb N}$
\cite[Theorem~1, Propositions~5.1--5.2 and Section~5]{hjorth_borel_1998}.
Hjorth showed that every uncountable abelian Polish group has Polish
subgroups of arbitrarily high Borel complexity
\cite[Introduction and Corollary~3.5]{hjorth_subgroups_2006}, and Ding
and Gao \cite[Theorem~1.1]{ding_separable_2008} and Malicki
\cite[Corollary~4]{malicki_polishable_2008} showed the same for the
Banach subspaces of every infinite-dimensional separable Banach
space.

These results determine the least multiplicative class containing a
Polish subgroup, and they bound the possible classes in the linear and
non-Archimedean settings. They leave open which nonselfdual Borel
classes occur for Polish subgroups of arbitrary Polish groups, and how
the class can be read off from the subgroup. At a successor rank the
multiplicative information leaves three candidates: an additive class,
a difference class, and a multiplicative class. The additive candidate
is the delicate one. A Polish subgroup that is both
$\boldsymbol\Sigma^0_3$ and $\boldsymbol\Pi^0_3$ is the intersection of
an $F_\sigma$ set and a $G_\delta$ set
(Proposition~\ref{prop:pi3-local}), but a $\boldsymbol\Sigma^0_3$
Polish subgroup need not be $\boldsymbol\Pi^0_3$, even among the
Banach subspaces of $\ell_2$
(Corollary~\ref{cor:additive-counterexample}). The present paper
answers these questions in general and in the linear and
non-Archimedean settings.

This paper is a corrected and extended version of the author's
article \cite{L}. It incorporates, with complete proofs, the
corrections announced in \cite{Corr}.
Remark~\ref{rem:earlier-formulations} lists the statements of
\cite{L} that are corrected here, and
Subsection~\ref{sec:intro-relation} describes what is new relative to
\cite{L} and \cite{Corr}.

\subsection{Solecki stages and the classification}\label{sec:intro-classification}

Solecki's canonical approximation connects the intrinsic topology of a
Polish subgroup with the Borel structure of the ambient group. To a
Polish subgroup $H\leq G$ it associates the decreasing transfinite
sequence
\[
 s_0^H(G)=\overline H^{\,G},\qquad
 s_{\xi+1}^H(G)=s_1^H(s_\xi^H(G)),\qquad
 s_\lambda^H(G)=\bigcap_{\xi<\lambda}s_\xi^H(G)
\]
at nonzero limits $\lambda$, where $s_1^H(G)$ is the smallest
$\boldsymbol\Pi^0_3$ subgroup of $G$ containing $H$, which exists by
\cite[Lemma~2.3]{solecki_polish_1999}. Each stage is a Polish subgroup
containing $H$ as a dense subgroup, and we always equip it with its
intrinsic Polish topology. The sequence reaches $H$ at a countable
stage \cite[Theorem~2.1]{solecki_polish_1999}; the least such stage is
the \emph{Solecki rank} of $H$ in $G$. More generally, $s_\alpha^H(G)$
is the smallest $\boldsymbol\Pi^0_{1+\alpha+1}$ subgroup of $G$
containing $H$ (Theorem~\ref{Theorem:characterize-solecki}). Farah and
Solecki showed that the least multiplicative class containing a Polish
subgroup of rank $\rho$ is $\boldsymbol\Pi^0_{1+\rho}$ if $\rho$ is zero
or a limit, and $\boldsymbol\Pi^0_{1+\rho+1}$ if $\rho$ is a successor
\cite[Theorem~3.1]{farah_borel_2006}. The construction and the
topologies of the stages are recalled in Section~\ref{Section:solecki}.

We use the usual boldface Borel classes, with closed sets at level
$\boldsymbol\Pi^0_1$ and open sets at level $\boldsymbol\Sigma^0_1$.
The class $D(\boldsymbol\Pi^0_\xi)$ consists of differences of two
$\boldsymbol\Pi^0_\xi$ sets. Saying that a set has a nonselfdual
\emph{complexity class} $\Gamma$ means that it belongs to $\Gamma$
and does not belong to the dual class. The ordinal sums in the
following statements are in the displayed order; in particular,
$1+\lambda=\lambda$ at every nonzero limit ordinal $\lambda$.
Proper open subgroups are clopen, so the subgroup classification
is stated for nonopen subgroups. Our first result determines the
possible complexity classes in terms of the rank.

\begin{theorem}\label{Corollary:complexity}
Let $H$ be a nonopen Polish subgroup of a Polish group $G$, and let
$\rho$ be its Solecki rank.
\begin{enumerate}
\item If $\rho$ is zero or a limit ordinal, then the complexity class
of $H$ is $\boldsymbol\Pi^0_{1+\rho}$.
\item If $\rho$ is a successor, then the complexity class of $H$ is
one of $\boldsymbol\Sigma^0_{1+\rho}$, $D(\boldsymbol\Pi^0_{1+\rho})$,
and $\boldsymbol\Pi^0_{1+\rho+1}$. With $S=s_{\rho-1}^H(G)$, it is
$\boldsymbol\Pi^0_{1+\rho+1}$ unless $H\in D(\boldsymbol\Pi^0_2)(S)$.
It is $\boldsymbol\Sigma^0_{1+\rho}$ if and only if
$H\in\boldsymbol\Sigma^0_2(S)$ when $\rho-1$ is zero or a limit, and
if and only if $H\in\boldsymbol\Sigma^0_3(s_{\rho-2}^H(G))$ when
$\rho-1$ is a successor.
\end{enumerate}
Conversely, for every $\rho<\omega_1$, each class listed for $\rho$ is
the complexity class of a nonopen Polish subgroup of Solecki rank
$\rho$ of some Polish group. All of them, except
$\boldsymbol\Sigma^0_{1+\rho}$ with $\rho-1$ a successor, occur among
the Polish subgroups of $\mathbb Z^{\mathbb N}$.
\end{theorem}

Item~(1) and the trichotomy in Item~(2) refine the theorem of Farah and
Solecki: the rank fixes the least multiplicative class, and at a
successor rank the position of $H$ in the last two Solecki stages
decides among the three candidates
(Theorem~\ref{Theorem:complexity} and the discussion following it).
The relative classes are computed in the intrinsic Polish topologies of
the stages. When $\rho-1$ is a successor, relative
$\boldsymbol\Sigma^0_2$ membership in $S$ still leaves two
possibilities, and both occur: the summable sequence groups of
Theorem~\ref{Theorem:complexityP}(3) have the difference class, and the
operator ranges of Theorem~\ref{thm:additive-ranges} have the additive
class. Relative $\boldsymbol\Sigma^0_3$ membership one stage earlier
distinguishes them, so the two-stage test cannot be shortened.

\subsection{The Borel index}\label{sec:intro-index}

The classes in Theorem~\ref{Corollary:complexity} form a linearly
ordered scale of cumulative pointclasses, and it is convenient to
encode them by a single index. We use the fractional scale and the
cumulative pointclasses of
\cite[Sections~2.1--2.2]{lupiniDerivationLengthAutomorphism2026d}.
The ordered set $\omega_1[1/2]$ contains $\alpha$ and $\alpha+1/2$ for
every $\alpha<\omega_1$, with
\[
 \alpha<\alpha+\tfrac12<\alpha+1.
\]
Its refinement $\omega_1^{\mathrm{Ph}}$ also contains
$\alpha+1/2+\varepsilon$ for every $\alpha<\omega_1$ and
$\sigma+1/2-\varepsilon$ for every countable successor $\sigma$. Thus,
for $\lambda$ zero or a countable limit and $\sigma$ a countable
successor,
\[
 \lambda<\lambda+\tfrac12<\lambda+\tfrac12+\varepsilon<\lambda+1,
 \qquad
 \sigma<\sigma+\tfrac12-\varepsilon<\sigma+\tfrac12
        <\sigma+\tfrac12+\varepsilon<\sigma+1.
\]
These are formal symbols, not ordinal fractions or real perturbations;
only within these indices is $\varepsilon$ formal, and in analytic
estimates it keeps its usual meaning as a positive real. Addition on
the left is extended by
\[
 \beta+(\alpha+\tfrac12)=(\beta+\alpha)+\tfrac12,\qquad
 \beta+(\alpha+\tfrac12\pm\varepsilon)
       =(\beta+\alpha)+\tfrac12\pm\varepsilon
\]
for countable $\beta$, whenever the symbols involved are defined.

Table~\ref{tab:fractional-classes} assigns a \emph{cumulative
pointclass} $\Gamma_\xi$ to each $\xi\in\omega_1^{\mathrm{Ph}}$. A union
of pointclasses there means membership in either pointclass, not a
union of two sets. In particular,
\[
 \Gamma_{\sigma+1/2}=\Gamma_{\sigma+1/2+\varepsilon}
       =D(\boldsymbol\Pi^0_{1+\sigma+1})
 \quad(\sigma\text{ a successor}).
\]
The family $(\Gamma_\xi)_\xi$ is increasing, so every Borel subset $A$
of a Polish space $X$ has a least index, its \emph{Borel index}
\[
 \ell_{\boldsymbol\Delta}(A;X)
   :=\min\{\xi\in\omega_1^{\mathrm{Ph}}:A\in\Gamma_\xi(X)\}.
\]
For every $\xi\in\omega_1^{\mathrm{Ph}}$, we have
$\ell_{\boldsymbol\Delta}(A;X)\leq\xi$ exactly when
$A\in\Gamma_\xi(X)$. The repeated successor upper half-step is never a
least index, so the Borel index belongs to
\[
 \omega_1^{\boldsymbol\Delta}
 :=\omega_1^{\mathrm{Ph}}\setminus
       \{\sigma+\tfrac12+\varepsilon:\sigma<\omega_1
                                      \text{ a successor}\}.
\]
For a Polish subgroup $H\leq G$, abbreviate
$\ell_{\boldsymbol\Delta}(G/H):=\ell_{\boldsymbol\Delta}(H;G)$. For
Polish subgroups the Borel index determines the complexity class, and
rounding it up to an ordinal recovers the Solecki rank
(Section~\ref{Section:complexity-polishable}).

\begin{table}[ht]
\centering
\small
\renewcommand{\arraystretch}{1.3}
\begin{tabular}{@{}lll@{}}
\hline
Index $\xi$ & Cumulative pointclass $\Gamma_\xi$
            & Exact nonselfdual class \\
\hline
$\lambda$ & $\boldsymbol\Pi^0_{1+\lambda}$
           & $\boldsymbol\Pi^0_{1+\lambda}$ \\
$\sigma$ & $\boldsymbol\Pi^0_{1+\sigma+1}$
          & $\boldsymbol\Pi^0_{1+\sigma+1}$ \\
$\lambda+1/2$ & $\boldsymbol\Sigma^0_{1+\lambda+1}$
              & $\boldsymbol\Sigma^0_{1+\lambda+1}$ \\
$\lambda+1/2+\varepsilon$ & $D(\boldsymbol\Pi^0_{1+\lambda+1})$
                          & $D(\boldsymbol\Pi^0_{1+\lambda+1})$ \\
$\sigma+1/2-\varepsilon$
 & $\boldsymbol\Pi^0_{1+\sigma+1}\cup\boldsymbol\Sigma^0_{1+\sigma+1}$
 & $\boldsymbol\Sigma^0_{1+\sigma+1}$ \\
$\sigma+1/2$ & $D(\boldsymbol\Pi^0_{1+\sigma+1})$
             & $D(\boldsymbol\Pi^0_{1+\sigma+1})$ \\
$\sigma+1/2+\varepsilon$ & $D(\boldsymbol\Pi^0_{1+\sigma+1})$
                         & never a least index \\
\hline
\end{tabular}
\caption{Fractional indices: $\lambda$ is zero or a countable limit,
and $\sigma$ is a countable successor. The last column applies to
nonopen Polish subgroups of Borel index $\xi$, by
Theorem~\ref{Corollary:complexity}.}
\label{tab:fractional-classes}
\end{table}

The cumulative class at $\sigma+1/2-\varepsilon$ is selfdual, so the
Borel index must be distinguished from the nonselfdual complexity
class: the sets whose least index is $\sigma+1/2-\varepsilon$ are
exactly the $\boldsymbol\Sigma^0_{1+\sigma+1}$ sets that are not
$\boldsymbol\Pi^0_{1+\sigma+1}$. For Polish subgroups the other rows
also give exact nonselfdual classes, as asserted below; for arbitrary
Borel sets a least index alone need not imply nonselfdual exactness.
For example, the exact classes $\boldsymbol\Sigma^0_3$ and
$D(\boldsymbol\Pi^0_3)$ correspond to the Borel indices
$1+1/2-\varepsilon$ and $1+1/2$, respectively, and for a positive
countable limit $\lambda$ the exact class
$D(\boldsymbol\Pi^0_{\lambda+1})$ corresponds to
$\lambda+1/2+\varepsilon$. Two subsets of $\omega_1^{\boldsymbol\Delta}$
describe the restricted lists below:
\[
 \begin{split}
 \mathcal I_{\mathbb Z}
   &:=\omega_1[1/2]\cup
       \{\lambda+\tfrac12+\varepsilon:
                 \lambda<\omega_1\text{ zero or limit}\},\\
 \mathcal I_{\mathrm B}
   &:=\omega_1^{\boldsymbol\Delta}\setminus
       \{\lambda:0<\lambda<\omega_1\text{ a limit}\}.
 \end{split}
\]

In this notation, Theorem~\ref{Corollary:complexity} states that
every nonopen Polish subgroup $H$ of a Polish group $G$ satisfies
\begin{equation}\label{eq:upper}
 \ell_{\boldsymbol\Delta}(G/H)\in\omega_1^{\boldsymbol\Delta},
\end{equation}
that its complexity class is the class in the last column of
Table~\ref{tab:fractional-classes} at this index, that every index in
$\omega_1^{\boldsymbol\Delta}$ occurs, and that every index in
$\mathcal I_{\mathbb Z}$ is realized by a nonopen Polish subgroup of
$\mathbb Z^{\mathbb N}$. The indices outside $\mathcal I_{\mathbb Z}$
are the indices $\sigma+1/2-\varepsilon$ with $\sigma$ a successor;
each of them is realized by a Banach subspace of a Hilbert space
(Theorem~\ref{thm:additive-ranges}).

\subsection{Non-Archimedean subgroups and coset relations}\label{sec:intro-nonarch}

A Polish group is \emph{non-Archimedean} if its identity has a
neighborhood basis of open subgroups. The list becomes shorter when the
intrinsic Polish group $H$ is non-Archimedean.

\begin{theorem}\label{Corollary:complexity-nonArchimedean}
Let $H$ be a nonopen Polish subgroup of a Polish group $G$.
If the intrinsic Polish group $H$ is non-Archimedean, then
\[
 \ell_{\boldsymbol\Delta}(G/H)\in\omega_1[1/2].
\]
Conversely, every $\xi\in\omega_1[1/2]$ is realized by a nonopen
non-Archimedean Polish subgroup of $\mathbb Z^{\mathbb N}$.
\end{theorem}

In terms of Theorem~\ref{Corollary:complexity}: if $H$ has Solecki
rank $\rho$, a successor, then the class $\boldsymbol\Sigma^0_{1+\rho}$
is excluded when $\rho-1$ is a successor, and the class
$D(\boldsymbol\Pi^0_{1+\rho})$ is excluded when $\rho-1$ is zero or a
limit. The list in Theorem~\ref{Corollary:complexity-nonArchimedean}
is already implicit in the results of Hjorth, Kechris, and Louveau
cited above. Here it follows from the analysis of the last two Solecki
stages, through a non-Archimedean form of the collapse of potential
classes in \cite[Theorem~4.1]{hjorth_borel_1998}
(Proposition~\ref{Proposition:HKL2}), and the construction realizes the
same list in $\mathbb Z^{\mathbb N}$.

This restriction is a property of the coset relation. The equivalence
between potential membership of $E_H^G$ and membership of $H$ in the
corresponding additive, multiplicative, or first difference class,
displayed at the beginning of this introduction, holds for arbitrary
Borel subgroups (Proposition~\ref{prop:criterion}); in particular,
$E_H^G$ is potentially $\boldsymbol\Sigma^0_2$ exactly when
$H\in\boldsymbol\Sigma^0_2(G)$. Consequently, Borel maps inducing
injections of coset spaces preserve the relevant complexity bounds, and
coset bijections preserve the Solecki rank; preservation of individual
stages requires a compatibility hypothesis
(Section~\ref{sec:coset-maps}). Through this identification, the
restricted list of Theorem~\ref{Corollary:complexity-nonArchimedean}
holds whenever $E_H^G$ is classifiable by countable structures, that
is, Borel reducible to isomorphism of countable structures
(Theorem~\ref{thm:countable-structures}). Every orbit relation of a
Borel action of a non-Archimedean Polish group is classifiable in this
sense, so the hypothesis is weaker than non-Archimedeanity of $H$.

The two lists differ already at low levels. The bounded-variation
subspace of $\mathbb R^{\mathbb N}$ in Example~\ref{ex:bv} has
complexity class $D(\boldsymbol\Pi^0_2)$, and the Banach subspace $H$
of $\ell_2$ in Corollary~\ref{cor:additive-counterexample} has
complexity class $\boldsymbol\Sigma^0_3$. By
Theorem~\ref{Theorem:Polishable-complexity}, the coset relation of
the latter, which is the orbit relation of the translation action of
the Polish group $H$ on $\ell_2$, is potentially
$\boldsymbol\Sigma^0_3$ but not potentially $\boldsymbol\Pi^0_3$. By contrast, a potentially
$\boldsymbol\Sigma^0_3$ orbit relation of a continuous action of a
closed subgroup of $S_\infty$ on a Polish space is potentially
$\boldsymbol\Sigma^0_2$ \cite[Theorem~4.1(ii), p.~84]{hjorth_borel_1998}.
This answers the question of Hjorth, Kechris, and Louveau recalled
above; see Subsection~\ref{sec:intro-questions}. A different question
concerns a single group acting on many spaces.
Allison proved that, for every countable discrete group $\Delta$, the
classes of the list of \cite[Theorem~1]{hjorth_borel_1998} attained as
potential classes by orbit relations of continuous actions of
$\Delta^{\mathbb N}$ form an initial segment of that list, and that
for $\Delta=\mathbb Z$ every class of the list is attained
\cite[Theorem~5.10 and Corollary~5.11]{allisonNonArchimedeanTSIPolish2020}.

\subsection{Operator ranges}\label{sec:intro-linear}

For linear maps, we call a subspace a \emph{Fr\'echet subspace} or
\emph{Banach subspace} when it is the range of a continuous injective
linear map from a separable Fr\'echet or Banach space, respectively.
The linear upper restrictions follow from Saint Raymond's theory and
the theorem of Banakh, Dobrowolski, and Plichko recalled above
\cite[Theorems~1.1--1.3 and Introduction]{BDP}; here they are derived
from the subgroup classification. The corresponding range
classifications are as follows.

\begin{theorem}\label{Corollary:complexity-Frechetable}
The possible Borel indices of the ranges of nonsurjective continuous
linear maps between separable Fr\'echet spaces are exactly
$\omega_1^{\boldsymbol\Delta}$. More explicitly, for every such
$T:X\to Y$ one has
$\ell_{\boldsymbol\Delta}(T(X);Y)\in\omega_1^{\boldsymbol\Delta}$,
and every index in this set occurs. The corresponding exact
nonselfdual complexity classes are given by
Table~\ref{tab:fractional-classes}.
\end{theorem}

\begin{theorem}\label{Corollary:complexity-Banachable}
The possible Borel indices of the ranges of nonsurjective continuous
linear maps between separable Banach spaces are exactly
$\mathcal I_{\mathrm B}$. Thus they are all the indices in
$\omega_1^{\boldsymbol\Delta}$ except the positive limit ordinals.
For every $\xi\in\mathcal I_{\mathrm B}$ there are separable
Banach spaces $X,Y$ and a nonsurjective compact injective linear operator
$T:X\to Y$ with $\ell_{\boldsymbol\Delta}(T(X);Y)=\xi$.
The corresponding exact nonselfdual complexity classes are given by
Table~\ref{tab:fractional-classes}.
\end{theorem}

Every class in the Banach list occurs in every prescribed separable
infinite-dimensional Banach target through a compact injective
operator, with dense range in each nonclosed class
(Corollary~\ref{cor:banach-fixed-target}). Much of this realization
theory is classical. Banakh, Dobrowolski, and Plichko stated the
fixed-target assertions in
\cite[Theorems~2.6--2.8 and Proposition~2.10]{BDP}, apart from
$D(\boldsymbol\Pi^0_{1+\lambda+1})$ for $\lambda$ zero or a
countable limit, the case $\lambda=0$ being noted there as known;
the closed case is elementary. Ding and Gao \cite{ding_separable_2008}
and Malicki \cite{malicki_polishable_2008} later constructed
arbitrarily complex ranges in fixed targets. Malicki closed his
paper, submitted in January 2007, with two questions
\cite[p.~322]{malicki_polishable_2008}. The first asks whether a
single Banach space has injective images of unbounded Borel complexity
in a given Banach space; the fixed-domain construction of
\cite[Theorem~2.6]{BDP}, published in 2000, already answers it
affirmatively. The second asks whether, for a given countable ordinal,
some Banach space has all of its nontrivial Polish linear subspaces
above that Borel class. For infinite-dimensional spaces, with
nontrivial understood as nonclosed, no lower bound beyond the obvious
one holds: a nonclosed Polish subgroup is never
$\boldsymbol\Pi^0_2$ (Theorem~\ref{Theorem:Polishable-complexity}),
and every separable infinite-dimensional Banach space contains a dense
Banach subspace of class $\boldsymbol\Sigma^0_2$. This is the index $1/2$ of
Corollary~\ref{cor:banach-fixed-target}, and it already follows from
the case $\alpha=0$ of \cite[Theorem~2.7]{BDP}; see
\cite[Section~5.5]{LOI}.

Two constructions complete the Banach realizations needed here.
The first gives, for every countable successor $\beta\geq2$, a
compact injective operator with range in
$\boldsymbol\Sigma^0_\beta\setminus\boldsymbol\Pi^0_\beta$
and with separable-bidual domain
(Theorem~\ref{thm:additive-ranges}). The existence assertion is
\cite[Theorem~2.7]{BDP}. Its proof there passes through the extension
theorem \cite[Theorem~1.10]{BDP}, and for the classes
$\boldsymbol\Sigma^0_{1+\lambda+n+1}$ with $n\geq1$ some of its steps
require a different argument (Remark~\ref{add:rem:BDP}). We give a
complete proof, based on a quotient-norm construction and a
saturated-fiber calculation. The second construction assembles, in an
$\ell_1$-sum, operators whose Solecki ranks are cofinal in a positive
limit $\lambda$, and cuts the resulting range by one hyperplane; the
range so obtained has complexity class
$D(\boldsymbol\Pi^0_{\lambda+1})$
(Theorem~\ref{thm:limit-difference}). It answers a question of
Banakh, Dobrowolski, and Plichko, discussed in
Subsection~\ref{sec:intro-questions}.

\subsection{Questions from the literature}\label{sec:intro-questions}

The question of Hjorth, Kechris, and Louveau recalled at the beginning
of this introduction has an affirmative answer. In their article,
received in December 1996 and published in 1998, they remarked that
little was known about the potential Wadge classes of the orbit
relations of actions of general Polish groups, and that it was unknown,
for example, whether $\boldsymbol\Sigma^0_3$ can be the potential Wadge
class of such a relation \cite[p.~68]{hjorth_borel_1998}. Here
$\Gamma$ is the potential class of $E$ if $E$ is potentially $\Gamma$
but not potentially in the dual class \cite[p.~65]{hjorth_borel_1998}.
The Banach subspace $H$ of $\ell_2$ in
Corollary~\ref{cor:additive-counterexample} is the range of a compact
injective operator, and its coset relation is the orbit relation of the
continuous translation action of the Polish group $H$ on $\ell_2$. This
relation is potentially $\boldsymbol\Sigma^0_3$ and not potentially
$\boldsymbol\Pi^0_3$ (Theorem~\ref{Theorem:Polishable-complexity}), so
its potential class is $\boldsymbol\Sigma^0_3$. More generally,
Theorem~\ref{thm:additive-ranges} gives coset relations of potential
class $\boldsymbol\Sigma^0_{1+\lambda+n}$ for every $\lambda$ zero or
a countable limit and every $n\geq2$; none of these classes occurs for
actions of closed subgroups of $S_\infty$
\cite[Theorem~1]{hjorth_borel_1998}. For Polish subgroups and classes
from the third level on, Hjorth, Kechris, and Louveau had already
identified the potential class of the coset relation with the
complexity class of the subgroup
\cite[Proposition~5.1, p.~88]{hjorth_borel_1998}, so their question
asks, within this family, for a Polish subgroup of complexity class
$\boldsymbol\Sigma^0_3$. Operator ranges of this class are asserted to
exist in \cite[Theorem~2.7]{BDP}, published in 2000, but the proof
given there passes, in this case, through the step of
\cite[Theorem~1.10]{BDP} discussed in Remark~\ref{add:rem:BDP}(1). The
list of \cite[Theorem~1.1]{L} omits these classes, and the
constructions of Section~\ref{sec:additive-ranges} were announced in
\cite{Corr}. We found no earlier reference connecting such subgroups
with the question.

Banakh, Dobrowolski, and Plichko published in 2000 the following
question \cite[Question~2.11, p.~47]{BDP}. In the present indexing, it
asks whether, for a positive countable limit $\lambda$, an operator
image can belong to $D(\boldsymbol\Pi^0_{\lambda+1})$ without belonging
to either $\boldsymbol\Sigma^0_{\lambda+1}$ or
$\boldsymbol\Pi^0_{\lambda+1}$. They note that the zero case in their
indexing was already known, and their operator images are ranges of
continuous linear maps between Fr\'echet spaces. The answer is
affirmative at every positive countable limit. For a Fr\'echet domain
it follows from the group construction, specialized to linear spaces in
Theorem~\ref{Theorem:complexity-Frechetable}.
Theorem~\ref{thm:limit-difference} strengthens the answer to a compact
injective operator between separable Banach spaces, and
Corollary~\ref{cor:banach-fixed-target} places such an operator into
every separable infinite-dimensional Banach target. The domain of this
construction contains a closed copy of $\ell_1$, so its dual and bidual
are nonseparable.

The answer has the following history. Without reference to the
question, \cite{L} lists the class $D(\boldsymbol\Pi^0_{\lambda+1})$
among the classes of Fr\'echet ranges \cite[Theorem~1.3]{L} and of
Banach ranges \cite[Theorem~1.4]{L}. The Fr\'echet case comes there
from the sequence groups of \cite[Sections~8--9]{L}, whose classes are
recomputed here in Theorems~\ref{Theorem:complexityP}(2)
and~\ref{Theorem:complexity-Frechetable}, so the answer in the
generality of the question is already contained in
\cite[Theorems~1.3 and~9.4]{L}. The Banach construction of
\cite[Section~10]{L} does not establish the additive and difference
classes at limit stages \cite[Section~5]{Corr}. The Banach case is
settled by the construction of Theorem~\ref{thm:limit-difference},
which first appeared in \cite[Section~7]{Corr};
\cite[Theorem~B(4), Sections~4.1 and~4.5]{LOI} gives stronger nuclear
realizations based on it. In a problem list published in 2007,
Banakh, Cauty, and Zarichnyi wrote, citing \cite{BDP}, that each Borel
class they list for operator images between separable Banach spaces,
the difference classes included, contains an operator image
\cite[p.~615]{banakhOpenProblemsInfinitedimensional2007}; for the
classes $D(\boldsymbol\Pi^0_{\lambda+1})$ with $\lambda$ a positive
limit, \cite{BDP} gives no construction and asks Question~2.11
instead. We found no other treatment of the question
in the literature. The other questions of \cite{BDP}, on images of
bounded sets, topological universality, and homeomorphism types of
operator images, are not addressed here; see \cite[Section~4.1]{LOI}.

For a Banach space $X$ with Schauder basis $(x_n)$, its coefficient
space is
\[
 \operatorname{coef}(X,(x_n))
 =\{(a_n)\in\mathbb R^{\mathbb N}:\textstyle\sum_na_nx_n
      \text{ converges in }X\}.
\]
The coefficient map is a continuous injection of $X$ into
$\mathbb R^{\mathbb N}$, so this space is a Banach subspace of
$\mathbb R^{\mathbb N}$, and in particular a Polish subgroup. Ding asked in a 2015 preprint whether such a
coefficient space can belong to $\boldsymbol\Delta^0_3$ without being
the intersection of an $F_\sigma$ set and a $G_\delta$ set
\cite[Question~7.3(I), p.~29]{Ding2015}; the question reappears as
\cite[Question~6.3(I), p.~1479]{ding_equivalence_2017} in the published
version. It has a negative answer. For coefficient spaces the
implication
\[
 \boldsymbol\Delta^0_3\ \Longrightarrow\
 D(\boldsymbol\Pi^0_2)=D(\boldsymbol\Sigma^0_2)
\]
already follows from \cite[Theorem~1.3, with $\alpha=1$]{BDP}, which
predates the question; here the two difference classes agree because
both consist of the intersections of an $F_\sigma$ set and a $G_\delta$
set. The negative answer is stated in \cite[p.~2198]{L}, where it is
derived from \cite[Theorem~1.1]{L}, whose list of classes is
incomplete. Proposition~\ref{prop:pi3-local} gives a direct proof of the
implication, valid for arbitrary Polish subgroups; see also
\cite[Theorem~A(i), Section~5.4, Proposition~5.6]{LOI}. The
ambiguous-class hypothesis is essential:
Corollary~\ref{cor:additive-counterexample} gives a
$\boldsymbol\Sigma^0_3$ Banach subspace of $\ell_2$ that is not
$\boldsymbol\Pi^0_3$, and hence not the intersection of an $F_\sigma$
set and a $G_\delta$ set.

A related question of Farah and Solecki asks whether every uncountable
Polish group contains Polish subgroups of every Solecki rank
\cite[Question~6.1, p.~540]{farah_borel_2006}. They note that Hjorth's
theorem gives a positive answer for abelian groups. Hjorth, working
from a preprint of \cite{farah_borel_2006}, presents his theorem as a
partial answer to their question and conjectures that some uncountable
Polish group has only discrete abelian subgroups
\cite[Introduction and Conjecture~4.1]{hjorth_subgroups_2006}; see also
\cite[Subsection~8.6]{LOI}. For ranks alone, Hjorth's theorem already
gives Polish subgroups of every rank in $G^{\mathbb N}$ for every
nontrivial Polish group $G$, and what Theorem~\ref{Theorem:main} adds is
the exact complexity class at each rank (Remark~\ref{rem:hjorth-ranks}).
Since $G^{\mathbb N}$ contains an uncountable closed abelian subgroup,
these results do not bear on the question beyond the abelian case.

\subsection{Relation to the earlier article and its corrigendum}\label{sec:intro-relation}

Relative to \cite{L}, the list of classes in
Theorem~\ref{Corollary:complexity} contains the successor additive
classes $\boldsymbol\Sigma^0_{1+\sigma+1}$, $\sigma$ a successor,
which are missing from \cite[Theorem~1.1]{L}. These classes, the
criterion of Theorem~\ref{Theorem:Polishable-complexity}(2), the
$\boldsymbol\Delta^0_3$ implication of Proposition~\ref{prop:pi3-local},
the corrected Banach tree spaces of Appendix~\ref{app:tree-spaces}, and
the constructions of Sections~\ref{sec:additive-ranges}
and~\ref{sec:banach-realizations} were announced in \cite{Corr}. The
following are new relative to \cite{L} and \cite{Corr}: the two-stage
test in Theorem~\ref{Theorem:complexity}(2d), which decides the
alternative between the additive and the difference class left open
in \cite{Corr}; the finite-child approximation of
Subsection~\ref{sec:finite-child-closure}, which proves the stage
identities for the sequence groups; the proof of the complexity class
in Theorem~\ref{Theorem:complexityP}(3), stated in
\cite[Theorem~8.3(3)]{L}; the realizations in $G^{\mathbb N}$ for
every nontrivial Polish group $G$ (Theorem~\ref{Theorem:main}), where
\cite[Theorem~8.1]{L} assumes $G$ to be CLI; and the formulation of
the classification through the Borel index. Related operator-range
results, including the stronger realizations mentioned above, are
developed in \cite{LOI}.

Throughout, $\mathbb N=\{1,2,\ldots\}$ and
$\omega=\{0,1,\ldots\}$. All coset spaces use left cosets;
no normality is assumed unless explicitly stated.

\subsection{Organization}\label{sec:intro-organization}

Section~\ref{Section:polishable} records the basic facts about
Polish subgroups, and Section~\ref{Section:potential} identifies their
complexity with the potential complexity of their coset relations.
Sections~\ref{Section:solecki} and~\ref{Section:complexity-solecki}
develop the Solecki approximation, characterize its stages by their
Borel classes, and compare Borel classes in a stage with those in the
ambient group. Section~\ref{Section:complexity-polishable} uses them to
prove the classification theorem, Theorem~\ref{Theorem:complexity}.
Section~\ref{sec:coset-maps} treats Borel maps of coset spaces and
classification by countable structures, and
Section~\ref{Section:SR} identifies the Solecki rank with the Saint
Raymond rank. Section~\ref{Section:existence} realizes every index of
$\mathcal I_{\mathbb Z}$ by a Polish subgroup, and every index of
$\omega_1[1/2]$ by a non-Archimedean one, and
Section~\ref{section:Frechetable} gives the
corresponding Fr\'echet-space results. Section~\ref{Section:Banachable}
establishes the Banach upper restrictions.
Section~\ref{sec:additive-ranges} constructs the successor additive
operator ranges, and Section~\ref{sec:banach-realizations} constructs
the positive-limit successor difference ranges and completes the
Banach realization theorem in every fixed target; its final
Subsection~\ref{sec:final-proofs} contains the proofs of
Theorems~\ref{Corollary:complexity}--\ref{Corollary:complexity-Banachable}.
Appendix~\ref{app:tree-spaces} gives explicit Banach-space models of
the multiplicative classes, with all their Solecki stages computed.

\section{Polish subgroups\label{Section:polishable}}

A \emph{Polish space} is a second countable topological space whose topology
is induced by a complete metric. A \emph{Polish group }is a group in the
category of Polish spaces, namely a Polish space that is endowed with a
continuous group operation such that the function that maps each element to
its inverse is also continuous (in fact, the latter requirement holds
automatically; see the remark after \cite[Corollary 9.15]%
{kechris_classical_1995}). A subgroup $H$ of a Polish group $G$ is called a \emph{Polish
subgroup} if it is Borel and there exists a Polish group topology on $H$
whose open sets are Borel in $G$. If such a topology on $H$
exists, it is unique by \cite[Theorem 9.10]{kechris_classical_1995}. In the
following, we will regard $H$ as a Polish group with respect to its unique
Polish group topology, which is in general finer than the subspace topology
induced from $G$. Equivalently, $H$ is a Polish subgroup of $G$ if and
only if there exists a Polish group $\tilde{H}$ and a continuous group
homomorphism $\varphi :\tilde{H}\rightarrow G$ with image equal to $H$.
Noticing that one can assume without loss of generality that $\varphi $ is
an injection, the equivalence of the two definitions follows from \cite[%
Theorem 9.10]{kechris_classical_1995} and the fact that if $f:X\rightarrow Y$
is an injective Borel function between standard Borel spaces, then for
every Borel $A\subseteq X$, $f(A)$ is
a Borel subset of $Y$ and $f|_{A}$ is a Borel isomorphism between $A$ and $%
f(A)$ \cite[Corollary 15.2]{kechris_classical_1995}. If $G$ is a Polish group
and $H\ $is a Polish subgroup of $G$, then $G$ is a \emph{Polish }$H$%
\emph{-space} with respect to the action $h\cdot x=xh^{-1}$ of $H$ on $G$ \cite%
[Section 2.2]{becker_descriptive_1996}. Its orbit relation is $E_H^G$, defined by $x^{-1}y\in H$. Recall that a Polish group $G$ is 
\emph{non-Archimedean }if it admits a basis of neighborhoods of the identity
consisting of open subgroups; see \cite[Theorem 2.4.1]{gao_invariant_2009}
for equivalent characterizations.

\begin{remark*}[Terminology]
The subgroups called \emph{Polish subgroups} here are usually called
\emph{Polishable subgroups} in the descriptive set theory literature;
see \cite[Section~1.6]{becker_descriptive_1996} and
\cite[Section~1.2]{farah_borel_2006}. We always equip such a subgroup
with its unique Polish group topology compatible with the inherited
Borel structure. This topology can be strictly finer than the subspace
topology. Likewise, our terms \emph{Fr\'echet subspace} and
\emph{Banach subspace} refer to the compatible intrinsic complete
topology, which need not be the topology inherited from the containing
space. The inclusion map is continuous in each case.
\end{remark*}

\begin{lemma}
\label{Lemma:intersection--Polishable}Suppose that $G$ is a Polish group.
Let $\left( G_{n}\right) _{n\in \omega }$ be a sequence of Polish
subgroups of $G$. Then $G_{\omega }:=\bigcap_{n\in \omega }G_{n}$ is a
Polish subgroup of $G$. If $G_{n}$ is non-Archimedean for every $n\in
\omega $, then $G_{\omega }$ is non-Archimedean as well. Suppose, in
addition, that the family $(G_n)_{n\in\omega}$ is
\emph{downward-directed under inclusion}: for every $m,n\in\omega$
there exists $k\in\omega$ such that $G_k\subseteq G_m\cap G_n$.
If $A\subseteq G_{\omega }$ is dense in $G_{n}$ for every $n\in\omega$,
then $A$ is dense in $G_{\omega }$ with its canonical Polish topology.
\end{lemma}

\begin{proof}
We have that $G_{\omega }$ is the image of the Polish group%
\begin{equation*}
Z:=\left\{ \left( x_{n}\right) _{n\in \omega }\in \prod_{n\in \omega
}G_{n}:\forall n\in \omega \text{, }x_{n}=x_{n+1}\right\} \subseteq
\prod_{n\in \omega }G_{n}
\end{equation*}%
under the continuous injective group homomorphism $Z\rightarrow G$, $\left(
x_{n}\right) _{n\in \omega }\mapsto x_{0}$. This shows that $G_{\omega }$ is
Polish. If $G_{n}$ is non-Archimedean for every $n\in \omega $, then $Z$
is non-Archimedean, and hence $G_{\omega }$ is non-Archimedean as well.

For the density assertion, assume that the family is downward-directed.
The diagonal description gives a basis for the topology on $G_\omega$
consisting of finite intersections
\[
 B=G_\omega\cap U_1\cap\cdots\cap U_r,
\]
where $U_i$ is open in some $G_{n_i}$. Suppose that $B$ is nonempty,
and choose $x\in B$. Directedness gives an index $m$ such that
$G_m\subseteq\bigcap_{i=1}^r G_{n_i}$.
Each inclusion $G_m\to G_{n_i}$ is a Borel homomorphism between
the canonical Polish groups, and is therefore continuous.
Thus $G_m\cap U_1\cap\cdots\cap U_r$ is a nonempty open subset of
$G_m$, containing $x$, and meets $A$ by density in $G_m$.
Since $A\subseteq G_\omega$, this point belongs to $B$.
Every nonempty basic open set in $G_\omega$ therefore meets $A$.
\end{proof}

\section{Potential complexity\label{Section:potential}}

This section shows that the complexity of a Polish subgroup $H$ of a
Polish group $G$ coincides with the potential complexity of its coset
equivalence relation $E_H^G$
(Theorem~\ref{Theorem:Polishable-complexity}). This is the device
through which Borel reductions between coset relations transfer
complexity bounds in Section~\ref{sec:coset-maps}. The equivalence
holds for every Borel subgroup (Proposition~\ref{prop:criterion}),
with one argument at all levels: category quantifiers applied to
the sections of a single Borel relation. We also show that the duals
of the difference classes do not occur as exact classes of Polish
subgroups (Lemma~\ref{Lemma:difference}), give an example separating
$\boldsymbol\Sigma^0_2$ from $D(\boldsymbol\Pi^0_2)$, and recall the
collapse of Farah and Solecki above limit levels.

A \emph{complexity class }$\Gamma $ is a function $X\mapsto \Gamma \left(
X\right) $ that assigns to each Polish space $X$ a collection $\Gamma \left(
X\right) $ of Borel subsets, such that if $X,Y$ are Polish spaces and $%
f:X\rightarrow Y$ is a continuous function, then $f^{-1}(A)\in \Gamma \left(
X\right) $ for every $A\in \Gamma \left( Y\right) $. For a complexity class $%
\Gamma $, we let $D\left( \Gamma \right) $ be the complexity class
consisting of \emph{differences }between sets in $\Gamma $. For the
Borel classes considered here, $D(\boldsymbol\Pi^0_\alpha)
=D(\boldsymbol\Sigma^0_\alpha)$, denoted by
$D_2(\boldsymbol\Sigma^0_\alpha)$ in
\cite[Section~22.E]{kechris_classical_1995}. We let $\check{\Gamma}$ be the \emph{dual }complexity
class of $\Gamma $, such that $\check{\Gamma}\left( X\right) $ comprises the 
\emph{complements} of the elements of $\Gamma \left( X\right) $. We say that 
$\Gamma $ is self-dual if $\Gamma =\check{\Gamma}$. If $\Gamma $ is a
complexity class that is not self-dual, then we say that $\Gamma $ is the
complexity class of $A\subseteq X$ if $A\in \Gamma \left( X\right) $ and $%
A\notin \check{\Gamma}\left( X\right) $. We will be mainly interested in the
complexity classes $\boldsymbol{\Sigma }_{\alpha }^{0}$, $\boldsymbol{\Pi }%
_{\alpha }^{0}$, $\boldsymbol{\Delta }_{\alpha }^{0}$, and $D(\boldsymbol{%
\Pi }_{\alpha }^{0})$ for $1\leq\alpha<\omega _1$; see \cite[Section 11.B]%
{kechris_classical_1995}.

If $X$ is a standard Borel space and $E$ is an equivalence relation on $X$,
then $E$ has potential complexity $\Gamma $ if there exists a Polish
topology $\tau $ on $X$ that induces the Borel structure of $X$ such that $%
E\in \Gamma \left( \tau \times \tau \right) $ \cite[Definition~1, p.~152]{louveau_reducibility_1994}. This is equivalent to the assertion that there
exists a Borel equivalence relation $F$ on a Polish space $Y$ such that
$F\in\Gamma(Y\times Y)$ and $E\leq_B F$. Indeed, one implication
uses the identity map. Conversely, choose a Polish realization of
the standard Borel space $X$ and a Borel reduction $f:X\to Y$.
Refine the topology on $X$, preserving its Borel sets, to make the
inverse images under $f$ of a countable open basis of $Y$ open
\cite[Theorem~13.11]{kechris_classical_1995}. Then $f$ is continuous,
and $E=(f\times f)^{-1}(F)$ belongs to $\Gamma$.
Compare \cite[Lemma~12.5.4]{gao_invariant_2009} for orbit relations.

\begin{proposition}\label{prop:criterion}
Let $K$ be a Borel subgroup of a Polish group $P$, let
$1\leq\alpha<\omega_1$, and let
\[
 \Gamma\in\{\boldsymbol{\Sigma}^0_\alpha,\boldsymbol{\Pi}^0_\alpha,
            D(\boldsymbol{\Pi}^0_\alpha)\}.
\]
Then $E_K^P$ is potentially $\Gamma$ if and only if $K\in\Gamma(P)$.
\end{proposition}

\begin{proof}
If $K\in\Gamma(P)$, then $E_K^P\in\Gamma(P\times P)$, since it is the
inverse image of $K$ under the continuous map $(x,y)\mapsto x^{-1}y$.

Conversely, by the equivalence just recalled there are a Polish
space $Y$, an equivalence relation $F\in\Gamma(Y\times Y)$, and a
Borel reduction $f\colon P\to Y$ of $E_K^P$ to $F$. There is a
comeager $G_\delta$ set $C\subseteq P$ such that the restriction of
$f$ to $C$ is continuous \cite[Theorem~8.38]{kechris_classical_1995}.
Then $E_K^P\cap(C\times C)$ is the inverse image of $F$ under the
continuous map $C\times C\to Y\times Y$ induced by $f$, so it is
relatively $\Gamma$ in $C\times C$, and there is
$B\in\Gamma(P\times P)$ such that
\[
 B\cap(C\times C)=E_K^P\cap(C\times C).
\]
Put
\[
 A=\{(h,x)\in P\times P:(x,xh)\in B\}.
\]
Then $A\in\Gamma(P\times P)$, as the inverse image of $B$ under a
continuous map. Fix $h\in P$. The set $C\cap Ch^{-1}$ is comeager in
$P$, and for each of its points $x$ both $x$ and $xh$ lie in $C$;
hence $(h,x)\in A$ if and only if $h=x^{-1}(xh)\in K$. Thus the
section $A_h=\{x\in P:(h,x)\in A\}$ is comeager in $P$ if $h\in K$,
and meager in $P$ if $h\notin K$. Consequently
\begin{equation}\label{eq:section-dichotomy}
 K=\{h\in P:A_h\text{ is nonmeager}\}
  =\{h\in P:A_h\text{ is comeager}\}.
\end{equation}
The nonmeager-section quantifier preserves additive classes, and the
comeager-section quantifier preserves multiplicative classes
\cite[Exercises~22.22--22.23]{kechris_classical_1995}. By
\eqref{eq:section-dichotomy}, $K\in\boldsymbol\Sigma^0_\alpha(P)$ if
$\Gamma=\boldsymbol\Sigma^0_\alpha$, and
$K\in\boldsymbol\Pi^0_\alpha(P)$ if $\Gamma=\boldsymbol\Pi^0_\alpha$.

Suppose finally that $\Gamma=D(\boldsymbol\Pi^0_\alpha)$. Write
$A=Q\setminus R$ with $Q,R\in\boldsymbol\Pi^0_\alpha(P\times P)$,
and let $Q^*$ and $R^*$ be the sets of those $h\in P$ for which $Q_h$,
respectively $R_h$, is comeager in $P$. Both belong to
$\boldsymbol\Pi^0_\alpha(P)$. If $h\in K$, then $Q_h\supseteq A_h$ is
comeager, while $R_h$ is disjoint from the comeager set $A_h$ and is
therefore meager; thus $h\in Q^*\setminus R^*$. If $h\notin K$ and
$Q_h$ is comeager, then $R_h\supseteq Q_h\setminus A_h$ is comeager,
since $A_h$ is meager; thus $h\notin Q^*\setminus R^*$. Therefore
$K=Q^*\setminus R^*\in D(\boldsymbol\Pi^0_\alpha)(P)$.
\end{proof}

For $\alpha\geq3$ and Polish subgroups, Proposition~\ref{prop:criterion}
is due to Hjorth, Kechris, and Louveau
\cite[Propositions~5.1--5.2]{hjorth_borel_1998}.
For Polish subgroups, membership in the dual of a difference class
reduces to membership in one of the two classes involved.

\begin{lemma}
\label{Lemma:difference}Suppose that $G$ is a Polish group, and $H$ is a
Polish subgroup of $G$. Let $1\leq\alpha<\omega_1$. If $H$ is $%
\check{D}(\boldsymbol{\Pi }_{\alpha }^{0})$, then $H$ is either $\boldsymbol{%
\Pi }_{\alpha }^{0}$ or $\boldsymbol{\Sigma }_{\alpha }^{0}$.
\end{lemma}

\begin{proof}
Adopt the notation of the Vaught transform with respect to the left
translation action of $H$ on $G$. We have that $H=A\cup B$ where $A$ is $%
\boldsymbol{\Sigma }_{\alpha }^{0}$ and $B$ is $\boldsymbol{\Pi }_{\alpha
}^{0}$. If $x\in H$, then we have that either $x\in A^{\Delta H}$ or $x\in
B^{\ast H}$. Since $A^{\Delta H}$ and $B^{\ast H}$ are $H$-invariant, we
have that either $H\subseteq A^{\Delta H}$ or $H\subseteq B^{\ast H}$. Since 
$A^{\Delta H}$ and $B^{\ast H}$ are contained in $H$, we have that either $%
H=A^{\Delta H}$ or $H=B^{\ast H}$. This concludes the proof.
\end{proof}

For Polish subgroups, Proposition~\ref{prop:criterion} takes the
following form.

\begin{theorem}\label{Theorem:Polishable-complexity}
Let $G$ be a Polish group, let $H$ be a Polish subgroup of $G$, and
let $E_H^G$ be its coset equivalence relation.
\begin{enumerate}
\item $E_H^G$ is potentially $\boldsymbol{\Pi}^0_2$ if and only if $H$ is closed in $G$.
\item $E_H^G$ is potentially $\boldsymbol{\Sigma}^0_2$ if and only if
$H\in\boldsymbol{\Sigma}^0_2(G)$.
\item For every $1\leq\alpha<\omega_1$ and
\[
 \Gamma\in\{\boldsymbol{\Sigma}^0_\alpha,\boldsymbol{\Pi}^0_\alpha,D(\boldsymbol{\Pi}^0_\alpha)\},
\]
$E_H^G$ is potentially $\Gamma$ if and only if $H\in\Gamma(G)$.
\end{enumerate}
In fractional notation, for every $\xi\in\omega_1^{\mathrm{Ph}}$,
\[
 E_H^G\text{ is potentially }\Gamma_\xi
 \quad\Longleftrightarrow\quad
 H\in\Gamma_\xi(G)
 \quad\Longleftrightarrow\quad
 \ell_{\boldsymbol\Delta}(G/H)\leq\xi.
\]
\end{theorem}

Item~(2) is the case $\alpha=2$ of Item~(3); it is stated separately
because it is used repeatedly below.

\begin{proof}
Items~(2) and~(3) are instances of Proposition~\ref{prop:criterion}.
By Item~(3), $E_H^G$ is potentially $\boldsymbol\Pi^0_2$ exactly when
$H$ is $G_\delta$ in $G$. A $G_\delta$ subgroup is closed: $H$ is then
a dense $G_\delta$ subset of the Polish group $\overline H^{\,G}$, so
each coset of $H$ in $\overline H^{\,G}$ is comeager there, and two
cosets cannot be disjoint. This proves Item~(1). The fractional
formulation follows by applying these equivalences to each pointclass
of Table~\ref{tab:fractional-classes}. At a lower half-step the
pointclass is a union, so membership, and potential membership,
means membership in one of its two constituent pointclasses.
\end{proof}

\begin{example}\label{ex:bv}
There is a Polish subgroup $H\leq\R^{\N}$ with
$H\in D(\boldsymbol{\Pi}^0_2)(\R^{\N})$ for which $E_H^{\R^{\N}}$ is not
potentially $\boldsymbol{\Sigma}^0_2$.
\end{example}
\begin{proof}
Take
\[
 H=\operatorname{bv}_0(\R)
 =\left\{x\in\R^{\N}:x_n\longrightarrow0,\quad
          \sum_{n=1}^{\infty}|x_n-x_{n+1}|<\infty\right\}.
\]
The variation norm $\|x\|_{\mathrm{var}}=\sum_n|x_n-x_{n+1}|$ makes
$H$ a separable Banach space: the map
\[
 \ell_1\longrightarrow H,\qquad
 (a_n)_n\longmapsto\bigl(\sum_{k\geq n}a_k\bigr)_n
\]
is an isometric isomorphism. Its inclusion into $\R^{\N}$ is
continuous, so $H$ is Polish.

Let $\mathrm{BV}$ consist of the sequences with finite variation.
Its bounded variation sublevel sets are closed in the product
topology, so $\mathrm{BV}$ is $\boldsymbol{\Sigma}^0_2$. The set
\[
 L=\{x:\liminf_n|x_n|=0\}
   =\bigcap_{k,m\in\N}\bigcup_{n\geq m}\{x:|x_n|<1/k\}
\]
is $G_\delta$. Since every sequence of finite variation converges,
$H=\mathrm{BV}\cap L\in D(\boldsymbol{\Pi}^0_2)$.

Suppose that $H=\bigcup_jF_j$ with each $F_j$ closed in $\R^{\N}$.
Baire category in the Banach topology of $H$ gives $z\in H$,
$\varepsilon>0$, and $j$ such that
\[
 z+B_{\mathrm{var}}(0,\varepsilon)\subseteq F_j.
\]
The product-closed set $F_j-z\subseteq H$ then contains
$B_{\mathrm{var}}(0,\varepsilon)$. Choose $0<c<\varepsilon$.
For every $m\in\N$, the sequence $x^{(m)}$ defined by
\[
 x^{(m)}_n=
 \begin{cases}c,&n\leq m,\\0,&n>m\end{cases}
\]
has variation norm $c$. These sequences lie in $F_j-z$ and converge
coordinatewise to the constant sequence $(c,c,\ldots)$, which does
not belong to $H$. This contradicts closedness. Hence $H$ is not
$\boldsymbol{\Sigma}^0_2$, and
Theorem~\ref{Theorem:Polishable-complexity}(2) proves the assertion.
\end{proof}

We now recall some results concerning the possible complexity classes of
Polish subgroups. The following proposition is a reformulation of \cite[%
Corollary 3.4]{farah_borel_2006}.

\begin{proposition}[Farah--Solecki]
\label{Proposition:FS-reduce}Suppose that $G$ is a Polish group, and $H$ is
a Polish subgroup of $G$. If $\lambda <\omega _{1}$ is either zero or a
limit ordinal, and $H$ is $\boldsymbol{\Pi }_{1+\lambda +1}^{0}$ in $G$,
then $H$ is $\boldsymbol{\Pi }_{1+\lambda }^{0}$ in $G$.
\end{proposition}

\section{Solecki subgroups\label{Section:solecki}}

Suppose that $G$ is a Polish group, and $H$ is a Polish subgroup of $G$.
Then $H$ admits a canonical approximation by Polish subgroups indexed by
countable ordinals. As these were originally described by Solecki in \cite%
{solecki_polish_1999}, we call them \emph{Solecki subgroups }of $G$
associated with $H$. They have also been considered in \cite%
{solecki_coset_2009, farah_borel_2006}.

Solecki's \cite[Lemma~2.3]{solecki_polish_1999} implies that $G$ has a
smallest $\boldsymbol{\Pi }_{3}^{0}$ subgroup containing $H$, which we
denote by $s_{1}^{H}(G)$. This can be explicitly described as the subgroup
of $G$ defined by%
\begin{equation*}
\bigcap_{V}\bigcup_{z_{0},z_{1}\in H}(z_{0}\overline{V}^{G}\cap \overline{V}%
^{G}z_{1})
\end{equation*}%
where $V$ ranges among the open neighborhoods of the identity in $H$, and $%
\overline{V}^{G}$ is the closure of $V$ inside of $G$. It is proved in \cite[%
Lemma 2.3]{solecki_polish_1999} that $s_{1}^{H}(G)$ satisfies the following
properties---see also \cite[Lemma 4.5]{tsankov_compactifications_2006} and 
\cite[Section 3]{farah_borel_2006}:

\begin{itemize}
\item $s_{1}^{H}\left( G\right) $ is a Polish subgroup of $G$;

\item $H$\emph{\ }is dense in $s_{1}^{H}(G)$;

\item a neighborhood basis of $x\in s_{1}^{H}(G)$ consists of sets of the
form $\overline{Wx}^{G}\cap s_{1}^{H}(G)$ where $W$ is an open neighborhood
of the identity in $H$;

\item if $A\subseteq G$ is $\boldsymbol{\Pi }_{3}^{0}$ and contains $H$,
then $A\cap s_{1}^{H}(G)$ is comeager in the Polish group topology of $%
s_{1}^{H}(G)$.
\end{itemize}

\begin{lemma}
\label{Lemma:non-Archimedean-Solecki}Suppose that $G$ is a Polish group, and 
$H$ is a non-Archimedean Polish subgroup of $G$. Then a neighborhood
basis of the identity in $s_{1}^{H}(G)$ consists of the sets of the form $%
\overline{W}^{G}\cap s_{1}^{H}(G)$ where $W$ is an open subgroup of $H$. In
particular, $s_{1}^{H}(G)$ is non-Archimedean.
\end{lemma}

\begin{proof}
Since $H$ is non-Archimedean, the first assertion follows from the remarks
above. If $W$ is an open subgroup of $H$, then $\overline{W}^{G}\cap
s_{1}^{H}(G)$ is a subgroup of $s_{1}^{H}(G)$ with nonempty interior, whence
it is open, since a subgroup with nonempty interior is open. Therefore, the second assertion follows from the
first one.
\end{proof}

A similar argument as in the proof of \cite[Lemma 2.3]{solecki_polish_1999}
gives the following.

\begin{lemma}
\label{Lemma:characterize-solecki}Suppose that $G$ is a Polish group, and
that $N$ is a Polish subgroup of $G$. Let $H$ be a Polish subgroup
of $G$ such that:

\begin{enumerate}
\item $N\subseteq H$ and $N$ is dense in the Polish group topology of $H$;

\item for every open neighborhood $V$ of the identity in $N$, $\overline{V}%
^{G}\cap H$ contains an open neighborhood of the identity in $H$.
\end{enumerate}

If $A\subseteq G$ is $\boldsymbol{\Pi }_{3}^{0}$ and contains $N$, then $%
A\cap H$ is comeager in $H$. In particular, $H\subseteq s_{1}^{N}(G)$. If $H$
is furthermore $\boldsymbol{\Pi }_{3}^{0}$, then $H=s_{1}^{N}(G)$.
\end{lemma}

\begin{proof}
It suffices to consider the case when $A$ is $\boldsymbol{\Sigma }_{2}^{0}$.
In this case, there exist closed subsets $L_{k}$ of $G$ for $k\in \omega $
such that $A=\bigcup_{k\in \omega }L_{k}$. Suppose that $U\subseteq H$ is a
nonempty open set. Since $N$ is dense in $H$, $U\cap N$ is a nonempty open
subset of $N$. By the Baire Category Theorem, there exists $k_{0}\in \omega $
such that $L_{k_{0}}\cap U\cap N$ is not meager in $N$. Thus, there exist $%
x\in N$ and an open neighborhood $V$ of the identity in $N$ such that 
\begin{equation*}
Vx\subseteq L_{k_{0}}\cap U\cap N\text{.}
\end{equation*}%
Since $L_{k_{0}}$ is closed in $G$, we have that $\overline{Vx}^{G}\subseteq
L_{k_{0}}$. By (2), there is an open neighborhood $W$ of the identity in $H$
such that $Wx\subseteq \overline{Vx}^{G}\subseteq L_{k_{0}}$. This shows
that $x\in U$ is in the interior of $L_{k_{0}}\cap H\subseteq A\cap H$.
Since this holds for every nonempty open subset of $H$, we have that $A\cap
H $ contains a dense open subset of $H$, and hence it is comeager in $H$.
Apply this conclusion to a $\boldsymbol\Pi^0_3$ subgroup
containing $N$. Its intersection with $H$ is a comeager subgroup
of $H$, hence is all of $H$: Pettis's theorem makes it open,
and an open proper subgroup cannot be comeager. Taking that subgroup
to be $s_1^N(G)$ gives $H\subseteq s_1^N(G)$; if $H$ itself is
$\boldsymbol\Pi^0_3$, minimality gives equality.
\end{proof}

The sequence of\emph{\ Solecki subgroups} $s_{\alpha }^{H}(G)$ for $\alpha
<\omega _{1}$ of $G$ associated with $H$ is defined recursively by setting:

\begin{itemize}
\item $s_{0}^{H}(G)=\overline{H}^{G}$;

\item $s_{\alpha +1}^{H}(G)=s_{1}^{H}\left( s_{\alpha }^{H}(G)\right) $ for $%
\alpha <\omega _{1}$;

\item $s_{\lambda }^{H}(G)=\bigcap_{\beta <\lambda }s_{\beta }^{H}(G)$ for a
limit ordinal $\lambda <\omega _{1}$.
\end{itemize}

At a countable limit stage the preceding Solecki subgroups form a decreasing
family, so Lemma \ref{Lemma:intersection--Polishable} applies. One can
therefore prove by induction on $\alpha <\omega _{1}$ that $s_{\alpha }^{H}(G)$ is a
Polish subgroup of $G$, and $H$ is dense in $s_{\alpha }^{H}(G)$.
Furthermore, by Lemma \ref{Lemma:non-Archimedean-Solecki}, if $H$ is
non-Archimedean, then $s_{\alpha }^{H}(G)$ is non-Archimedean for every $%
1\leq \alpha <\omega _{1}$. It is proved in \cite[Theorem 2.1]%
{solecki_polish_1999} that there exists $\alpha <\omega _{1}$ such that $%
s_{\alpha }^{H}(G)=H$. We call the least countable ordinal $\alpha $ such
that $s_{\alpha }^{H}(G)=H$ the \emph{Solecki rank} of $H$ in $G$, and
denote it by
\[
 \operatorname{rank}_{\mathrm{Sol}}(H,G)
 =\min\{\alpha<\omega_1:s_\alpha^H(G)=H\}.
\]
Since $s_0^H(G)=\overline H^{\,G}$, the rank is zero exactly when $H$
is closed in $G$.

One can define the Polish groups $s_{\alpha }^{H}(G)$ solely in terms of $H$
endowed with the subspace topology inherited from $G$. Indeed, $s_{0}^{H}(G)$
can be seen as the completion of $H$ with respect to a suitable metric that
induces the subspace topology inherited from $G$; see \cite[Section 2.1]%
{solecki_polish_1999}. Using Lemma \ref{Lemma:characterize-solecki} one can
describe the Solecki subgroups of products, as follows.

\begin{lemma}
\label{Lemma:characterize-solecki-product}Suppose that, for every $n\in 
\omega$, $G_{n}$ is a Polish group, and $N_{n}$ is a Polish
subgroup. Define $G=\prod_{n\in \omega }G_{n}$ and $N=\prod_{n\in \omega
}N_{n}$. Then we have that 
\begin{equation*}
s_{\gamma }^{N}(G)=\prod_{n\in \omega }s_{\gamma }^{N_{n}}\left( G_{n}\right)
\end{equation*}%
for every $\gamma <\omega _{1}$.
\end{lemma}

\begin{proof}
It suffices to consider the case when $\gamma =1$. In this case, set%
\begin{equation*}
H_{n}:=s_{1}^{N_{n}}\left( G_{n}\right)
\end{equation*}%
for $n\in \omega $, and%
\begin{equation*}
H:=\prod_{n\in \omega }H_{n}\text{.}
\end{equation*}%
Then we have that $H$ is a $\boldsymbol{\Pi }_{3}^{0}$ Polish subgroup
of $G$, $N\subseteq H$, and $N$ is dense in $H$. If $V$ is an open
neighborhood of the identity in $N$, then there exist $n\in \omega $ and
open neighborhoods $V_{i}$ of the identity in $N_{i}$ for $i<n$ such that $V$
contains 
\begin{equation*}
\prod_{i<n}V_{i}=\{x\in N:\forall i<n\text{, }x_{i}\in V_{i}\}\text{.}
\end{equation*}%
For $i<n$, we have that $\overline{V}_{i}^{G_{i}}\cap H_{i}$ contains an
open neighborhood $W_{i}$ of the identity in $H_{i}$. Therefore, we have
that $\overline{V}^{G}\cap H$ contains%
\begin{equation*}
\prod_{i<n}W_{i}=\{x\in H:\forall i<n\text{, }x_{i}\in W_{i}\}\text{,}
\end{equation*}%
which is an open neighborhood of the identity in $H$. The conclusion thus
follows from Lemma \ref{Lemma:characterize-solecki}.
\end{proof}

\section{Complexity of Solecki subgroups\label{Section:complexity-solecki}}

This section bounds the Borel complexity of the Solecki subgroups and
compares Borel classes computed in a Solecki subgroup with those
computed in the ambient group. Its main result,
Theorem~\ref{Theorem:characterize-solecki}, characterizes
$s_\alpha^H(G)$ as the smallest $\boldsymbol\Pi^0_{1+\alpha+1}$
subgroup of $G$ containing $H$. Together with the relativization
results of Lemma~\ref{Lemma:relativize} and
Corollary~\ref{Corollary:relativize}, it reduces the classification
of Section~\ref{Section:complexity-polishable} to a question about
the last two Solecki stages.

Suppose that $G$ is a Polish group, and $H$ is a Polish subgroup of $G$.
For a complexity class $\Gamma $, we define $\Gamma (G)|_{H}$ to be the
collection of sets of the form $A\cap H$ for $A\in \Gamma (G)$. The
following results are essentially established in \cite{farah_borel_2006}. In
the statements and proofs, we adopt the Vaught transform notation in
reference to the left translation action $h\cdot x=hx$ of $H$ on
$G$; see \cite[Section~3.2]{gao_invariant_2009}. This is a separate
action from the right translation action used to represent the
left-coset relation $E_H^G$. The subgroup calculations here use
the indicated left action. Every subgroup containing $H$ is
invariant under both actions, so the convention does not change
the subgroup relativization or canonical approximation statements.

\begin{lemma}
\label{Lemma:unrelativize}Suppose that $G$ is a Polish group, $H$ is a
Polish subgroup of $G$, and $\alpha ,\beta <\omega _{1}$ are ordinals.
Then%
\begin{equation*}
\boldsymbol{\Sigma }_{1+\beta }^{0}\left( s_{\alpha }^{H}(G)\right)
\subseteq \boldsymbol{\Sigma }_{1+\alpha +\beta }^{0}(G)|_{s_{\alpha
}^{H}(G)}
\end{equation*}%
and%
\begin{equation*}
\boldsymbol{\Pi }_{1+\beta }^{0}\left( s_{\alpha }^{H}(G)\right) \subseteq 
\boldsymbol{\Pi }_{1+\alpha +\beta }^{0}(G)|_{s_{\alpha }^{H}(G)}\text{.}
\end{equation*}
\end{lemma}

\begin{proof}
Put $\rho=\operatorname{rank}_{\mathrm{Sol}}(H,G)$.
For $\alpha\leq\rho$, the proof of
\cite[Theorem~3.1]{farah_borel_2006} establishes
\begin{equation*}
\boldsymbol{\Sigma }_{1}^{0}\left( s_{\alpha }^{H}(G)\right)
\subseteq \boldsymbol{\Sigma }_{1+\alpha }^{0}(G)|_{s_{\alpha }^{H}(G)}.
\end{equation*}

For $\alpha>\rho$, we have $s_\alpha^H(G)=H=s_\rho^H(G)$,
so the same inclusion follows from its case $\rho$ and monotonicity
of the Borel hierarchy. By taking complements, we have that $\boldsymbol{\Pi }_{1}^{0}\left( s_{\alpha
}^{H}(G)\right) \subseteq \boldsymbol{\Pi }_{1+\alpha }^{0}(G)|_{s_{\alpha
}^{H}(G)}$. This is the case $\beta =0$ of the statement above. The rest
follows by induction on $\beta $.
\end{proof}

\begin{lemma}
\label{Lemma:relativize}Suppose that $G\ $is a Polish group, $H\ $is a
Polish subgroup of $G$, $\alpha ,\beta <\omega _{1}$, and $U\subseteq H$
is open in $H$. If $A\in \boldsymbol{\Sigma }_{1+\alpha +\beta }^{0}(G)$ and 
$B\in \boldsymbol{\Pi }_{1+\alpha +\beta }^{0}(G)$, then $A^{\Delta U}\cap
s_{\alpha }^{H}(G)\in \boldsymbol{\Sigma }_{1+\beta }^{0}\left( s_{\alpha
}^{H}(G)\right) $, and $B^{\ast U}\cap s_{\alpha }^{H}(G)\in \boldsymbol{\Pi 
}_{1+\beta }^{0}\left( s_{\alpha }^{H}(G)\right) $.
\end{lemma}

\begin{proof}
Put $\rho=\operatorname{rank}_{\mathrm{Sol}}(H,G)$ and first suppose
$\alpha\leq\rho$. When $\beta =0$, the assertion about $A$ is Claim~3.3 in the
proof of \cite[Theorem 3.1]{farah_borel_2006}. The assertion about $B$
follows by taking complements. This concludes the proof when $\beta =0$. The
case of an arbitrary $\beta $ is established by induction on $\beta $ using
the properties of the Vaught transform; see \cite[Propositions~3.2.5--3.2.6]%
{gao_invariant_2009}.
If $\alpha>\rho$, then $s_\alpha^H(G)=H$. For any Borel $A\subseteq G$,
choose open $O\subseteq H$ such that $(A\cap H)\mathbin\triangle O$
is meager in $H$. Left translation gives
$A^{\Delta U}\cap H=U^{-1}O$, which is open in $H$.
The dual transform of $B$ is closed, by complements. These stronger
conclusions prove all the asserted bounds beyond $\rho$.
\end{proof}

\begin{corollary}
\label{Corollary:relativize}Suppose that $G\ $is a Polish group, $H$ is a
Polish subgroup of $G$, and $\alpha ,\beta <\omega _{1}$. Let $L$ be a
Borel subgroup of $G$ containing $H$. If $L\in \boldsymbol{\Sigma }%
_{1+\alpha +\beta }^{0}(G)$, then $L\cap s_{\alpha }^{H}(G)\in \boldsymbol{%
\Sigma }_{1+\beta }^{0}\left( s_{\alpha }^{H}(G)\right) $. If $L\in 
\boldsymbol{\Pi }_{1+\alpha +\beta }^{0}(G)$, then $L\cap s_{\alpha
}^{H}(G)\in \boldsymbol{\Pi }_{1+\beta }^{0}\left( s_{\alpha }^{H}(G)\right) 
$. If $L\in D(\boldsymbol{\Pi }_{1+\alpha +\beta }^{0})\left( G\right) $,
then $L\cap s_{\alpha }^{H}(G)\in D(\boldsymbol{\Pi }_{1+\beta }^{0})\left(
s_{\alpha }^{H}(G)\right) $.
\end{corollary}

\begin{proof}
Observe that $L=L^{\ast H}=L^{\Delta H}$. Thus, the first two assertions
follow immediately from Lemma \ref{Lemma:relativize}. If $L=A\cap B$ where $%
A $ is $\boldsymbol{\Sigma }_{1+\alpha+\beta}^{0}$ in $G$ and $B$ is $%
\boldsymbol{\Pi }_{1+\alpha+\beta}^{0}$ in $G$, then we have that $L\cap
s_{\alpha }^{H}(G)=A^{\Delta H}\cap B^{\ast H}\cap s_{\alpha }^{H}(G)$ where 
$A^{\Delta H}\cap s_{\alpha }^{H}(G)\in \boldsymbol{\Sigma }_{1+\beta
}^{0}\left( s_{\alpha }^{H}(G)\right) $ and $B^{\ast H}\cap s_{\alpha
}^{H}(G)\in \boldsymbol{\Pi }_{1+\beta }^{0}\left( s_{\alpha }^{H}(G)\right) 
$ by Lemma \ref{Lemma:relativize}. Hence, $L\cap s_{\alpha }^{H}(G)\in D(%
\boldsymbol{\Pi }_{1+\beta }^{0})\left( s_{\alpha }^{H}(G)\right) $.
\end{proof}

Recall that, by Proposition \ref{Proposition:FS-reduce}, if $\alpha $ is
either zero or a countable limit ordinal, and $H$ is a $\boldsymbol{\Pi }%
_{1+\alpha +1}^{0}$ Polish subgroup of a Polish group, then $H$ is $%
\boldsymbol{\Pi }_{1+\alpha }^{0}$.

\begin{theorem}
\label{Theorem:characterize-solecki}Suppose that $G$ is a Polish group, $H$
is a Polish subgroup of $G$, and $\alpha <\omega _{1}$. Then
$s_\alpha^H(G)$ is the smallest $\Gamma_\alpha$ subgroup of $G$
containing $H$. Equivalently, it is the smallest
$\boldsymbol\Pi^0_{1+\alpha+1}$ subgroup containing $H$.
\end{theorem}

\begin{proof}
Farah and Solecki prove that $s_\alpha^H(G)$ belongs to
$\boldsymbol\Pi^0_{1+\alpha+1}(G)$
\cite[proof of Theorem~3.1]{farah_borel_2006} when
$\alpha\leq\rho=\operatorname{rank}_{\mathrm{Sol}}(H,G)$.
For $\alpha>\rho$, the stage equals $H$, and its bound at $\rho$
implies the displayed bound by monotonicity.
Let $L$ be any subgroup in this class containing $H$.
Corollary~\ref{Corollary:relativize}, with $\beta=1$, gives
\[
 L\cap s_\alpha^H(G)\in
 \boldsymbol\Pi^0_2\bigl(s_\alpha^H(G)\bigr).
\]
A $G_\delta$ subgroup of a Polish group is closed: it is comeager
in its closure, Pettis's theorem makes it open there, and an open
subgroup is closed there. Since $L\cap s_\alpha^H(G)$ contains
the dense subgroup $H$, it equals $s_\alpha^H(G)$.
Thus $s_\alpha^H(G)\subseteq L$, proving minimality in the
displayed multiplicative class. If $\alpha$ is zero or limit,
Proposition~\ref{Proposition:FS-reduce} also gives
$s_\alpha^H(G)\in\boldsymbol\Pi^0_{1+\alpha}(G)=\Gamma_\alpha(G)$;
for successors the assigned class is already
$\boldsymbol\Pi^0_{1+\alpha+1}$. This proves the equivalent
$\Gamma_\alpha$ minimality formulation.
\end{proof}

In particular, the stages are monotone in the generating subgroup:
if $N\subseteq N'$ are Polish subgroups of $G$, then
$s_\xi^N(G)\subseteq s_\xi^{N'}(G)$ for every $\xi<\omega_1$.
Indeed, the latter is a $\boldsymbol\Pi^0_{1+\xi+1}$ subgroup
containing $N$, so the preceding minimality theorem applies.
The inclusion of the canonically topologized stages is continuous
by automatic continuity.

\begin{lemma}
\label{Lemma:unrelativize-subgroup}Suppose that $G$ is a Polish group, $H$
is a Polish subgroup of $G$, and $\alpha <\omega _{1}$. Let $L$ be a
subgroup of $s_{\alpha }^{H}(G)$.

\begin{enumerate}
\item If $L\in \boldsymbol{\Pi }_{3}^{0}\left( s_{\alpha }^{H}(G)\right) $,
then $L\in \boldsymbol{\Pi }_{1+\alpha +2}^{0}(G)$.

\item If $L\in D(\boldsymbol{\Pi }_{2}^{0})(s_{\alpha }^{H}(G))$, then $L\in
D(\boldsymbol{\Pi }_{1+\alpha +1}^{0})(G)$.

\item If $L\in \boldsymbol{\Sigma }_{2}^{0}\left( s_{\alpha }^{H}(G)\right) $
and $\alpha $ is either zero or limit, then $L\in \boldsymbol{\Sigma }%
_{1+\alpha +1}^{0}(G)$.
\end{enumerate}
\end{lemma}

\begin{proof}
By Lemma \ref{Lemma:unrelativize} we have that%
\begin{equation*}
\boldsymbol{\Pi }_{3}^{0}\left( s_{\alpha }^{H}(G)\right) \subseteq 
\boldsymbol{\Pi }_{1+\alpha +2}^{0}(G)|_{s_{\alpha }^{H}(G)}
\end{equation*}%
and%
\begin{equation*}
D(\boldsymbol{\Pi }_{2}^{0})\left( s_{\alpha }^{H}(G)\right) \subseteq
D\left( \boldsymbol{\Pi }_{1+\alpha +1}^{0}\right) (G)|_{s_{\alpha }^{H}(G)}%
\text{.}
\end{equation*}%
Furthermore, $s_{\alpha }^{H}(G)\in \boldsymbol{\Pi }_{1+\alpha +1}^{0}(G)$
by Theorem \ref{Theorem:characterize-solecki}. Therefore, we have that%
\begin{equation*}
\boldsymbol{\Pi }_{3}^{0}\left( s_{\alpha }^{H}(G)\right) \subseteq 
\boldsymbol{\Pi }_{1+\alpha +2}^{0}(G)|_{s_{\alpha }^{H}(G)}\subseteq 
\boldsymbol{\Pi }_{1+\alpha +2}^{0}(G)
\end{equation*}%
and%
\begin{equation*}
D(\boldsymbol{\Pi }_{2}^{0})\left( s_{\alpha }^{H}(G)\right) \subseteq
D\left( \boldsymbol{\Pi }_{1+\alpha +1}^{0}\right) (G)|_{s_{\alpha
}^{H}(G)}\subseteq D\left( \boldsymbol{\Pi }_{1+\alpha +1}^{0}\right) (G)%
\text{.}
\end{equation*}%
This concludes the proof of (1) and (2).

When $\alpha $ is either zero or limit, we have by Theorem \ref%
{Theorem:characterize-solecki} and Proposition \ref{Proposition:FS-reduce}
that $s_{\alpha }^{H}(G)\in \boldsymbol{\Pi }_{1+\alpha }^{0}(G)\subseteq 
\boldsymbol{\Sigma }_{1+\alpha +1}^{0}(G)$. Therefore, in this case we have
that%
\begin{equation*}
\boldsymbol{\Sigma }_{2}^{0}\left( s_{\alpha }^{H}(G)\right) \subseteq 
\boldsymbol{\Sigma }_{1+\alpha +1}^{0}(G)|_{s_{\alpha }^{H}(G)}\subseteq 
\boldsymbol{\Sigma }_{1+\alpha +1}^{0}(G)\text{.}
\end{equation*}%
This concludes the proof of (3).
\end{proof}

\begin{lemma}
\label{Lemma:product}Suppose that, for every $k\in \omega $, $G_{k}$ is a
Polish group and $H_{k}$ is a subgroup of $G_{k}$. Define $G=\prod_{k\in
\omega }G_{k}$ and $H=\prod_{k\in \omega }H_{k}$. Assume that, for every $%
k\in \omega $, $H_{k}$ is $\boldsymbol{\Pi }_{\alpha }^{0}$ in $G_{k}$, and
for every $\beta <\alpha $ there exist infinitely many $k\in \omega $ such
that $H_{k}$ is not $\boldsymbol{\Pi }_{\beta }^{0}$ in $G_{k}$. Assume $\alpha\geq2$. Then $%
\boldsymbol{\Pi }_{\alpha }^{0}$ is the complexity class of $H$ in $G$.
\end{lemma}

\begin{proof}
We first explain the zero-dimensional parametrization used to apply
\cite[Theorem~22.10]{kechris_classical_1995}. Every nonempty Polish
space $X$ is the continuous open image of $\omega^\omega$.
One constructs such a map by choosing a tree of nonempty open sets
$U_s$ whose children cover $U_s$, whose child closures lie in $U_s$,
and whose diameters tend to zero in a compatible complete metric.
The branch map $\pi$ then satisfies $\pi([s])=U_s$.
For this map, a saturated inverse image preserves exact membership
in each Borel class. Here are the relevant category bounds.
For $B\subseteq\omega^\omega$ and open $U$, let $B^{\Delta U}$
be the set of $x$ for which $B\cap U\cap\pi^{-1}(x)$ is
nonmeager in $U\cap\pi^{-1}(x)$, and define $B^{*U}$ by
comeagerness. Empty fibers give false and true, respectively.
For open $B$, $B^{\Delta U}=\pi(B\cap U)$ is open.
Complements exchange the two transforms. If $B=\bigcup_n B_n$,
with each $B_n$ in a lower multiplicative class, then
\[
 B^{\Delta U}=\bigcup_{n,V\subseteq U}
     \bigl(\pi(V)\cap B_n^{*V}\bigr),
\]
where $V$ runs through a countable open basis. Each fiber is closed
in $\omega^\omega$ and hence Polish, so the category argument
applies. Intersecting an open set with a closed set is
$\boldsymbol\Sigma^0_2$; intersecting an open set with a
$\boldsymbol\Pi^0_\beta$ set stays
$\boldsymbol\Pi^0_\beta$ when $\beta\geq2$.
These bounds justify the induction at level $2$ as well as at
the higher successor and limit levels. Induction gives preservation of additive classes by
$\Delta$ and multiplicative classes by $*$.
For $B=\pi^{-1}(A)$ both transforms with $U=\omega^\omega$
equal $A$. Thus $A$ and $\pi^{-1}(A)$ belong to exactly the same
additive and multiplicative Borel classes.

Consequently, if a Borel subset $A$ of an arbitrary Polish space
is not $\boldsymbol\Pi^0_\beta$, apply the zero-dimensional
completeness theorem to $\pi^{-1}(A)$ and compose with $\pi$:
every $\boldsymbol\Sigma^0_\beta$ subset of a zero-dimensional
Polish space continuously reduces to $A$.

Clearly $H=\prod_kH_k$ belongs to $\boldsymbol\Pi^0_\alpha$.
Let $B\in\boldsymbol\Pi^0_\alpha(Z)$, where $Z$ is
zero-dimensional Polish, and write $B=\bigcap_j B_j$, with
$B_j\in\boldsymbol\Sigma^0_{\beta_j}(Z)$ and $\beta_j<\alpha$.
Choose distinct indices $k_j$ with
$H_{k_j}\notin\boldsymbol\Pi^0_{\beta_j}(G_{k_j})$.
The preceding paragraph supplies continuous reductions
$f_j:Z\to G_{k_j}$ of $B_j$ to $H_{k_j}$.
Use these as the coordinates of a map $Z\to G$ and use the
identity elements in all other coordinates. Its inverse image of
$H$ is exactly $B$. Hence $H$ is
$\boldsymbol\Pi^0_\alpha$-hard and is not
$\boldsymbol\Sigma^0_\alpha$.
\end{proof}

\begin{lemma}
\label{Lemma:product2}Suppose that, for every $k\in \omega $, $G_{k}$ is a
Polish group, $H_{k}$ is a Polish subgroup of $G_{k}$, and $0<\alpha
<\omega _{1}$. Define $G=\prod_{k\in \omega }G_{k}$ and $H=\prod_{k\in
\omega }H_{k}$. If $H_{k}$ has Solecki rank $\alpha $ in $G_{k}$ for every $%
k\in \omega $, then $\boldsymbol{\Pi }_{1+\alpha +1}^{0}$ is the complexity
class of $H$ in $G$ if $\alpha $ is a successor ordinal, and $\boldsymbol{%
\Pi }_{1+\alpha }^{0}$ is the complexity class of $H$ in $G$ if $\alpha $ is
a nonzero limit ordinal.
\end{lemma}

\begin{proof}
Define $\lambda =1+\alpha +1$ if $\alpha $ is a successor ordinal, and $%
\lambda =1+\alpha $ if $\alpha $ is a nonzero limit ordinal. By
Theorem \ref{Theorem:characterize-solecki} and, in the limit case,
Proposition~\ref{Proposition:FS-reduce}, for every $k\in \omega $, $H_{k}$
is $\boldsymbol{\Pi }_{\lambda }^{0}$ but not $\boldsymbol{\Pi }_{\beta
}^{0} $ for $\beta <\lambda $ in $G_{k}$. Therefore, by Lemma \ref%
{Lemma:product}, $\boldsymbol{\Pi }_{\lambda }^{0}$ is the complexity of $H$
in $G$.
\end{proof}
\begin{remark}\label{rem:rank-zero-products}
If all factors have rank zero, the product is closed. It is nonopen
exactly when some factor is nonopen or infinitely many factors are
proper. In that case its complexity class is $\boldsymbol\Pi^0_1$;
otherwise it is clopen.
\end{remark}

\section{Complexity of Polish subgroups}\label{Section:complexity-polishable}
The canonical approximation reduces the classification to a local
question at the last successor step. We first characterize, by
countable translates, when a Polish subgroup is $D(\boldsymbol\Pi^0_2)$
(Proposition~\ref{prop:local}), and show that a Polish subgroup that
is both $\boldsymbol\Sigma^0_3$ and $\boldsymbol\Pi^0_3$ is
$D(\boldsymbol\Pi^0_2)$ (Proposition~\ref{prop:pi3-local}). The
distinction between $\boldsymbol\Sigma^0_3$ and
$\boldsymbol\Delta^0_3$ matters: a $\boldsymbol\Sigma^0_3$ Polish
subgroup need not be $\boldsymbol\Pi^0_3$. After a collapse for
non-Archimedean subgroups (Proposition~\ref{Proposition:HKL2}), the
classification theorem (Theorem~\ref{Theorem:complexity}) reads off
the complexity class from the position of $H$ in the last two Solecki
stages.

\subsection{A local characterization by countable translates}

The following criterion isolates the exact local condition for the
condition $H\in D(\boldsymbol\Pi^0_2)(G)$. Neighborhoods and density below are specified in $H$,
and closures are taken in $G$. The auxiliary sets $K_n$ and $P_n$
are endowed with the subspace topologies inherited from $G$.

\begin{proposition}\label{prop:local}
Let $H$ be a Polish subgroup of a Polish group $G$. Fix a countable
identity neighborhood basis $(V_n)_{n\in\N}$ in $H$ and a countable
set $D$ dense in $H$. Set
\[
 K_n=\overline{V_n}^{\,G},\qquad
 P_n=\bigcup_{d\in D}dK_n.
\]
Each $P_n$ is a $\boldsymbol{\Sigma}^0_2$ set containing $H$. The following are
equivalent:
\begin{enumerate}
\item $H\in D(\boldsymbol{\Pi}^0_2)(G)$;
\item for some $n$, $H\cap K_n$ is $G_\delta$ in $K_n$;
\item for some $n$, $H$ is $G_\delta$ in $P_n$.
\end{enumerate}
\end{proposition}

\begin{proof}
For $h\in H$, density supplies
$d\in D\cap h(\operatorname{int}_H V_n)^{-1}$, so $h\in dV_n$.
Thus $H\subseteq P_n$, and $P_n$ is $\boldsymbol{\Sigma}^0_2$ because $K_n$ is closed.

Suppose (1) holds. Write
\[
 H=\bigl(\bigcup_jF_j\bigr)\cap U,
\]
where each $F_j$ is closed in $G$ and $U$ is $G_\delta$ in $G$.
The sets $F_j\cap H$ are closed in $H$ and cover $H$. Baire category gives $j$, $h_0\in H$, and $n$ with
$h_0V_n\subseteq F_j$. Taking closures in $G$ yields
\[
 K_n\subseteq h_0^{-1}F_j,\qquad
 H\cap K_n=K_n\cap h_0^{-1}U.
\]
This proves (2).

Suppose (2) holds and choose a $G_\delta$ subset $U$ of $G$ such that
$H\cap K_n=K_n\cap U$. Put
\begin{equation}\label{eq:local-factor}
 Q_n=\bigcap_{d\in D}\bigl((G\setminus dK_n)\cup dU\bigr).
\end{equation}
The set $Q_n$ is $G_\delta$. If $h\in H$ and $h\in dK_n$, then
$d^{-1}h\in H\cap K_n\subseteq U$, so $h\in Q_n$.
Conversely, if $x\in P_n\cap Q_n$, choose $d\in D$ with $x\in dK_n$.
Then $x\in dU$ by \eqref{eq:local-factor}, and consequently
$d^{-1}x\in K_n\cap U\subseteq H$. Hence
\begin{equation}\label{eq:local-decomposition}
 H=P_n\cap Q_n.
\end{equation}
This proves (3). Finally, (3) implies (1) since $P_n$ is $\boldsymbol{\Sigma}^0_2$.
\end{proof}

\subsection{The local criterion for \texorpdfstring{$\boldsymbol{\Pi}^0_3$}{Pi03} subgroups}

For a $\boldsymbol{\Pi}^0_3$ Polish subgroup, a regular neighborhood
basis converts an additive bound into a local $G_\delta$ condition.
This gives the ambiguous-class implication used in the classification.
The additive bound enters through the following regularization.

\begin{lemma}\label{lem:sigma3-gdelta}
Let $H$ be a Polish subgroup of a Polish group $G$ such that
$H\in\boldsymbol\Sigma^0_3(G)$. Then there is a $G_\delta$ subset $A$
of $G$ such that $A\subseteq H$ and $A$ contains an identity
neighborhood of $H$.
\end{lemma}

\begin{proof}
Write $H=\bigcup_jB_j$, with each $B_j$ a $G_\delta$ subset of $G$, and
fix a countable basis $\mathcal U$ of nonempty open subsets of $H$.
For $j\in\N$ and $U\in\mathcal U$ put
\[
 A_{j,U}=\{x\in G:\{h\in U:xh^{-1}\in B_j\}
                  \text{ is comeager in }U\}.
\]
These are the Vaught transforms $B_j^{\ast U}$ for the action
$h\cdot x=xh^{-1}$ of $H$ on $G$. If $x\in A_{j,U}$, then
$xh^{-1}\in B_j\subseteq H$ for some $h\in U$, so $x\in H$.
Conversely, if $x\in H$, the Borel sets
$\{h\in H:xh^{-1}\in B_j\}$, $j\in\N$, cover $H$; one of them is
nonmeager in $H$, and hence comeager in some $U\in\mathcal U$.
Therefore $H=\bigcup_{j,U}A_{j,U}$.

Each $A_{j,U}$ is $G_\delta$ in $G$. Indeed, write
$B_j=\bigcap_rO_{j,r}$ with $O_{j,r}$ open in $G$. Since the inclusion
of $H$ into $G$ is continuous, the sets
$\{h\in U:xh^{-1}\in O_{j,r}\}$ are open in $U$, and their
intersection over $r$ is comeager in $U$ exactly when each of them is
dense in $U$. Thus $x\in A_{j,U}$ if and only if, for every $r$ and
every $W\in\mathcal U$ with $W\subseteq U$, there is $h\in W$ with
$xh^{-1}\in O_{j,r}$; each of these conditions defines the open set
$\bigcup_{h\in W}O_{j,r}h$.

Each $A_{j,U}$ is also closed in $H$. For $x\in H$, the map
$h\mapsto xh^{-1}$ is a homeomorphism of $H$ onto itself, and $B_j$ is
$G_\delta$ in $H$. Hence $x\in A_{j,U}$ if and only if $B_j$ is
comeager in the open subset $xU^{-1}$ of $H$, that is, if and only if
$xU^{-1}\subseteq\overline{B_j}^{\,H}$. This condition defines a
closed subset of $H$.

By the Baire category theorem in $H$, some $A_{j,U}$ has nonempty
interior in $H$. If $y$ belongs to this interior, then
$A=y^{-1}A_{j,U}$ is a $G_\delta$ subset of $G$ contained in $H$ that
contains an identity neighborhood of $H$.
\end{proof}

\begin{proposition}\label{prop:pi3-local}\label{Lemma:Delta-complexity}\label{lem:delta3-subgroup}
Let $H\in\boldsymbol{\Pi}^0_3(G)$ be a Polish subgroup. There is a countable
identity neighborhood basis $(V_n)$ in $H$ satisfying
\begin{equation}\label{eq:regular-neighborhood}
 V_n=\overline{V_n}^{\,G}\cap H.
\end{equation}
For any such basis, the following conditions are equivalent:
\begin{enumerate}
\item $H\in\boldsymbol{\Sigma}^0_3(G)$;
\item $H\in D(\boldsymbol{\Pi}^0_2)(G)$;
\item $V_n\in\boldsymbol{\Pi}^0_2(G)$ for some $n$.
\end{enumerate}
In particular, if a Polish subgroup $H$ of a Polish group $G$
is both $\boldsymbol{\Sigma}^0_3$ and $\boldsymbol{\Pi}^0_3$ in $G$, then
$H\in D(\boldsymbol{\Pi}^0_2)(G)$: there exist a $\boldsymbol{\Sigma}^0_2$
subset $A$ of $G$ and a $\boldsymbol{\Pi}^0_2$ subset $B$ of $G$ such
that $H=A\cap B$.
\end{proposition}

The multiplicative hypothesis cannot be omitted: a
$\boldsymbol\Sigma^0_3$ Polish subgroup need not be
$D(\boldsymbol\Pi^0_2)$ (Corollary~\ref{cor:additive-counterexample}).
For ranges of injective operators between separable Fr\'echet spaces,
the $\boldsymbol\Delta^0_3$ implication is the case $\alpha=1$ of
\cite[Theorem~1.3]{BDP}. The present proposition applies without a
linear structure. Its regular neighborhood basis also follows from
\cite[Lemma~4.5 and the proof of Proposition~4.6]{tsankov_compactifications_2006}.

\begin{proof}
Since $H$ is a Polish $\boldsymbol{\Pi}^0_3$ subgroup,
Theorem~\ref{Theorem:characterize-solecki} gives $H=s_1^H(G)$. Solecki's description of the first
Solecki subgroup, recalled in Section~\ref{Section:solecki} for an arbitrary
Polish subgroup, then shows that, for a basis $(W_n)$ of open
identity neighborhoods in $H$, the sets
$V_n=\overline{W_n}^{\,G}\cap H$ form an identity neighborhood basis
in $H$. Since
$W_n\subseteq V_n\subseteq\overline{W_n}^{\,G}$, their closures in $G$
agree, proving \eqref{eq:regular-neighborhood}.

Proposition~\ref{prop:local} and \eqref{eq:regular-neighborhood}
give the equivalence of conditions~(2) and~(3) in the proposition: a
relative $G_\delta$ subset of the
closed set $\overline{V_n}^{\,G}$ is $G_\delta$ in $G$.
The implication (2)$\Longrightarrow$(1) is immediate.

Conversely, suppose that (1) holds. Lemma~\ref{lem:sigma3-gdelta}
supplies a $G_\delta$ subset $A$ of $G$ contained in $H$ and
containing an identity neighborhood of $H$. Choose $n$ with
$V_n\subseteq A$. Then
\[
 V_n=\overline{V_n}^{\,G}\cap H
     =\overline{V_n}^{\,G}\cap A\in\boldsymbol{\Pi}^0_2(G).
\]
This proves (3) and the claimed $\boldsymbol\Delta^0_3$ implication.
\end{proof}

\begin{lemma}\label{Lemma:limit-complexity}
If $\lambda>0$ is a countable limit ordinal and a Polish subgroup
$H\leq G$ belongs to $\boldsymbol\Sigma^0_\lambda(G)$, then
$H\in\boldsymbol\Pi^0_\mu(G)$ for some $\mu<\lambda$.
\end{lemma}
\begin{proof}
Suppose $\lambda>0$ is a countable limit ordinal and
$H\in\boldsymbol{\Sigma}^0_\lambda(G)$. Write
\[
 H=\bigcup_k F_k,\qquad F_k\in\boldsymbol{\Sigma}^0_{\mu_k}(G),\quad
 1\leq\mu_k<\lambda.
\]
This follows from the definition of the limit class, increasing each
lower rank by one if necessary. Baire category in $H$ gives an
$F_k$ that is nonmeager in $H$. For the action
$h\cdot x=xh^{-1}$, one has $H=F_k^{\Delta H}$: at $x\in H$
the section is $F_k^{-1}x$, and off $H$ it is empty. The Vaught-transform bound \cite[Theorem~3.2.7, p.~78]{gao_invariant_2009} gives
$H\in\boldsymbol{\Sigma}^0_{\mu_k}(G)\subseteq\boldsymbol{\Pi}^0_{\mu_k+1}(G)$, with
$\mu_k+1<\lambda$.

\end{proof}

\begin{lemma}
\label{Lemma:Delta-complexity2}Let $H$ be a Polish subgroup of a Polish
group $G$, let $\lambda<\omega_1$ be zero or a limit ordinal, and let
$1\leq n<\omega$. If $H\in\boldsymbol\Delta^0_{1+\lambda+n+1}(G)$, then
$H\in D(\boldsymbol\Pi^0_2)(s_{\lambda+n-1}^H(G))$, and consequently
$H\in D(\boldsymbol\Pi^0_{1+\lambda+n})(G)$.
\end{lemma}

For ranges of injective operators between separable Fr\'echet spaces,
Lemma~\ref{Lemma:Delta-complexity2} is contained in
\cite[Theorem~1.3]{BDP}.

\begin{proof}
Fix a countable open basis $\left\{ V_{j}:j\in \omega \right\} $ for $H$.
Since $H\in\boldsymbol\Pi^0_{1+\lambda+n+1}(G)$,
Theorem~\ref{Theorem:characterize-solecki} gives
\begin{equation*}
H=s_{\lambda +n}^{H}(G)=s_{1}^{H}\left( s_{\lambda +n-1}^{H}(G)\right) \in 
\boldsymbol{\Pi }_{3}^{0}\left( s_{\lambda +n-1}^{H}(G)\right) \text{.}
\end{equation*}%
Furthermore, we can write $H=\bigcup_{k\in \omega }F_{k}$ where $F_{k}$ is $%
\boldsymbol{\Pi }_{1+\lambda +n}^{0}$ in $G$ for every $k\in \omega $. Thus,
we have that $H=\bigcup_{k,\ell \in \omega }F_{k}^{\ast V_{\ell }}$, where $%
F_{k}^{\ast V_{\ell }}\in \boldsymbol{\Pi }_{2}^{0}\left( s_{\lambda
+n-1}^{H}(G)\right) $ by Lemma \ref{Lemma:relativize}. Thus, we have that $%
H\in \boldsymbol{\Sigma }_{3}^{0}\left( s_{\lambda +n-1}^{H}(G)\right) $.
Hence, by Proposition~\ref{prop:pi3-local} we have that $H\in D(%
\boldsymbol{\Pi }_{2}^{0})\left( s_{\lambda +n-1}^{H}(G)\right) $.
Furthermore, $D(\boldsymbol{\Pi }_{2}^{0})\left( s_{\lambda
+n-1}^{H}(G)\right) $ is contained in $D(\boldsymbol{\Pi }_{1+\lambda
+n}^{0})(G)$ by Lemma \ref{Lemma:unrelativize-subgroup}, concluding the
proof.
\end{proof}

\subsection{The non-Archimedean base case}
\label{sec:nonarch-base}
The following collapse uses the non-Archimedean topology on $H$ itself.

\begin{proposition}[Hjorth--Kechris--Louveau]
\label{Proposition:HKL2}Suppose that $G$ is a Polish group, and $H$ is a%
\emph{\ non-Archimedean} Polish subgroup of $G$. Suppose that $\lambda
<\omega _{1}$ is either zero or limit. If $H$ is $\boldsymbol{\Sigma }%
_{1+\lambda +2}^{0}$, then $H$ is $\boldsymbol{\Sigma }_{1+\lambda +1}^{0}$.
\end{proposition}
For orbit relations of Polish spaces of closed subgroups of
$S_\infty$, the corresponding collapse of potential classes is
\cite[Theorem~4.1(ii),(iii)]{hjorth_borel_1998}.

\begin{proof}
We first treat the case $\lambda=0$. Suppose that $G$ is a Polish
group and $H\leq G$ is a non-Archimedean Polish subgroup with
$H\in\boldsymbol{\Sigma}^0_3(G)$. Lemma~\ref{lem:sigma3-gdelta} gives
$F\in\boldsymbol{\Pi}^0_2(G)$, with $F\subseteq H$, containing an
identity neighborhood in $H$. Put $S=s_1^H(G)$.
Then $H=s_1^H(S)$ by Theorem~\ref{Theorem:characterize-solecki}. Lemma~\ref{Lemma:non-Archimedean-Solecki} supplies an open subgroup
$V\leq H$ such that
\[
 U=\overline V^{\,S}\cap H\subseteq F.
\]
The set $U$ is an open subgroup of $H$. By Lemma~\ref{Lemma:unrelativize} there is
$B\in\boldsymbol{\Pi}^0_2(G)$ such that $\overline V^{\,S}=B\cap S$.
Since $U\subseteq F\subseteq H$, it follows that $U=B\cap F$.
Thus $U$ is a $G_\delta$ subgroup of $G$ and is comeager in its
closure. Pettis's theorem makes it open in its closure; an open
subgroup is closed there, and density then forces $U$ to equal its
closure.
Its countably many cosets in $H$ show that $H\in\boldsymbol{\Sigma}^0_2(G)$.

Now let $\lambda$ be arbitrary, and put $T=s_\lambda^H(G)$. Since
$H\in\boldsymbol\Sigma^0_{1+\lambda+2}(G)$,
Corollary~\ref{Corollary:relativize} gives
$H\in\boldsymbol\Sigma^0_3(T)$. The subgroup $H$ is a non-Archimedean
Polish subgroup of the Polish group $T$, so the case already proved
gives $H\in\boldsymbol\Sigma^0_2(T)$, and
Lemma~\ref{Lemma:unrelativize-subgroup}(3) gives
$H\in\boldsymbol\Sigma^0_{1+\lambda+1}(G)$.
\end{proof}

\begin{theorem}\label{Theorem:complexity}
Let $H$ be a nonopen Polish subgroup of a Polish group $G$, with
Solecki rank $\alpha=\lambda+n$, where $\lambda$ is zero or a countable
limit ordinal and $n<\omega$.
\begin{enumerate}
\item If $n=0$, then $\ell_{\boldsymbol\Delta}(G/H)=\lambda$.
\item If $n\geq1$, put $\delta=\lambda+n-1$ and $S=s_\delta^H(G)$.
\begin{enumerate}[label=\textup{(\alph*)}]
\item If $H\notin D(\boldsymbol\Pi^0_2)(S)$, then
$\ell_{\boldsymbol\Delta}(G/H)=\alpha$.
\item If $H\in D(\boldsymbol\Pi^0_2)(S)\setminus\boldsymbol\Sigma^0_2(S)$,
then
\[
 \ell_{\boldsymbol\Delta}(G/H)=
 \begin{cases}
   \lambda+\tfrac12+\varepsilon,&n=1,\\
   \delta+\tfrac12,&n\geq2.
 \end{cases}
\]
\item If $n=1$ and $H\in\boldsymbol\Sigma^0_2(S)$, then
$\ell_{\boldsymbol\Delta}(G/H)=\lambda+1/2$.
\item If $n\geq2$ and $H\in\boldsymbol\Sigma^0_2(S)$, put
$T=s_{\lambda+n-2}^H(G)$. Then
\[
 \ell_{\boldsymbol\Delta}(G/H)=
 \begin{cases}
  \delta+\tfrac12-\varepsilon,
                &H\in\boldsymbol\Sigma^0_3(T),\\
  \delta+\tfrac12,&H\notin\boldsymbol\Sigma^0_3(T).
 \end{cases}
\]
\end{enumerate}
\end{enumerate}
The ordinary nonselfdual complexity classes are those in
Table~\ref{tab:fractional-classes} at the displayed indices.
For non-Archimedean $H$, case \textup{(2b)} and the lower-half-step
alternative in \textup{(2d)} are excluded.
\end{theorem}

\begin{proof}[Proof of Theorem~\ref{Theorem:complexity}]
We prove the corresponding exact pointclass statements in
Table~\ref{tab:fractional-classes}; at a successor rank put
$\beta=1+\lambda+n$. Their indices are precisely those in the theorem.
Item~(1) follows from Theorem~\ref{Theorem:characterize-solecki}, Proposition~\ref{Proposition:FS-reduce}, and
Lemma~\ref{Lemma:limit-complexity}. Theorem~\ref{Theorem:characterize-solecki} gives
$H=s_\lambda^H(G)\in\boldsymbol{\Pi}^0_{1+\lambda+1}(G)$, and
Proposition~\ref{Proposition:FS-reduce} lowers this to $\boldsymbol{\Pi}^0_{1+\lambda}(G)$.
For the lower bound, $H$ is nonopen when $\lambda=0$; and when
$\lambda$ is a limit, $H\in\boldsymbol{\Sigma}^0_\lambda(G)$ would give
$H\in\boldsymbol{\Pi}^0_\mu(G)$ for some $\mu<\lambda$ by Lemma~\ref{Lemma:limit-complexity},
hence Solecki rank below $\lambda$ by Theorem~\ref{Theorem:characterize-solecki}.
At a successor rank $H\in\boldsymbol{\Pi}^0_3(S)$; Lemma~\ref{Lemma:unrelativize-subgroup} supplies all the
displayed upper bounds.

For (2a), suppose that $H$ also belonged to
$\boldsymbol{\Sigma}^0_{\beta+1}(G)$. Since
$H=s_{\lambda+n}^H(G)\in\boldsymbol\Pi^0_{\beta+1}(G)$, this would give
$H\in\boldsymbol\Delta^0_{\beta+1}(G)$, and
Lemma~\ref{Lemma:Delta-complexity2} would give
$H\in D(\boldsymbol{\Pi}^0_2)(S)$, a contradiction.

For (2b), exact rank excludes $\boldsymbol{\Pi}^0_\beta(G)$ by Theorem~\ref{Theorem:characterize-solecki}.
If $H$ belonged to $\boldsymbol{\Sigma}^0_\beta(G)$, Corollary~\ref{Corollary:relativize}, applied at
stage $\lambda+n-1$, would imply $H\in\boldsymbol{\Sigma}^0_2(S)$, a contradiction.
If $H$ belonged to the dual of $D(\boldsymbol{\Pi}^0_\beta)$, Lemma~\ref{Lemma:difference} would
force membership in $\boldsymbol{\Pi}^0_\beta$ or $\boldsymbol{\Sigma}^0_\beta$.
Both are excluded, proving that the complexity class of $H$ in $G$
is $D(\boldsymbol{\Pi}^0_\beta)$.
This argument works for every $n\geq1$.

In (2c), the additive upper bound from Lemma~\ref{Lemma:unrelativize-subgroup}(3) and the rank
exclusion of the multiplicative class prove the claim.
For (2d), the difference upper bound and the same rank argument apply.
We claim that
\[
 H\in\boldsymbol{\Sigma}^0_\beta(G)
 \quad\Longleftrightarrow\quad
 H\in\boldsymbol{\Sigma}^0_3(T).
\]
The forward implication is Corollary~\ref{Corollary:relativize}, since
$1+(\lambda+n-2)+2=\beta$. Conversely,
Lemma~\ref{Lemma:unrelativize} writes $H=A\cap T$ for some
$A\in\boldsymbol{\Sigma}^0_\beta(G)$. By
Theorem~\ref{Theorem:characterize-solecki},
$T\in\boldsymbol{\Pi}^0_{\beta-1}(G)\subseteq
\boldsymbol{\Sigma}^0_\beta(G)$, so $H$ is
$\boldsymbol{\Sigma}^0_\beta$ in $G$.
The rank excludes $\boldsymbol{\Pi}^0_\beta(G)$.
If the additive condition fails, Lemma~\ref{Lemma:difference}
excludes the dual difference class, giving the asserted exact
difference class.

For non-Archimedean $H$, Proposition~\ref{Proposition:HKL2}, applied
inside $S$, excludes (2b), since $D(\boldsymbol{\Pi}^0_2)\subseteq\boldsymbol{\Sigma}^0_3$.
In (2d), if $H\in\boldsymbol{\Sigma}^0_\beta(G)$, put
$T=s_{\lambda+n-2}^H(G)$. Corollary~\ref{Corollary:relativize} gives $H\in\boldsymbol{\Sigma}^0_3(T)$.
The same base case gives
$H\in\boldsymbol{\Sigma}^0_2(T)\subseteq\boldsymbol{\Pi}^0_3(T)$, whence
\[
 H=s_1^H(T)=s_{\lambda+n-1}^H(G).
\]
This contradicts the exact rank $\lambda+n$.
\end{proof}

In particular, the Solecki rank is recovered by rounding the Borel
index up to an ordinal:
\[
 \operatorname{rank}_{\mathrm{Sol}}(H,G)
   =\min\{\alpha<\omega_1:
                 \ell_{\boldsymbol\Delta}(G/H)\leq\alpha\}.
\]
For open subgroups both sides are zero, so this identity holds
without the nonopenness assumption.

Suppose that $n\geq2$, and put $\beta=1+\lambda+n$. The hypothesis
$H\in\boldsymbol\Sigma^0_2(S)$ in case~(2d) is implied by the condition $H\in\boldsymbol\Sigma^0_3(T)$
that decides it. Indeed, the proof of~(2d) shows that
$H\in\boldsymbol\Sigma^0_3(T)$ implies
$H\in\boldsymbol\Sigma^0_\beta(G)$, and
Corollary~\ref{Corollary:relativize}, applied at stage $\lambda+n-1$,
then gives $H\in\boldsymbol\Sigma^0_2(S)$. Thus the
complexity class of $H$ is $\boldsymbol\Sigma^0_\beta$ if
$H\in\boldsymbol\Sigma^0_3(T)$; otherwise it is
$D(\boldsymbol\Pi^0_\beta)$ if $H\in D(\boldsymbol\Pi^0_2)(S)$, and
$\boldsymbol\Pi^0_{\beta+1}$ if not. Case~(2b) and the nonadditive
alternative in case~(2d) both give $\delta+1/2$: the Borel index does
not distinguish these different terminal relative classes.

Theorem~\ref{thm:additive-ranges} realizes the lower-half-step
alternative in case~(2d) at every permitted index. The relative
$\boldsymbol\Sigma^0_2$ condition alone therefore does not choose
between the two alternatives in that case. The general and
non-Archimedean classification lists of the Introduction follow from
the theorem; their realization assertions are proved in the later
construction sections.

\begin{remark}[Earlier formulations]\label{rem:earlier-formulations}
Several statements of the author's article \cite{L} are corrected
here; the corrections were announced in \cite{Corr}.
\begin{enumerate}
\item \cite[Theorem~3.3(2)]{L} asserts that $E_H^G$ is potentially
$\boldsymbol\Sigma^0_2$ if and only if $H\in D(\boldsymbol\Pi^0_2)(G)$.
The correct criterion is Theorem~\ref{Theorem:Polishable-complexity}(2),
and Example~\ref{ex:bv} shows that the two conditions differ.
\item \cite[Lemma~6.5]{L} asserts that every $\boldsymbol\Sigma^0_3$
Polish subgroup is $D(\boldsymbol\Pi^0_2)$.
Corollary~\ref{cor:additive-counterexample} shows that this fails;
Proposition~\ref{prop:pi3-local} proves it under the additional
hypothesis $H\in\boldsymbol\Pi^0_3(G)$.
\item The additive alternatives in Item~(2d) correct the successor
assertion of \cite[Theorem~6.1(2b)]{L} and its abelian restatement
in \cite[Theorem~4.18(2b)]{Heart}, and
Corollary~\ref{cor:additive-counterexample} exhibits the missing
additive case. The two-stage test in Item~(2d) decides the
alternative between the additive and the difference class, which
\cite{Corr} leaves open.
\item The restatements in \cite[Propositions~4.10 and~4.12]{Heart}
require, respectively, the additional additive classes in
Theorem~\ref{Corollary:complexity} and the criterion of
Theorem~\ref{Theorem:Polishable-complexity}(2).
\item Remark~\ref{rem:rank-zero-products} records the qualification
needed in the rank-zero clause of \cite[Lemma~5.7]{L}.
\item The complexity class of $S_{1+\lambda+n}$, $n\geq1$, in
Theorem~\ref{Theorem:complexityP}(3) is stated in
\cite[Theorem~8.3(3)]{L}; it is proved here.
\item The Banach tree spaces of \cite[Section~10]{L} require both
defining $c_0$ conditions at limit stages. Their corrected form is
given in Appendix~\ref{app:tree-spaces}, and
Example~\ref{ex:banach-limit} shows that convergence of the top row
alone does not characterize the space at the first limit stage.
\end{enumerate}
\end{remark}

\section{Coset maps and canonical approximations}\label{sec:coset-maps}
Potential complexity translates Borel reductions of coset relations
into bounds for the subgroups themselves. The first consequence
requires only an injection of coset spaces; preservation of an
individual canonical stage requires more structure.

\begin{corollary}[Coset injections and bijections]
\label{cor:coset-bijection}
Let $G,H$ be Polish groups and let $N\leq G$, $M\leq H$ be
Polish subgroups. Suppose that a Borel function $f:G\to H$
induces an injection
\[
 \overline f:G/N\longrightarrow H/M,
 \qquad gN\longmapsto f(g)M.
\]
Then, for every $1\leq\alpha<\omega_1$ and
\[
 \Gamma\in
 \{\boldsymbol{\Sigma}^0_\alpha,\boldsymbol{\Pi}^0_\alpha,
 D(\boldsymbol{\Pi}^0_\alpha)\},
\]
one has
\[
 M\in\Gamma(H)
 \quad\Longrightarrow\quad N\in\Gamma(G).
\]
In particular, if $\overline f$ is bijective, then
\[
 N\in\Gamma(G)
 \quad\Longleftrightarrow\quad M\in\Gamma(H).
\]
\end{corollary}

\begin{proof}
Put $E=E_N^G$ and $F=E_M^H$. The injection hypothesis means
exactly that $f$ is a Borel reduction from $E$ to $F$.
If $M\in\Gamma(H)$, then
$F\in\Gamma(H^2)$ by continuity of
$(u,v)\mapsto u^{-1}v$. Hence $E$ is potentially
$\Gamma$, and Theorem~\ref{Theorem:Polishable-complexity} gives
$N\in\Gamma(G)$.

Now suppose that $\overline f$ is bijective, so the range of
$f$ meets every $F$-class. The relation $E$ is idealistic,
since it is the orbit relation of the continuous action
$n\cdot x=xn^{-1}$ of the Polish group $N$ on $G$.
Equivalently, its cosets carry the comeager filters transported
from $N$. The classwise lifting theorem for idealistic
equivalence relations, in the form
\cite[Corollary~3.8]{BLP} of the Kechris--Macdonald theorem,
whose target relation $F=E_M^H$ is Borel since $M$ is Polish,
supplies a Borel lift $g:H\to G$ of $\overline f^{-1}$.
In particular,
\[
 f(g(h))\,F\,h\quad(h\in H),\qquad
 g(f(x))\,E\,x\quad(x\in G),
\]
and $g$ is a Borel reduction from $F$ to $E$.
Applying the first assertion to $f$ and $g$ proves the equivalence.
\end{proof}

The complexity invariance in Corollary~\ref{cor:coset-bijection}
does not imply that the given map preserves the individual
Solecki subgroups. Recall that $s_0^N(G)=\overline N^{\,G}$.

\begin{example}[Coset bijections need not preserve Solecki stages]
\label{ex:solecki-coset-swap}
There exist a locally compact Polish abelian group $G$, a Polish
subgroup $N$, and a Borel involution $f:G\to G$ fixing $N$ pointwise
and inducing a bijection of $N$-cosets, such that
\[
 f(s_0^N(G))\ne s_0^N(G).
\]
This failure persists after passing to coset spaces or saturating
the image by $N$.
\end{example}
\begin{proof}
Take
\[
 G=\R\times\mathbb Z/2\mathbb Z,\qquad
 N=\mathbb Q\times\{0\},\qquad S=s_0^N(G)=\R\times\{0\}.
\]
The countable subgroup $N$ is Polish in its discrete topology.
Put $a=(\sqrt2,0)$, $b=(0,1)$, $C=a+N$, and $D=b+N$, and define
\[
 f(x)=
 \begin{cases}
  x+b-a,&x\in C,\\
  x+a-b,&x\in D,\\
  x,&x\notin C\cup D.
 \end{cases}.
\]
The sets $C,D$ are disjoint Borel cosets, neither equal to $N$.
The map swaps them by translations and fixes every other coset
pointwise. Thus $f$ is a Borel involution, fixes $N$ pointwise,
and induces a coset bijection. Nevertheless
\[
 f(S)=(S\setminus C)\cup D\ne S.
\]
Since the displayed image is already $N$-saturated, saturation
does not repair the failure.

In fact $f(x-n)=f(x)-n$ for every $x\in G$ and $n\in N$.
For the action $n\cdot x=x-n$, its action cocycle
\[
 (n,x)\longmapsto f(x)-f(x-n)=n
\]
is continuous as a map $N\times G\to N$.
Even continuity of the action cocycle everywhere does not
force preservation of the Solecki stages.
\end{proof}

More generally, the same construction applies to any intermediate
Solecki subgroup $N\subsetneq s_\alpha^N(G)\subsetneq G$:
swap an $N$-coset inside it with one outside it, fixing $N$.
For nonabelian groups, the left translation $x\mapsto ba^{-1}x$
maps $aN$ onto $bN$. Thus the same construction works with left
cosets and is not confined to the zero stage.

Coset injections do not increase the Solecki rank
$\operatorname{rank}_{\mathrm{Sol}}$ defined in
Section~\ref{Section:solecki}.

\begin{corollary}[Solecki rank under coset injections]
\label{cor:solecki-rank-borel}
Let $N\leq G$ and $M\leq H$ be Polish subgroups of arbitrary
Polish groups. Suppose that a Borel function $f:G\to H$ induces
an injection $\overline f:G/N\to H/M$, $gN\mapsto f(g)M$. Then
\[
 \operatorname{rank}_{\mathrm{Sol}}(N,G)
 \leq \operatorname{rank}_{\mathrm{Sol}}(M,H).
\]
If $\overline f$ is bijective, then the Solecki ranks are equal.
\end{corollary}
\begin{proof}
The injection hypothesis means that $f$ is a Borel reduction from
$E_N^G$ to $E_M^H$.
Put $\rho=\operatorname{rank}_{\mathrm{Sol}}(M,H)$ and
$\xi=1+\rho+1$. By Theorem~\ref{Theorem:characterize-solecki},
$M=s_\rho^M(H)\in\boldsymbol{\Pi}^0_\xi(H)$, so $E_M^H\in\boldsymbol{\Pi}^0_\xi(H^2)$.
The reduction and Proposition~\ref{prop:criterion}
give $N\in\boldsymbol{\Pi}^0_\xi(G)$. Minimality in Theorem~\ref{Theorem:characterize-solecki}
therefore gives $s_\rho^N(G)=N$, proving the inequality.
For a coset bijection, the inverse admits a Borel lift by
Corollary~\ref{cor:coset-bijection}; apply the inequality in
both directions.
\end{proof}

\begin{corollary}[Nonnormal subgroups and compatible lifts]
\label{cor:solecki-nonnormal}
Let $N\leq G$ and $M\leq H$ be Polish subgroups of arbitrary
Polish groups. Suppose that a Borel function $f:G\to H$ satisfies
\begin{equation}\label{eq:coset-multiplicativity}
 f(N)\subseteq M,\qquad
 f(xy)M=f(x)f(y)M\quad(x,y\in G).
\end{equation}
The expression $f(x)f(y)M$ denotes a left coset in $H$;
no multiplication on $H/M$ is being defined.
Write $S_\alpha=s_\alpha^N(G)$ and $T_\alpha=s_\alpha^M(H)$.
Then:
\begin{enumerate}
\item The function $f$ induces a well-defined pointed map
\[
 F:G/N\longrightarrow H/M,\qquad
 xN\longmapsto f(x)M,
\]
and, for every $\alpha<\omega_1$,
\begin{equation}\label{eq:nonnormal-stage-inclusion}
 f(S_\alpha)\subseteq T_\alpha,\qquad
 F(S_\alpha/N)\subseteq T_\alpha/M.
\end{equation}
\item The map $F$ is injective if and only if $f^{-1}(M)=N$.
In that case its restrictions to the displayed stage coset
spaces are injective, and the rank inequality of
Corollary~\ref{cor:solecki-rank-borel} applies.
\item Suppose that $F$ is bijective and its inverse has a Borel
lift $g:H\to G$ satisfying
\[
 g(M)\subseteq N,\qquad
 g(hk)N=g(h)g(k)N\quad(h,k\in H).
\]
Then, for every $\alpha<\omega_1$,
\begin{gather}
 F(S_\alpha/N)=T_\alpha/M,
 \label{eq:nonnormal-stage-quotient}\\
 f^{-1}(T_\alpha)=S_\alpha,\qquad
 f(S_\alpha)M=T_\alpha.
 \label{eq:nonnormal-stage-equality}
\end{gather}
If $f:G\to H$ itself is surjective, then
$f(S_\alpha)=T_\alpha$.
\end{enumerate}
\end{corollary}
\begin{proof}
For $n\in N$ and $x\in G$, the assumptions give
$f(xn)M=f(x)f(n)M=f(x)M$. Also $f(1_G)\in M$.
Thus $F$ is well-defined and pointed.

Fix $\alpha$, put $\xi=1+\alpha+1$, and set
$K=f^{-1}(T_\alpha)$. This is a Borel subgroup containing $N$.
Indeed, \eqref{eq:coset-multiplicativity} gives
$f(xy)\in f(x)f(y)M$, so $K$ is closed under products because
$M\leq T_\alpha$. Applying the same identity to $xx^{-1}$ gives
$f(x)f(x^{-1})\in M$, since $f(1_G)\in M$. Thus $x\in K$
implies $x^{-1}\in K$.
Next apply the identity to $y=x(x^{-1}y)$ and multiply the
resulting coset equality on the left by $f(x)^{-1}$. This gives
\begin{equation}\label{eq:nonnormal-difference}
 f(x^{-1}y)M=f(x)^{-1}f(y)M.
\end{equation}
Consequently,
\[
 x^{-1}y\in K
 \quad\Longleftrightarrow\quad
 f(x)^{-1}f(y)\in T_\alpha.
\]
Thus $f$ reduces $E_K^G$ to $E_{T_\alpha}^H$.
By Theorem~\ref{Theorem:characterize-solecki}, $T_\alpha\in\boldsymbol{\Pi}^0_\xi(H)$.
Proposition~\ref{prop:criterion}, applied to the Borel subgroup $K$,
yields $K\in\boldsymbol{\Pi}^0_\xi(G)$. The minimality assertion of
Theorem~\ref{Theorem:characterize-solecki} concerns arbitrary
subgroups in this class, and gives $S_\alpha\subseteq K$.
This proves Item~(1).

The fiber of $F$ over the distinguished coset is represented
by $f^{-1}(M)$. Equation~\eqref{eq:nonnormal-difference} gives
\[
 x^{-1}y\in f^{-1}(M)
 \quad\Longleftrightarrow\quad f(x)^{-1}f(y)\in M.
\]
Thus $F$ is injective exactly when $f^{-1}(M)=N$, proving Item~(2).

For Item~(3), apply Item~(1) to both $f$ and $g$. The fact
that their induced coset maps are inverse gives
\eqref{eq:nonnormal-stage-quotient}. Since the stages contain
$N$ and $M$, respectively, this equality gives both identities
in \eqref{eq:nonnormal-stage-equality}. If $f$ itself is
surjective, the preimage equality gives literal equality of the
stage images.
\end{proof}

\begin{corollary}[Normal subgroups and quotient homomorphisms]
\label{cor:solecki-isomorphism}
Let $G,H$ be arbitrary Polish groups, and let $N\lhd G$,
$M\lhd H$ be normal Polish subgroups. Suppose that a Borel
function $f:G\to H$ induces a group homomorphism
\[
 \varphi:G/N\longrightarrow H/M.
\]
For $\alpha<\omega_1$, write
$S_\alpha=s_\alpha^N(G)$ and $T_\alpha=s_\alpha^M(H)$.
Then the following hold.
\begin{enumerate}
\item For every $\alpha<\omega_1$,
\begin{equation}\label{eq:solecki-hom-inclusion}
 f(S_\alpha)\subseteq T_\alpha,\qquad
 \varphi(S_\alpha/N)\subseteq T_\alpha/M.
\end{equation}
\item If $\varphi$ is injective, its restriction is an injective
homomorphism $S_\alpha/N\to T_\alpha/M$ for every $\alpha$, and
\begin{equation}\label{eq:solecki-rank-monotonicity}
 \operatorname{rank}_{\mathrm{Sol}}(N,G)
 \le \operatorname{rank}_{\mathrm{Sol}}(M,H).
\end{equation}
\item If $\varphi$ is an isomorphism, then
\begin{equation}\label{eq:solecki-quotient-isomorphism}
 \varphi(S_\alpha/N)=T_\alpha/M
\end{equation}
and, equivalently,
\begin{equation}\label{eq:solecki-lift-saturation}
 f^{-1}(T_\alpha)=S_\alpha,\qquad
 f(S_\alpha)M=T_\alpha.
\end{equation}
If in addition $f:G\to H$ is surjective, then
$f(S_\alpha)=T_\alpha$ for every $\alpha$.
\end{enumerate}
In Item~(2), injectivity of $\varphi$ alone does not imply Item~(3),
or equality of the stage preimages.
\end{corollary}

The abelian instance of Item~(1) is
\cite[Theorem~4.17, p.~23]{Heart}.
\begin{proof}
Conjugations preserve the distinguished normal subgroups and
therefore their canonical Solecki stages. Thus $S_\alpha$
and $T_\alpha$ are normal.

The quotient homomorphism assumption gives
$f(N)\subseteq M$ and
$f(xy)M=f(x)f(y)M$. Item~(1) follows from
Corollary~\ref{cor:solecki-nonnormal}(1).
For an injective quotient homomorphism, $f^{-1}(M)=N$;
the restrictions are injective homomorphisms, and
Corollary~\ref{cor:solecki-rank-borel} gives Item~(2).

If $\varphi$ is an isomorphism, its inverse has a Borel lift
by the classwise lifting argument in
Corollary~\ref{cor:coset-bijection}. Since the inverse is also
a quotient homomorphism, this lift satisfies the reverse
compatibility condition of
Corollary~\ref{cor:solecki-nonnormal}(3).
That result gives Item~(3), including the assertion for a
surjective lift.

For the final assertion, take $G=\mathbb R$, $N=\{0\}$,
$H=\mathbb R^2$, $M=\mathbb Q^2$, and $f(x)=(x,\sqrt2x)$.
Then $f^{-1}(M)=\{0\}$, so the quotient homomorphism is injective,
but $T_0=\overline M=H$ and $S_0=\{0\}$, whence
$f^{-1}(T_0)=\mathbb R\ne S_0$.
\end{proof}

The next theorem gives the same list under a hypothesis on the coset
relation. Recall that an equivalence relation on a standard Borel
space is \emph{classifiable by countable structures} if it Borel
reduces to isomorphism of countable structures in some countable
language, equivalently to isomorphism of binary relations on $\N$
\cite[p.~62]{Friedman}. To apply the results of Hjorth, Kechris, and
Louveau on orbit relations, we first replace the target by a Borel
orbit relation. For an equivalence relation $E$ on a standard Borel
space $X$, let $E^{+}$ be the equivalence relation on $X^{\N}$ given
by $(x_n)_n\mathrel{E^{+}}(y_n)_n$ if and only if every $x_n$ is
$E$-related to some $y_m$ and every $y_m$ is $E$-related to some
$x_n$. Friedman defines equivalence relations $T(\alpha)$, for
$\alpha<\omega_1$, by letting $T(0)$ be equality on $\N$,
$T(\alpha+1)=T(\alpha)^{+}$, and letting $T(\lambda)$ be the disjoint
sum of the relations $T(\alpha)$, $\alpha<\lambda$, at limits $\lambda$
\cite[p.~62]{Friedman}.

\begin{lemma}\label{lem:css-borel-target}
Let $E$ be a Borel equivalence relation on a standard Borel space $X$.
If $E$ is classifiable by countable structures, then there are a
non-Archimedean Polish group $Q$, a standard Borel space $Y$ with a
Borel action of $Q$ whose orbit equivalence relation $F$ is Borel,
and a Borel reduction of $E$ to $F$.
\end{lemma}

\begin{proof}
By \cite[Theorem~1.5]{Friedman}, $E$ is Borel reducible to $T(\alpha)$
for some $\alpha<\omega_1$. It therefore suffices to show, by
induction on $\alpha$, that $T(\alpha)$ is Borel reducible to the
orbit relation $F_\alpha$ of a Borel action of a non-Archimedean
Polish group $Q_\alpha$ on a standard Borel space $Y_\alpha$, with
$F_\alpha$ Borel. For $\alpha=0$, let $Q_0$ be trivial and $Y_0=\N$.

Suppose that $f\colon X_\alpha\to Y_\alpha$ reduces $T(\alpha)$, on
its domain $X_\alpha$, to $F_\alpha$. Let $Q_{\alpha+1}$ be the
semidirect product $Q_\alpha^{\N}\rtimes S_\infty$, in which
$S_\infty$ permutes the coordinates, acting on $Y_\alpha^{\N}$ by
\[
 \bigl((q_n)_n,\pi\bigr)\cdot(y_n)_n
   =\bigl(q_{\pi^{-1}(n)}\cdot y_{\pi^{-1}(n)}\bigr)_n.
\]
This is a Borel action of a non-Archimedean Polish group: for
$k\in\N$ and an open subgroup $V$ of $Q_\alpha$, the elements with
$q_n\in V$ and $\pi(n)=n$ for $n<k$ form an open subgroup, and these
subgroups form a basis at the identity. Two sequences $(y_n)_n$ and
$(z_n)_n$ lie in the same orbit if and only if there is a bijection
$\pi$ of $\N$ with $z_n\mathrel{F_\alpha}y_{\pi^{-1}(n)}$ for every $n$,
that is, if and only if every $F_\alpha$-class contains as many
terms of $(y_n)_n$ as of $(z_n)_n$. Equivalently, for every $n$ the
sets $\{m:y_m\mathrel{F_\alpha}y_n\}$ and
$\{m:z_m\mathrel{F_\alpha}y_n\}$ have the same cardinality, and
likewise with the roles of the two sequences interchanged. For each
$k$, the condition that such a set has at least $k$ elements is
Borel, so the orbit relation $F_{\alpha+1}$ is Borel. Fix a surjection
$p\colon\N\to\N$ whose fibers are all infinite, and put
$g((x_n)_n)=(f(x_{p(n)}))_n$. Each $F_\alpha$-class meets a sequence
in the range of $g$ in either no term or infinitely many terms.
Hence $g((x_n)_n)$ and $g((x'_n)_n)$ lie in the same orbit if and only
if the terms of the two sequences meet the same $F_\alpha$-classes,
that is, if and only if $(x_n)_n\mathrel{T(\alpha+1)}(x'_n)_n$.

At a limit $\lambda$, with reductions
$f_\alpha\colon X_\alpha\to Y_\alpha$ for $\alpha<\lambda$, let
$Y_\lambda$ be the disjoint sum of the spaces $Y_\alpha$, and let the
non-Archimedean Polish group $Q_\lambda=\prod_{\alpha<\lambda}Q_\alpha$
act on $Y_\alpha$ through its $\alpha$-th coordinate. The orbit
relation is the disjoint sum of the relations $F_\alpha$, which is
Borel, and the disjoint sum of the maps $f_\alpha$ reduces
$T(\lambda)$ to it.
\end{proof}

\begin{theorem}[Coset relations classifiable by countable structures]
\label{thm:countable-structures}
Let $N$ be a nonopen Polish subgroup of a Polish group $G$.
Suppose that $E_N^G$ is classifiable by countable structures. Then
\begin{equation}\label{eq:countable-structures-list}
 \ell_{\boldsymbol\Delta}(G/N)\in\omega_1[1/2].
\end{equation}
Conversely, every $\xi\in\omega_1[1/2]$ is realized by a nonopen
non-Archimedean Polish subgroup $N$ of $\mathbb Z^{\mathbb N}$
whose coset relation is classifiable by countable structures.

Moreover, for every Polish subgroup $K$ of $G$ with $N\leq K\leq G$,
\begin{equation}\label{eq:countable-structures-collapse}
 N\in\boldsymbol\Sigma^0_3(K)
 \quad\Longrightarrow\quad
 \ell_{\boldsymbol\Delta}(K/N)\leq\tfrac12,
\end{equation}
equivalently, $N\in\Gamma_{1/2}(K)=\boldsymbol\Sigma^0_2(K)$.
\end{theorem}

The first step of the proof shows that $E_N^K$ is Borel bireducible
with the orbit relation of a Polish $Q$-space, where $Q$ is
isomorphic to a closed subgroup of $S_\infty$. Applied with $K=G$,
this allows one to read off the list \eqref{eq:countable-structures-list}
also from the potential-complexity classification
\cite[Theorem~1]{hjorth_borel_1998}, together with
Theorem~\ref{Theorem:Polishable-complexity}. We derive it instead from
Theorem~\ref{Theorem:complexity} and the relative collapse
\eqref{eq:countable-structures-collapse}.

\begin{proof}
We first prove \eqref{eq:countable-structures-collapse}.
The inclusion $K\to G$ is continuous and reduces $E_N^K$ to
$E_N^G$, so $E_N^K$ is classifiable by countable structures.
Since $N$ is Polish, $E_N^K$ is a Borel orbit equivalence
relation, induced by the continuous action $n\cdot x=xn^{-1}$.

By Lemma~\ref{lem:css-borel-target}, there are a non-Archimedean
Polish group $Q$, a standard Borel $Q$-space $Y$ whose orbit relation
$F$ is Borel, and a Borel reduction $f\colon K\to Y$ of $E_N^K$ to $F$.
Orbit equivalence relations are idealistic: on each orbit $xN$,
the subsets $A$ for which
\[
 \{n\in N:xn^{-1}\in A\}
\]
is meager in $N$ form the corresponding orbit $\sigma$-ideal,
independent of the representative $x$. For every standard Borel
parameter space $Z$ and Borel $B\subseteq Z\times K$, the condition
that $\{n\in N:(z,xn^{-1})\in B\}$ is meager is Borel in $(z,x)$
by the category quantifier. This is the parameter-dependent definition
of idealistic in the authors' Corrections note to \cite{KM}.
Thus \cite[Lemma~3.8, p.~191, with the authors' Corrections]{KM}
applies. The $F$-saturation $Y_0$ of $f(K)$ is a $Q$-invariant Borel
subset of $Y$, and $E_N^K$ is classwise Borel isomorphic, and hence
Borel bireducible, with $F\mathbin{\upharpoonright}Y_0$. This
restriction is the orbit relation of the Borel action of $Q$ on $Y_0$.
By the Becker--Kechris theorem \cite[Theorem~4.4.6]{gao_invariant_2009},
$Y_0$ carries a Polish topology, with the same Borel sets, for which
this action is continuous, and $Q$ is isomorphic to a closed subgroup
of $S_\infty$ \cite[Theorem~2.4.1]{gao_invariant_2009}. The change of
topology affects neither $F\mathbin{\upharpoonright}Y_0$ nor its
potential complexity classes.

If $N\in\boldsymbol{\Sigma}^0_3(K)$, continuity of
$(x,y)\mapsto x^{-1}y$ gives
$E_N^K\in\boldsymbol{\Sigma}^0_3(K\times K)$.
Potential pointclass membership is invariant under Borel
bireducibility. Hence $F\mathbin{\upharpoonright}Y_0$ is
potentially $\boldsymbol{\Sigma}^0_3$.
By \cite[Theorem~4.1(ii), p.~84]{hjorth_borel_1998}, it is potentially
$\boldsymbol{\Sigma}^0_2$; see also \cite[Theorem~12.5.7]{gao_invariant_2009}.
The same holds for $E_N^K$, and Theorem~\ref{Theorem:Polishable-complexity}(2)
therefore gives $N\in\boldsymbol{\Sigma}^0_2(K)$.
This proves \eqref{eq:countable-structures-collapse}, in
particular when $K$ is any Solecki subgroup relative to $N$.

Write the Solecki rank of $N$ in $G$ as $\alpha=\lambda+n$.
Apply Theorem~\ref{Theorem:complexity} with $H=N$.
When $n\geq1$, put $S=s_{\lambda+n-1}^N(G)$.
Case~\textup{(2b)} is impossible: membership in
$D(\boldsymbol{\Pi}^0_2)(S)$ implies membership in
$\boldsymbol{\Sigma}^0_3(S)$, and
\eqref{eq:countable-structures-collapse} then gives
$N\in\boldsymbol{\Sigma}^0_2(S)$, contrary to that case.

For case~\textup{(2d)}, suppose that $n\geq2$ and that the
additive alternative holds. Put
$T=s_{\lambda+n-2}^N(G)$.
Since $N\in\boldsymbol{\Sigma}^0_{1+\lambda+n}(G)$,
Corollary~\ref{Corollary:relativize} gives
$N\in\boldsymbol{\Sigma}^0_3(T)$.
Equation~\eqref{eq:countable-structures-collapse} yields
$N\in\boldsymbol{\Sigma}^0_2(T)\subseteq\boldsymbol{\Pi}^0_3(T)$.
As $s_1^N(T)$ is the smallest Polish
$\boldsymbol{\Pi}^0_3$ subgroup of $T$ containing $N$, this implies
\[
 N=s_1^N(T)=s_{\lambda+n-1}^N(G),
\]
contradicting the rank $\lambda+n$.
The remaining alternatives in Theorem~\ref{Theorem:complexity}
are precisely those in \eqref{eq:countable-structures-list}.

Finally, the realization assertion of Theorem~\ref{Corollary:complexity-nonArchimedean}, with its
proof in Section~\ref{Section:existence}, supplies a non-Archimedean Polish
subgroup of $\mathbb Z^\N$ for each listed class.
Every Borel action of a non-Archimedean Polish group has an
orbit relation classifiable by countable structures: up to
isomorphism the group is a closed subgroup of $S_\infty$
\cite[Theorem~2.4.1]{gao_invariant_2009}, its action extends to a
Borel $S_\infty$-action in which every orbit contains exactly one of
the original orbits \cite[Theorem~2.3.5(a)]{becker_descriptive_1996},
and the logic action of a language with symbols of unbounded arity is
a universal Borel $S_\infty$-action
\cite[Theorem~2.7.3]{becker_descriptive_1996}. Applying this to the
translation action proves the converse.
\end{proof}

Relative to \eqref{eq:upper}, Theorem~\ref{thm:countable-structures}
excludes precisely the least indices
\[
 \sigma+\tfrac12-\varepsilon
 \quad(\sigma\text{ a countable successor}),\qquad
 \lambda+\tfrac12+\varepsilon
 \quad(\lambda\text{ zero or a countable limit}).
\]
The exclusion concerns exact complexity classes, not membership in
the corresponding cumulative pointclasses. The realization assertion
shows that no further index in \eqref{eq:countable-structures-list}
can be excluded under this hypothesis alone.

\section{The Saint Raymond rank\label{Section:SR}}

Saint Raymond introduced in \cite[D\'efinition~18]{saint-raymond_espaces_1976}
a notion of rank (therein called \emph{degree}) for Fr\'echet
subspaces of Fr\'{e}chet spaces; see Section \ref{section:Frechetable}. We
consider in this section the natural generalization of such a notion to
Polish subgroups of Polish groups. Recall that for a complexity class $%
\Gamma $, and a Polish subgroup $H$ of a Polish group $G$, we define $%
\Gamma (G)|_{H}$ to be the collection of sets of the form $A\cap H$ for $%
A\in \Gamma (G)$.

\begin{definition}
\label{Definition:degree}Suppose that $G$ is a Polish group, and $H\subseteq
G$ is a Polish subgroup. The \emph{Saint Raymond} \emph{rank }of $H$ is
the least countable ordinal $\alpha $ such that every open subset in the
Polish group topology of $H$ belongs to $\boldsymbol{\Sigma }_{1+\alpha
}^{0}(G)|_{H}$.
\end{definition}

Suppose that $X,Y$ are Polish spaces, and $\alpha $ is a countable ordinal.
As in \cite[p.~216]{saint-raymond_espaces_1976}, one can define $\mathcal{B%
}_{\alpha }\left( X,Y\right) $ to be the set of Borel functions that have
class at most $\alpha$, meaning that they are $\boldsymbol{\Sigma }_{1+\alpha }^{0}$-measurable. This is Saint Raymond's indexing convention;
for infinite $\alpha$ it differs from the usual Baire-class indexing
in \cite[Definition~24.2 and Theorem~24.3]{kechris_classical_1995}, which uses
$\boldsymbol\Sigma^0_{\alpha+1}$-measurability.
By definition, the Saint Raymond rank of $H$
is the least countable ordinal $\alpha $ such that the identity function of $%
H$ belongs to $\mathcal{B}_{\alpha }\left( X,Y\right) $, where $X$ is equal
to $H$ endowed with the subspace topology inherited from $G$, and $Y$ is
equal to $H$ endowed with its Polish group topology. Adapting an argument of
Tsankov from \cite{tsankov_compactifications_2006}, we now show that the
Saint Raymond rank and the Solecki rank of a Polish subgroup of $G$
coincide.

\begin{theorem}\label{thm:saint-raymond-rank}
Suppose that $G$ is a Polish group, and $H\subseteq G$ is a Polish
subgroup. Then the Saint Raymond rank is equal to the Solecki rank.
\end{theorem}

For Fr\'echet subspaces, the underlying linear rank criterion
is due to Saint Raymond
\cite[Corollaire~45 and Th\'eor\`eme~54]{saint-raymond_espaces_1976};
his canonical approximation is given in
\cite[Th\'eor\`eme~31 and Corollaire~53]{saint-raymond_espaces_1976}.
The theorem extends this identification to arbitrary Polish subgroups.

\begin{proof}
By Lemma \ref{Lemma:unrelativize} we have that the Saint Raymond rank is
less than or equal to the Solecki rank.\ We prove the converse inequality as
in the proof of \cite[Proposition 4.6]{tsankov_compactifications_2006}. If $%
H $ has Saint Raymond rank $\alpha $, then every open set in $H$ belongs to $%
\boldsymbol{\Sigma }_{1+\alpha }^{0}(G)|_{H}$. Suppose that $U$ is an open
neighborhood of the identity in $H$, and let $V$ be an open neighborhood of
the identity in $H$ such that $V^{-1}V\subseteq U$. Then there exists $A\in 
\boldsymbol{\Sigma }_{1+\alpha }^{0}(G)$ such that $A\cap H=V$. Thus, $1\in
A^{\Delta V}$, where $A^{\Delta V}\cap s_{\alpha }^{H}(G)$ is open in $%
s_{\alpha }^{H}(G)$ by\ Lemma \ref{Lemma:relativize}. Furthermore, we have
that $A^{\Delta V}\cap H\subseteq V^{-1}V\subseteq U$. This shows that $U$
contains a neighborhood of the identity with respect to the topology on $H$
induced by $s_{\alpha }^{H}(G)$. This shows that the Polish topology on $H$
is the subspace topology induced by $s_{\alpha }^{H}(G)$, whence $H$ is
closed in $s_{\alpha }^{H}(G)$. As $H$ is dense in $s_{\alpha }^{H}(G)$, we
have that $H=s_{\alpha }^{H}(G)$. This shows that $H$ has Solecki rank at
most $\alpha $.
\end{proof}

\section{Polish subgroups in each complexity class\label%
{Section:existence}}

This section realizes every index of $\mathcal I_{\mathbb Z}$ by a
Polish subgroup of $G^{\mathbb N}$, for every nontrivial Polish group
$G$ (Theorem~\ref{Theorem:main}), and every index of $\omega_1[1/2]$
by a non-Archimedean one (Theorems~\ref{thm:nonarch-cli-realization}
and~\ref{Theorem:complexitynAP}). The examples are tree-indexed
sequence groups. Over an abelian base, the groups $P_\gamma$ require
the rows of a tree-indexed array to converge up to level $\gamma$,
and three conditions on the top row---summability, bounded variation
with limit zero, and convergence---define the subgroups $S_\alpha$,
$D_\alpha$, and $P_\alpha$. Subsection~\ref{sec:finite-child-closure}
computes their Solecki stages by finite-child approximation, and
Theorem~\ref{Theorem:complexity} then gives their exact classes
(Theorem~\ref{Theorem:complexityP}).
Subsection~\ref{sec:relative-lower-bounds} removes the assumptions on
the base and gives the non-Archimedean realizations.

Recall that a
Polish group is \emph{CLI} if it admits a compatible complete left-invariant
metric or, equivalently, its left uniformity is complete \cite%
{malicki_polish_2011}.

\begin{theorem}\label{Theorem:main}
Let $\xi\in\mathcal I_{\mathbb Z}$ and let $G$ be a nontrivial
Polish group. There exists a nonopen abelian CLI Polish subgroup $H$ of
$G^{\mathbb N}$ with
$\ell_{\boldsymbol\Delta}(G^{\mathbb N}/H)=\xi$.
\end{theorem}

\begin{remark}
\label{Remark:nondiscrete}After replacing $G$ with $G^{\mathbb{N}}$, we can
assume that $G$ is not discrete. We will assume that $G$ is not discrete in
the rest of this section. The closed class $\boldsymbol\Pi^0_1$ in
Theorem~\ref{Theorem:main} is realized by the trivial subgroup of
this nondiscrete group.
\end{remark}

Recall that a pseudo-length function on a group $H$ is a function $%
L:H\rightarrow \lbrack 0,+\infty )$ such that, for $h,h^{\prime }\in H$:

\begin{itemize}
\item $L\left( 1_{H}\right) =0$;

\item $L\left( h^{-1}\right) =L\left( h\right) $;

\item $L\left( hh^{\prime }\right) \leq L\left( h\right) +L\left( h^{\prime
}\right) $.
\end{itemize}

A length function is a pseudo-length function $L$ such that $L\left(
h\right) =0\Rightarrow h=1_{H}$ for $h\in H$. A (pseudo-)length function $L$
gives rise to a left-invariant (pseudo-)metric $d$ defined by setting $%
d\left( h,h^{\prime }\right) =L\left( h^{-1}h^{\prime }\right) $, and every
left-invariant metric arises in this fashion.

For the construction, suppose that $G$ is a nondiscrete abelian Polish group, and let $L_{G}$ be a compatible complete invariant length function
on $G$ that induces the Polish topology on $G$. On $G^{\N}$ define the extended-valued functions
\[
 L_1(g)=\sum_{n\in\N}L_G(g_n),\qquad
 L_\infty(g)=\sup_{n\in\N}L_G(g_n),
 \qquad g=(g_n)_{n\in\N}.
\]
For $i\in\{1,\infty\}$, the function $d_i(g,h)=L_i(g^{-1}h)$ is an
extended left-invariant metric on $G^{\N}$. On the subgroup
$\{g:L_i(g)<\infty\}$, the function $L_i$ is a length and $d_i$ is its
associated metric.
We say that $\left( g_{n}\right) _{n\in \mathbb{N}}$ is $L_{G}$-summable%
\emph{\ }if $L_{1}(\left( g_{n}\right) _{n\in \mathbb{N}})<\infty $, and has
bounded (left) $L_{G}$-variation if $\sum_{n\in \mathbb{N}%
}L_{G}(g_{n+1}^{-1}g_{n})<\infty $. We let $\ell _{1}\left( G,L_{G}\right)
\subseteq G^{\mathbb{N}}$ to be the CLI Polish subgroup of $L_{G}$%
-summable sequences, $\mathrm{bv}_{0}\left( G,L_{G}\right) \subseteq G^{%
\mathbb{N}}$ to be the CLI Polish subgroup of vanishing sequences of
bounded $L_{G}$-variation, and $\mathrm{c}\left( G\right) \subseteq G^{%
\mathbb{N}}$ be the CLI Polish subgroup of convergent sequences.

Fix, for each limit ordinal $\lambda <\omega _{1}$, an increasing cofinal
sequence $\left( \lambda _{i}\right) _{i\in \mathbb{N}}$ in $\lambda $. If $%
\gamma =\delta +1$ is a successor ordinal, define $\gamma _{i}=\delta $ for
every $i\in \mathbb{N}$. Define by recursion on $\alpha <\omega _{1}$, $%
I_{0}^{0}=\left\{ \left( \varnothing ,\varnothing \right) \right\} $ where $%
\varnothing $ is the empty tuple, and $I_{0}^{\alpha }$ to be the set of
tuples $\left( n_{0},\ldots ,n_{d};\beta _{0},\ldots ,\beta _{d}\right) $
for $d\in \omega $, $n_{0},\ldots ,n_{d}\in \mathbb{N}$, $0=\beta
_{0}<\cdots <\beta _{d}=\alpha _{n_{d}}$, and $\left( n_{0},\ldots
,n_{d-1};\beta _{0},\ldots ,\beta _{d-1}\right) \in I_{0}^{\alpha _{n_{d}}}$.

Similarly for a fixed $\gamma <\omega _{1}$ we define $I_{\gamma }^{\alpha }$
by recursion on $\alpha \geq \gamma $, by setting $I_{\gamma }^{\gamma
}=\left\{ \left( \varnothing ,\varnothing \right) \right\} $, and $I_{\gamma
}^{\alpha }$ to be the set of tuples $\left( n_{0},\ldots ,n_{d};\beta
_{0},\ldots ,\beta _{d}\right) $ for $d\in \omega $, $n_{0},\ldots ,n_{d}\in 
\mathbb{N}$, $\gamma =\beta _{0}<\cdots <\beta _{d}=\alpha _{n_{d}}$, and $%
\left( n_{0},\ldots ,n_{d-1};\beta _{0},\ldots ,\beta _{d-1}\right) \in
I_{\gamma }^{\alpha _{n_{d}}}$. Notice that, by definition, if $\left(
n_{0},\ldots ,n_{d};\beta _{0},\ldots ,\beta _{d}\right) \in I_{\gamma
}^{\alpha }$ for some $\gamma \geq 1$, then $\left( m,n_{0},\ldots
,n_{d};\gamma _{m},\beta _{0},\ldots ,\beta _{d}\right) \in I_{\gamma
_{m}}^{\alpha }$ for every $m\in \mathbb{N}$.

Thus, for example we have that, for $1\leq k<\omega $, $I_{\gamma }^{\gamma
+k}$ is the set of tuples 
\begin{equation*}
\left( n_{0},\ldots ,n_{k-1};\gamma ,\gamma +1,\ldots ,\gamma +k-1\right)
\end{equation*}%
for $n_{0},\ldots ,n_{k-1}\in \mathbb{N}$, and $I_{\gamma }^{\gamma +\omega
} $ is the set of tuples%
\begin{equation*}
(n_{0},\ldots ,n_{d};\gamma ,\gamma +1,\ldots ,\left( \gamma +\omega \right)
_{n_{d}})
\end{equation*}%
for $d\in \omega $ such that $\left( \gamma +\omega \right) _{n_{d}}\geq
\gamma $, and $n_{0},\ldots ,n_{d}\in \mathbb{N}$.

We also define $I_{\alpha }^{\alpha }=\left\{ \left( \varnothing
;\varnothing \right) \right\} $. We denote by $I^{\alpha }$ the union of $%
I_{\gamma }^{\alpha }$ for $\gamma \leq \alpha $. If $\gamma \leq \alpha $,
we denote by $I_{\leq \gamma }^{\alpha }$ the union of $I_{\delta }^{\alpha
} $ for $\delta \leq \gamma $, and by $I_{<\gamma }^{\alpha }$ to be the
union of $I_{\delta }^{\alpha }$ for $\delta <\gamma $. For $\left( n;\beta
\right) $ and $\left( m;\tau \right) $ in $I^{\alpha }$ we define $\left(
n;\beta \right) \leq \left( m;\tau \right) $ if and only if there exist $%
\gamma _{0}\leq \gamma _{1}\leq \alpha $ such that $\left( n;\beta \right)
\in I_{\gamma _{0}}^{\alpha }$, $\left( m;\tau \right) \in I_{\gamma
_{1}}^{\alpha }$, $m$ is a tail of $n$, and $\tau $ is a tail of $\beta $,
i.e.\ we have that, for some $\ell \leq d<\omega $, $\left( n;\beta \right)
=\left( n_{0},\ldots ,n_{d};\beta _{0},\ldots ,\beta _{d}\right) $, $\left(
m;\tau \right) =\left( m_{0},\ldots ,m_{\ell };\tau _{0},\ldots ,\tau _{\ell
}\right) $, and for $0\leq i\leq \ell $, $m_{i}=n_{i+d-\ell }$ and $\tau
_{i}=\beta _{i+d-\ell }$ . We regard $I^{\alpha }$ as an ordered set with
respect to this order relation. Observe that $I_{\leq \gamma }^{\alpha }$
and $I_{<\gamma }^{\alpha }$ are downward-closed. For a subset $F$ of $%
I^{\alpha }$, we denote by $F_{\downarrow }$ its downward closure. Notice
also that if $F\subseteq I^{\alpha }$ is finite, and $\left( n;\beta \right)
\in I_{\gamma }^{\alpha }$ for some $\gamma \geq 1$ is such that $\left(
k,n;\gamma _{k},\beta \right) \in F_{\downarrow }$ for infinitely many $k\in 
\mathbb{N}$, then $\left( n;\beta \right) \in F_{\downarrow }$.

Fix a countable ordinal $\alpha $. We define, by recursion on $\gamma
\leq\alpha $, a decreasing sequence $\left( P_{\gamma }\right) _{\gamma \leq\alpha
}$ of CLI Polish subgroups of $G^{I_{0}^{\alpha }}$. Furthermore, for $%
x\in P_{\gamma }$, we define the values $x\left( n;\beta \right) \in G$ for $%
\left( n;\beta \right) \in I_{\gamma }^{\alpha }$. If $\gamma \geq 1$ and $%
\left( n;\beta \right) \in I_{\gamma }^{\alpha }$, then we let $\boldsymbol{x%
}\left( n;\beta \right) $ be the convergent sequence $\left( x\left(
k,n;\gamma _{k},\beta \right) \right) _{k\in \mathbb{N}}$ in $G$ with limit $%
x\left( n;\beta \right) $. If $\left( n;\beta \right) \in I_{0}^{\alpha }$,
then we let $\boldsymbol{x}\left( n;\beta \right) $ be the sequence
constantly equal to $x\left( n;\beta \right) $.

Define $P_{0}$ to be $G^{I_{0}^{\alpha }}$. This is an abelian CLI Polish group with
topology induced by the pseudo-length functions 
\begin{equation*}
L_{0}^{\left( n;\beta \right) }\left( x\right) =L_{G}(x\left( n;\beta
\right) )
\end{equation*}%
for $\left( n;\beta \right) \in I_{0}^{\alpha }$. Suppose that $1\leq \gamma
\leq \alpha $, and that $P_{\delta }$ has been defined for every $\delta
<\gamma $. Define $P_{<\gamma }=\bigcap_{\delta <\gamma }P_{\delta }$, and $%
P_{\gamma }\subseteq P_{<\gamma }$ to contain those $x\in P_{<\gamma }$ such
that, for every $\left( n;\beta \right) \in I_{\gamma }^{\alpha }$, the
sequence $\boldsymbol{x}\left( n;\beta \right) :=\left( x\left( i,n;\gamma
_{i},\beta \right) \right) _{i\in \mathbb{N}}$ is convergent. For $x\in
P_{\gamma }$ and $\left( n;\beta \right) \in I_{\gamma }^{\alpha }$, we
define $x\left( n;\beta \right) $ to be the limit of $\boldsymbol{x}\left(
n;\beta \right) $. Then we have that the Polish topology on $P_{\gamma }$ is
induced by the restriction to $P_{\gamma }$ of the continuous pseudo-length
functions on $P_{\delta }$ for $\delta <\gamma $, together with the
pseudo-length functions $L_{\gamma }^{\left( n;\beta \right) }\left(
x\right) =L_{\infty }(\boldsymbol{x}\left( n;\beta \right) )$ for $\left(
n;\beta \right) \in I_{\gamma }^{\alpha }$. Indeed, the map
$x\mapsto\bigl(x,(\boldsymbol x(n;\beta))_{(n;\beta)\in I_\gamma^\alpha}\bigr)$
identifies $P_\gamma$ with a closed subgroup of the product of
$P_{<\gamma}$ with countably many copies of the Polish group
$\mathrm c(G)$ of convergent sequences, with the supremum length
$L_\infty$; the displayed pseudo-lengths generate the product
topology. This concludes the recursive
definition of the CLI\ Polish subgroups $P_{\gamma }$ of $%
G^{I_{0}^{\alpha }}$ for $\gamma \leq \alpha $. Notice that in particular $%
P_{\alpha }$ contains the elements $x\in P_{<\alpha }$ such that the
sequence $\boldsymbol{x}\left( \alpha \right) :=\left( x\left( n,\alpha
_{n}\right) \right) _{n\in \mathbb{N}}$ belongs to $\mathrm{c}\left(
G\right) $. We also define $S_{\alpha }$ and $D_{\alpha }$ to be the
subgroups of $P_{\alpha }$ containing the elements $x\in P_{<\alpha }$ such
that the sequence $\boldsymbol{x}\left( \alpha \right) $ belongs to $\ell
_{1}\left( G,L_{G}\right) $ and $\mathrm{bv}_{0}\left( G,L_{G}\right) $,
respectively. Theorem \ref{Theorem:main} will be a consequence of the
following.

\begin{theorem}\label{Theorem:complexityP}
Let $G$ be a nondiscrete abelian Polish group with a compatible
complete invariant length. Fix $\alpha=1+\lambda+n$, where
$\lambda$ is zero or a countable limit and $n<\omega$.
\begin{enumerate}
\item At a positive limit $\lambda$, $P_{<\lambda}$ has Solecki rank
$\lambda$ and complexity class $\boldsymbol\Pi^0_\lambda$ in
$G^{I_0^\alpha}$.
\item If $n=0$, the groups $S_{1+\lambda}$, $D_{1+\lambda}$,
and $P_{1+\lambda}$ have rank $\lambda+1$ and complexity classes
$\boldsymbol\Sigma^0_{1+\lambda+1}$,
$D(\boldsymbol\Pi^0_{1+\lambda+1})$, and
$\boldsymbol\Pi^0_{1+\lambda+2}$, respectively.
\item If $n\geq1$, all three groups indexed by $1+\lambda+n$
have rank $\lambda+n+1$. The complexity classes of $D_{1+\lambda+n}$
and $P_{1+\lambda+n}$ are
$D(\boldsymbol\Pi^0_{1+\lambda+n+1})$ and
$\boldsymbol\Pi^0_{1+\lambda+n+2}$, respectively. The complexity
class of $S_{1+\lambda+n}$ is
$D(\boldsymbol\Pi^0_{1+\lambda+n+1})$.
\end{enumerate}
\end{theorem}

In Item~(3) the relative classes in the last Solecki stage do not
decide the class of $S_{1+\lambda+n}$; the proof uses the two-stage
test of Theorem~\ref{Theorem:complexity}(2d).

We prove the closure and stage identities by finite-subtree truncations.
For this argument only, write the abelian group operation additively
and write $\ell=L_G$. Node notation abbreviates the tuple notation
already defined; the resulting groups and topologies are the same.

\subsection{Finite-child approximation and Solecki stages}
\label{sec:finite-child-closure}
\paragraph{The construction and finite truncations.}
Fix $0<\beta<\omega_1$. We recall the tree form of the index sets $I_\delta^\beta$ defined above. There is a root $r$ of level $\beta$. A node $u$ of positive level $\delta$ has children $u_i$, $i\in\N$, of levels $\delta_i<\delta$. For successor $\delta$, all $\delta_i$ equal its predecessor; for limit $\delta$, the sequence $(\delta_i)$ is increasing and cofinal in $\delta$. Level-zero nodes are leaves. This is the same tree as the finite-tuple construction above, with descent corresponding to extension of tuples. Write $I_{<\gamma}^\beta$ and $I_{\leq\gamma}^\beta$ for unions of levels, and $F_\downarrow$ for the descendants of a set $F$, including $F$ itself.

Set $A=G^{I_0^\beta}=P_0$. Recursively put $P_{<\gamma}=\bigcap_{\delta<\gamma}P_\delta$, and let $P_\gamma$ consist of the $x\in P_{<\gamma}$ for which $(x(u_i))_i$ converges at every node $u$ of level $\gamma$. Define $x(u)=\lim_i x(u_i)$. The Polish topologies from the tuple construction above are generated by the pseudolengths
\[
 q_u(x)=
 \begin{cases}
 \ell(x(u)),&u\text{ has level }0,\\
 \sup_i\ell(x(u_i)),&u\text{ has positive level},
 \end{cases}
\]
using levels at most $\gamma$ for $P_\gamma$ and strictly below $\gamma$ for $P_{<\gamma}$. Node evaluations are continuous, and $\ell(x(u))\leq q_u(x)$ whenever defined.

Let $v_i=r_i$ be the children of the root, of levels $\beta_i$. The subgroup $S_\beta$ consists of those $x\in P_\beta$ for which
\[
 t(x)=\sum_i\ell(x(v_i))<\infty.
\]
Its Polish topology is generated by $t$ and the $q_u$. These are the groups of the tuple construction above; only the finite approximations used in their analysis change below.

\begin{lemma}\label{Lemma:construct-element}
Let $\sigma<\beta$, $x\in P_\sigma$, and let $F\subseteq I_{\leq\sigma}^\beta$ be finite. Define $z\in A$ by retaining $x$ on the leaves of $F_\downarrow$ and setting all other leaf values to zero. Then $z\in S_\beta$, and at every node $u$,
\[
 z(u)=\begin{cases}x(u),&u\in F_\downarrow,\\0,&u\notin F_\downarrow.\end{cases}
\]
In particular, $S_\beta$ is dense in $P_\sigma$.
\end{lemma}
\begin{proof}
If $u\notin F_\downarrow$, only finitely many children of $u$ lie in $F_\downarrow$. Indeed, no member of $F$ is an ancestor of $u$, and each child in $F_\downarrow$ requires a member of $F$ in its own branch. This is the finite-branch observation established above.

Induct on the level of $u$. The assertion holds at leaves by definition. If $u\in F_\downarrow$ has positive level, its children also lie there, their values are unchanged, and they converge to $x(u)$; all these levels are at most $\sigma$. If $u\notin F_\downarrow$, its sequence of child values is eventually zero by the preceding observation, and hence converges to zero. This proves the assertion through the root, and proves $z\in P_\beta$. The root is outside $F_\downarrow$, since $\sigma<\beta$. Its child sequence is therefore eventually zero, so $t(z)<\infty$.

For density, choose $F$ to contain the finitely many nodes whose pseudolengths control a prescribed neighborhood of $x$ in $P_\sigma$. At each such node the leaf value or entire child sequence is retained, so $q_u(z-x)=0$. Thus the neighborhood meets $S_\beta$.
\end{proof}

\begin{lemma}\label{Lemma:closure-neighborhood}
Let $1\leq\gamma<\beta$. For every open neighborhood $V$ of zero in $S_\beta$, there is an open neighborhood $W$ of zero in $P_\gamma$ such that
\[
 W\subseteq\overline V^{\,P_{<\gamma}}\cap P_\gamma.
\]
\end{lemma}
\begin{proof}
Choose $\epsilon>0$ and a finite set $F$ of nodes, including the root, such that
\[
 \{z\in S_\beta:t(z)<\epsilon,\ q_u(z)<\epsilon\ (u\in F)\}\subseteq V.
\]
Set
\[
 B=(F\cap I_{\leq\gamma}^\beta)
 \cup\{u_i:u\in F\text{ has level }\delta>\gamma,\ \delta_i<\gamma\}.
\]
This set is finite. At a successor level $\delta>\gamma$, no child has level strictly below $\gamma$; at a limit level, only finitely many children do, by increasing cofinality. Moreover, $B\subseteq I_{\leq\gamma}^\beta$. The same observation applies to the root. Consequently
\[
 W=\left\{x\in P_\gamma:
       \sum_{\beta_i<\gamma}\ell(x(v_i))<\epsilon\right\}
       \cap\bigcap_{u\in B}\{x\in P_\gamma:q_u(x)<\epsilon\}
\]
is an open neighborhood of zero. The sum is finite and all its node evaluations are defined and continuous on $P_\gamma$.

Fix $x\in W$, and let $U$ be any neighborhood of $x$ in $P_{<\gamma}$. Choose a finite set $E\subseteq I_{<\gamma}^\beta$ and $\eta>0$ such that
\[
 \{y\in P_{<\gamma}:q_u(y-x)<\eta\ (u\in E)\}\subseteq U.
\]
Apply Lemma~\ref{Lemma:construct-element} with $\sigma=\gamma$ and $F=E$. The resulting truncation $z$ belongs to $S_\beta\cap U$. Since $E_\downarrow$ has only levels strictly below $\gamma$, all root children at levels at least $\gamma$ have value zero for $z$. Thus
\[
 t(z)\leq\sum_{\beta_i<\gamma}\ell(x(v_i))<\epsilon.
\]
If $u\in F$ has level at most $\gamma$, deleting child values cannot increase the supremum, and hence $q_u(z)\leq q_u(x)<\epsilon$, because $u\in B$.

If $u\in F$ has level $\delta>\gamma$, all its children of level at least $\gamma$ vanish for $z$. Each remaining child belongs to $B$. There are finitely many of them, and
\[
 q_u(z)\leq\max_{\delta_i<\gamma}\ell(x(u_i))
       \leq\max_{\delta_i<\gamma}q_{u_i}(x)<\epsilon
\]
when this set is nonempty. If it is empty then $q_u(z)=0$. Therefore $z\in V\cap U$. Since $U$ was arbitrary, $x\in\overline V^{\,P_{<\gamma}}$.
\end{proof}

The range $1\leq\gamma<\beta$ is the one needed in the successor-stage argument. Stage zero is supplied directly by density.

\paragraph{The Solecki-stage identity.}

\begin{proposition}\label{Proposition:solecki-subgroups-P0}
Let $N$ be a Polish subgroup of $A$, and let $\xi$ be a countable
ordinal such that $1+\xi\leq\beta$ and
$S_\beta\subseteq N\subseteq P_{<(1+\xi)}$. Then
\[
 s_\eta^{N}(A)=P_{<(1+\eta)}\qquad(\eta\leq\xi).
\]
In particular, this holds for $N\in\{S_\beta,D_\beta,P_\beta,P_{<\beta}\}$
and every $\xi$ with $1+\xi\leq\beta$, including $\xi=\beta$ when
$\beta$ is infinite, so that $1+\beta=\beta$; and it holds for
$N=P_{<\mu}$ and $\xi=\mu$, where $\mu\leq\beta$ is a positive limit.
\end{proposition}
\begin{proof}
For $1\leq\gamma<\beta$, membership in $P_\gamma$ requires countably many sequences of continuous node evaluations to be Cauchy in the complete metric of $G$. Each requirement is of the form
\[
 \bigcap_{m\geq1}\ \bigcup_N\ \bigcap_{i,j\geq N}
 \{x\in P_{<\gamma}:\ell(x(u_i)-x(u_j))\leq 2^{-m}\}.
\]
The innermost sets are closed. Therefore $P_\gamma$ is $\PiB^0_3$ in $P_{<\gamma}$. Lemma~\ref{Lemma:construct-element} gives density of $S_\beta$ in $P_\gamma$, and Lemma~\ref{Lemma:closure-neighborhood} gives the neighborhood-closure hypothesis of Lemma~\ref{Lemma:characterize-solecki}. That lemma, with ambient group $P_{<\gamma}$, subgroup $N=S_\beta$, and $H=P_\gamma$, yields
\[
 s_1^{S_\beta}(P_{<\gamma})=P_\gamma.
\]

We first prove the identity for $N=S_\beta$ and every $\xi$ with
$1+\xi\leq\beta$, by transfinite induction on $\xi$. At zero, density in $A$ gives $s_0^{S_\beta}(A)=A=P_{<1}$. At a successor $\xi=\eta+1$ satisfying the stated bound, set $\gamma=1+\eta$. Then $1\leq\gamma<\beta$, so the induction hypothesis and the preceding identity give
\[
 s_{\eta+1}^{S_\beta}(A)
 =s_1^{S_\beta}(P_{<(1+\eta)})
 =P_{1+\eta}=P_{<(1+\xi)}.
\]
At a nonzero limit $\xi$, one has $1+\xi=\xi$, and the ordinals $1+\eta$, $\eta<\xi$, are cofinal in $\xi$. Taking intersections gives $P_{<\xi}$. Uniqueness of the Polish group topology on a Polish subgroup identifies the cover topologies at each step.

Now let $N$ be as in the statement, and let $\eta\leq\xi$. By
Theorem~\ref{Theorem:characterize-solecki}, $s_\eta^N(A)$ is a
$\PiB^0_{1+\eta+1}$ subgroup of $A$ containing $S_\beta$, so it
contains $s_\eta^{S_\beta}(A)=P_{<(1+\eta)}$. For the reverse
inclusion we induct on $\eta$. The zero case is immediate. At a
successor $\eta=\zeta+1$, put $\gamma=1+\zeta$. Then
$\gamma<1+\xi$, so $N\subseteq P_\gamma$, and
$s_\zeta^N(A)\subseteq P_{<\gamma}$ by the induction hypothesis. Since
the inclusion of $s_\zeta^N(A)$ into $P_{<\gamma}$ is continuous and
$P_\gamma$ is $\PiB^0_3$ in $P_{<\gamma}$, the subgroup
$P_\gamma\cap s_\zeta^N(A)$ is $\PiB^0_3$ in $s_\zeta^N(A)$ and
contains $N$. Theorem~\ref{Theorem:characterize-solecki} gives
\[
 s_{\eta}^N(A)=s_1^N\bigl(s_\zeta^N(A)\bigr)\subseteq P_\gamma
 =P_{<(1+\eta)}.
\]
Limit stages follow by intersections. Finally, the subgroups listed
in the last assertion satisfy the hypothesis: $D_\beta$ is a subgroup
because the base is abelian, and $S_\beta\subseteq D_\beta\subseteq
P_\beta\subseteq P_{<\beta}\subseteq P_{<(1+\xi)}$ when
$1+\xi\leq\beta$, while $S_\beta\subseteq P_{<\mu}$ and
$1+\mu=\mu$.
\end{proof}

\begin{lemma}
\label{Lemma:proper}For every $1\leq\gamma \leq \alpha $, there exists a
continuous group homomorphism $\Phi :G^{I_{<\gamma }^{\alpha }}\rightarrow
P_{<\gamma }$ such that $\Phi \left( z\right) \left( k,n;\gamma _{k},\beta
\right) =z\left( k,n;\gamma _{k},\beta \right) $ for every $z\in
G^{I_{<\gamma }^{\alpha }}$, $\left( n;\beta \right) \in I_{\gamma }^{\alpha
}$, and $k\in \mathbb{N}$.
\end{lemma}

\begin{proof}
For $z\in G^{I_{<\gamma }^{\alpha }}$, define $\Phi \left( z\right) :=x\in
G^{I_{0}^{\alpha }}$ by setting, for $\left( m;\tau \right) \in
I_{0}^{\alpha }$,%
\begin{equation*}
x\left( m;\tau \right) :=\left\{ 
\begin{array}{ll}
z\left( k,n;\gamma _{k},\beta \right) & \text{if }\left( m,\tau \right) \leq
\left( k,n;\gamma _{k},\beta \right) \text{ for some }k\in \mathbb{N}\text{
and }\left( n;\beta \right) \in I_{\gamma }^{\alpha }\text{;} \\ 
1_{G} & \text{otherwise.}%
\end{array}%
\right.
\end{equation*}%
One can prove by induction on $\delta <\gamma $ that $x\in P_{\delta }$, and
for $\left( m,\tau \right) \in I_{\delta }^{\alpha }$,%
\begin{equation*}
x\left( m;\tau \right) =\left\{ 
\begin{array}{ll}
z\left( k,n;\gamma _{k},\beta \right) & \text{if }\left( m,\tau \right) \leq
\left( k,n;\gamma _{k},\beta \right) \text{ for some }k\in \mathbb{N}\text{
and }\left( n;\beta \right) \in I_{\gamma }^{\alpha }\text{;} \\ 
1_{G} & \text{otherwise.}%
\end{array}%
\right.
\end{equation*}%
Thus, $\Phi :G^{I_{<\gamma }^{\alpha }}\rightarrow P_{<\gamma }$ is a
well-defined group homomorphism. Since $P_{<\gamma }$ is a Polish
subgroup of $G^{I_{0}^{\alpha }}$, its Borel structure is the one inherited
from $G^{I_{0}^{\alpha }}$. This easily implies that $\Phi $ is Borel, and
hence continuous.
\end{proof}

\begin{lemma}
\label{Lemma:complexity-sequences}We have that:

\begin{enumerate}
\item $\boldsymbol{\Sigma }_{2}^{0}$ is the complexity class of $\ell
_{1}\left( G,L_{G}\right) $ in $G^{\mathbb{N}}$;

\item $D(\boldsymbol{\Pi }_{2}^{0})$ is the complexity class of $\mathrm{bv}%
_{0}\left( G,L_{G}\right) $ in $G^{\mathbb{N}}$;

\item $\boldsymbol{\Pi }_{3}^{0}$ is the complexity class of $\mathrm{c}%
\left( G\right) $ in $G^{\mathbb{N}}$.
\end{enumerate}
\end{lemma}

\begin{proof}
(1): It is clear that $\ell _{1}\left( G,L_{G}\right) $ is $\boldsymbol{%
\Sigma }_{2}^{0}$ in $G^{\mathbb{N}}$. Since $\ell _{1}\left( G,L_{G}\right) 
$ is a dense, proper subgroup of $G^{\mathbb{N}}$, it is not closed.
A $G_\delta$ subgroup is closed, by the Pettis argument in the proof
of Theorem~\ref{Theorem:characterize-solecki}; hence it is not
$\boldsymbol\Pi^0_2$ either.

(2): We have that $\left( g_{n}\right) _{n\in \mathbb{N}}\in \mathrm{bv}%
_{0}\left( G,L_{G}\right) $ if and only if%
\begin{equation*}
\sum_{n\in \mathbb{N}}L_{G}\left( g_{n+1}^{-1}g_{n}\right) <+\infty
\end{equation*}%
and for all $\varepsilon >0$ and $n_{0}\in \mathbb{N}$ there exists $n\geq
n_{0}$ such that $L_{G}\left( g_{n}\right) <\varepsilon $. This shows that $%
\mathrm{bv}_{0}\left( G,L_{G}\right) $ is $D(\boldsymbol{\Pi }_{2}^{0})$ in $%
G^{\mathbb{N}}$. By Lemma~\ref{Lemma:difference}, it remains to show
that it is neither $\boldsymbol\Pi^0_2$ nor $\boldsymbol\Sigma^0_2$.
It is a dense proper subgroup of $G^{\mathbb N}$, so it is not
closed, and hence not $\boldsymbol\Pi^0_2$, as in~(1). It remains to
prove that it is not $\boldsymbol{\Sigma }%
_{2}^{0}$. Suppose by contradiction that%
\begin{equation*}
\mathrm{bv}_{0}\left( G,L_{G}\right) =\bigcup_{k\in \mathbb{N}}F_{k}
\end{equation*}%
where $F_{k}\subseteq G^{\mathbb{N}}$ is closed for every $k\in \mathbb{N}$.
By the Baire Category Theorem, we can assume without loss of generality that 
$F_{0}$ contains a neighborhood of the identity. Thus, there exists $%
\varepsilon >0$ such that%
\begin{equation*}
\left\{ \left( g_{n}\right) _{n\in \mathbb{N}}\in \mathrm{bv}_{0}\left(
G,L_{G}\right) :\sum_{n\in \mathbb{N}}L_{G}\left( g_{n+1}^{-1}g_{n}\right)
<\varepsilon \text{ and }\mathrm{\mathrm{sup}}_{n\in \mathbb{N}}L_{G}\left(
g_{n}\right) <\varepsilon \right\} \subseteq F_{0}\text{.}
\end{equation*}%
Since we are assuming that $G$ is not discrete---see Remark \ref%
{Remark:nondiscrete}---there exists $g\in G$ such that $0<L_{G}(g)<%
\varepsilon $. Define for $N\in \mathbb{N}$, $x^{\left( N\right) }\in 
\mathrm{bv}_{0}\left( G,L_{G}\right) $ by setting%
\begin{equation*}
x_{k}^{\left( N\right) }=\left\{ 
\begin{array}{cc}
g & \text{if }k\leq N\text{;} \\ 
1_{G} & \text{otherwise.}%
\end{array}%
\right.
\end{equation*}%
Then we have that $x^{\left( N\right) }\in F_{0}$ for every $N\in \mathbb{N}$%
. The sequence $\left( x^{\left( N\right) }\right) _{N\in \mathbb{N}}$
converges in $G^{\mathbb{N}}$ to the element $x\in G^{\mathbb{N}}$ that is
the sequence constantly equal to $g$. Since $F_{0}$ is closed in $G^{\mathbb{%
N}}$, we have that $x\in \mathrm{bv}_{0}\left( G,L_{G}\right) $, which is a
contradiction to the fact that $L_{G}\left( g\right) >0$.

(3): By definition, we have that $\mathrm{c}\left( G\right) $ is $%
\boldsymbol{\Pi }_{3}^{0}$ in $G^{\mathbb{N}}$. By Theorem \ref%
{Theorem:Polishable-complexity}, it suffices to prove that $\mathrm{c}\left(
G\right) $ is not potentially $\boldsymbol{\Sigma }_{3}^{0}$. Let $E_{0}$ be
the relation of tail equivalence in $2^{\mathbb{N}}$, and let $E_{0}^{%
\mathbb{N}}$ be the corresponding product equivalence relation on $\left( 2^{%
\mathbb{N}}\right) ^{\mathbb{N}}=2^{\mathbb{N}\times \mathbb{N}}$.\ The finite-support subgroup of $(\mathbb Z/2\mathbb Z)^{\mathbb N}$ is dense, proper, and $\boldsymbol\Sigma^0_2$, hence has rank $1$. Lemma~\ref{Lemma:product2} gives exact $\boldsymbol\Pi^0_3$ complexity for its countable power. Its coset relation is $E_0^{\mathbb N}$, so Theorem~\ref{Theorem:Polishable-complexity} excludes potential $\boldsymbol\Sigma^0_3$ complexity.

Thus, it suffices to define a Borel function $2^{\mathbb{N}\times \mathbb{N}%
}\rightarrow G^{\mathbb{N}}$ that is a Borel reduction from $E_{0}^{\mathbb{N%
}}$ to the coset relation of $\mathrm{c}\left( G\right) $ inside $G^{\mathbb{%
N}}$. We argue as in \cite[Lemma 8.5.3]{gao_invariant_2009}. Fix a bijection 
$\left\langle \cdot ,\cdot \right\rangle :\mathbb{N}\times \mathbb{N}%
\rightarrow \mathbb N $ such that, if $n\leq n^{\prime }$ and $m\leq m^{\prime
} $, then $\left\langle n,m\right\rangle \leq \left\langle n^{\prime
},m^{\prime }\right\rangle $. Fix also a sequence $\left( g_{n}\right)
_{n\in \mathbb{N}}$ in $G$ such that $0<L_{G}\left( g_{n}\right) <2^{-(n+1)}$
for every $n\in \mathbb{N}$. Define $\Xi :2^{\mathbb{N}\times \mathbb{N}%
}\rightarrow G^{\mathbb N }$, $\varphi \mapsto a$ by setting%
\begin{equation*}
a_{\left\langle n,m\right\rangle }=\left\{ 
\begin{array}{cc}
g_{n} & \varphi \left( n,m\right) =1\text{;} \\ 
1_{G} & \varphi \left( n,m\right) =0\text{.}%
\end{array}%
\right.
\end{equation*}%
Fix $\varphi ,\psi \in 2^{\mathbb{N}\times \mathbb{N}}$. Define $\Xi \left(
\varphi \right) =a$ and $\Xi \left( \psi \right) =b$.

Suppose that $\varphi E_{0}^{\mathbb{N}}\psi $. Thus we have that for every $%
n\in \mathbb{N}$ there exists $M_{n}\in \mathbb{N}$ such that $\varphi
\left( n,m\right) =\psi \left( n,m\right) $ for $m\geq M_{n}$. Fix $%
\varepsilon >0$ and fix $N\in \mathbb{N}$ such that $2^{-N}<\varepsilon $.
Define then%
\begin{equation*}
M=\max \left\{ M_{n}:n<N\right\} \text{.}
\end{equation*}%
We claim that for $k\geq \left\langle N,M\right\rangle $ we have that $%
L_{G}\left( a_{k}^{-1}b_{k}\right) <\varepsilon $. Indeed, suppose that $%
k\geq \left\langle N,M\right\rangle $. Then $k=\left\langle n,m\right\rangle 
$ for some $n,m\in \mathbb{N}$. If $n\geq N$ then we have that%
\begin{eqnarray*}
L_{G}\left( a_{k}^{-1}b_{k}\right) &\leq &L_{G}\left( a_{k}\right)
+L_{G}\left( b_{k}\right) \\
&\leq &2L_{G}\left( g_{n}\right) <2^{-n}\leq 2^{-N}<\varepsilon \text{.}
\end{eqnarray*}%
If $n<N$ then we must have that $m\geq M\geq M_{n}$, and hence $\varphi
\left( n,m\right) =\psi \left( n,m\right) $ and $a_{k}=b_{k}$. This shows
that $a^{-1}b\in \mathrm{c}\left( G\right) $.

Conversely, suppose that $a^{-1}b\in\mathrm c(G)$. Its difference
sequence has a subsequence tending to $1_G$, obtained by choosing
pairing indices whose first coordinates tend to infinity, since
$g_n\to1_G$. Thus the limit of the whole convergent sequence is
$1_G$. Fix $%
n_{0}\in \mathbb{N}$. Then we have that there exists $k_{0}\in \mathbb{N}$
such that for $k\geq k_{0}$, $L_{G}\left( a_{k}^{-1}b_{k}\right)
<L_{G}\left( g_{n_{0}}\right) $. Thus, for $k\geq k_{0}$, if $k=\left\langle
n_{0},m\right\rangle $ for some $m\in \mathbb{N}$, we must have $a_{k}=b_{k}$
and $\varphi \left( n_{0},m\right) =\psi \left( n_{0},m\right) $. Thus, if $%
m_{0}\in \mathbb{N}$ is such that $\left\langle n_{0},m_{0}\right\rangle
\geq k_{0}$, we must have that $\varphi \left( n_{0},m\right) =\psi \left(
n_{0},m\right) $ for all $m\geq m_{0}$. As this holds for every $n_{0}\in 
\mathbb{N}$, $\varphi E_{0}^{\mathbb{N}}\psi $, concluding the proof.
\end{proof}

\begin{corollary}
\label{Corollary:basic-complexity}For every $0<\gamma \leq\alpha $, $P_{\gamma }$
is a proper subgroup of $P_{<\gamma }$. The complexity class of $S_{\alpha }$%
, $D_{\alpha }$, $P_{\alpha }$, respectively, inside $P_{<\alpha }$ is $%
\boldsymbol{\Sigma }_{2}^{0}$, $D(\boldsymbol{\Pi }_{2}^{0})$, and $%
\boldsymbol{\Pi }_{3}^{0}$, respectively.
\end{corollary}

\begin{proof}
Fix $0<\gamma \leq\alpha $ and $\left( n;\beta \right) \in I_{\gamma }^{\alpha }$%
. By Lemma \ref{Lemma:proper} there exists $x\in P_{<\gamma }$ such that $%
\boldsymbol{x}\left( n;\beta \right) $ is not convergent. Such an $x$ does
not belong to $P_{\gamma }$, thus showing that $P_{\gamma }$ is a proper
subgroup of $P_{<\gamma }$.

We now prove the assertion about $P_{\alpha }$, as the other assertions are
proved in a similar fashion. Recall that $I_{\alpha }^{\alpha }=\left\{
\left( \varnothing ;\varnothing \right) \right\} $. Define 
\begin{equation*}
H=\left\{ x\in G^{I_{<\alpha }^{\alpha }}:\left( x\left( k;\alpha _{k}\right) \right)
_{k\in \mathbb{N}}\in \mathrm{c}\left( G\right) \right\} \text{.}
\end{equation*}%
By Lemma \ref{Lemma:complexity-sequences} we have that $\boldsymbol{\Pi }%
_{3}^{0}$ is the complexity class of $H$ in $G^{I_{<\alpha }^{\alpha }}$ and of $%
\mathrm{c}\left( G\right) $ in $G^{\mathbb{N}}$.

By Lemma \ref{Lemma:proper} there exists a continuous group homomorphism $%
\Phi :G^{I_{<\alpha }^{\alpha }}\rightarrow P_{<\alpha }$ such that $\Phi \left(
x\right) \left( k;\alpha _{k}\right) =x\left( k;\alpha _{k}\right) $ for
every $k\in \mathbb{N}$, and hence $\Phi ^{-1}\left( P_{\alpha }\right) =H$.
Similarly, the function $\Psi :P_{<\alpha }\rightarrow G^{\mathbb{N}}$, $%
x\mapsto \boldsymbol{x}\left( \alpha \right) $ is a continuous group
homomorphism such that $\Psi ^{-1}\left( \mathrm{c}\left( G\right) \right)
=P_{\alpha }$. Thus, $\boldsymbol{\Pi }_{3}^{0}$ is the complexity class of $%
P_{\alpha }$ in $P_{<\alpha }$.
\end{proof}

We are now in position to present the proof of Theorem \ref%
{Theorem:complexityP}.

\begin{proof}[Proof of Theorem~\ref{Theorem:complexityP}]
Put $\rho=\lambda+n$, so $1+\rho=\alpha$.
Proposition~\ref{Proposition:solecki-subgroups-P0} gives
$s_\rho^{N}(P_0)=P_{<\alpha}$ for $N\in\{S_\alpha,D_\alpha,P_\alpha\}$.
Corollary~\ref{Corollary:basic-complexity} gives their properness
and their relative complexity classes
$\boldsymbol\Sigma^0_2$, $D(\boldsymbol\Pi^0_2)$, and
$\boldsymbol\Pi^0_3$ in $P_{<\alpha}$.
Consequently each has exact rank $\rho+1$.
Theorem~\ref{Theorem:complexity}(2a)--(2d) gives the stated
successor classes, apart from deciding the case of $S_\alpha$
when $n\geq1$. In that case write $\alpha=\delta+1$, so
$\delta\geq1$, and choose one child $u$ of the root; it has
level $\delta$. In Lemma~\ref{Lemma:proper}, prescribe arbitrary
values $y_i$ at the children $u_i$, and the identity at the
children of every other level-$\delta$ node. This gives a
continuous map
\[
 \Psi:G^{\mathbb N}\longrightarrow P_{<\delta}.
\]
Its support lies in the subtree below $u$. If $y$ converges,
all level-$\delta$ rows converge, and the root row has just one
possibly nonidentity entry, namely the limit of $y$. It is
therefore summable. Conversely, membership in $S_\alpha$
requires convergence of the row at $u$. Thus
\[
 \Psi^{-1}(S_\alpha)=\mathrm c(G).
\]
Lemma~\ref{Lemma:complexity-sequences} gives
$\mathrm c(G)\notin\boldsymbol\Sigma^0_3(G^{\mathbb N})$.
Proposition~\ref{Proposition:solecki-subgroups-P0} identifies
$P_{<\delta}=s_{\lambda+n-1}^{S_\alpha}(P_0)$.
Hence Theorem~\ref{Theorem:complexity}(2d), or directly
Corollary~\ref{Corollary:relativize}, excludes the global additive
alternative. The class of $S_\alpha$ is the asserted difference
class, including when $\delta$ is a limit.

For Item~(1), let $\mu\leq\alpha$ be a positive limit.
Proposition~\ref{Proposition:solecki-subgroups-P0}, with
$N=P_{<\mu}$, gives
\[
 s_\xi^{P_{<\mu}}(P_0)=P_{<(1+\xi)}\quad(\xi\leq\mu).
\]
Properness of the earlier stages gives
exact rank $\mu$, and Theorem~\ref{Theorem:complexity}(1) gives
complexity $\boldsymbol\Pi^0_\mu$. Taking $\mu=\lambda$
proves Item~(1), also when the root is higher than $\lambda$.
\end{proof}

\subsection{Realization in an arbitrary Polish target}
\label{sec:relative-lower-bounds}
Over an abelian base, the groups of Theorem~\ref{Theorem:complexityP},
together with $P_{<\lambda}$ and closed subgroups, realize every index
in $\mathcal I_{\mathbb Z}$. We now remove the assumptions on the base.
To prove Theorem~\ref{Theorem:main} for an arbitrary nontrivial
Polish group $K$, choose $g\ne1_K$ and let
$A=\overline{\langle g\rangle}^{\,K}$. This is a nontrivial closed
abelian Polish group. Every abelian Polish group admits a compatible
complete invariant metric, so $A$ is CLI. Apply the construction to the nondiscrete
abelian group $A^{\mathbb N}$. The resulting subgroup embeds in
the closed subgroup $(A^{\mathbb N})^{\mathbb N}$ of
$(K^{\mathbb N})^{\mathbb N}\cong K^{\mathbb N}$.
All required upper bounds extend from a closed subspace, and all
lower bounds persist by restriction. The constructed complete
invariant metric also proves the CLI assertion.

\begin{remark}[Ranks and classes]\label{rem:hjorth-ranks}
For Solecki ranks alone, the existence of subgroups of every rank in
$K^{\mathbb N}$ also follows from a theorem of Hjorth: every
uncountable abelian Polish group has Polish subgroups of arbitrarily
high Borel complexity
\cite[Introduction and Corollary~3.5]{hjorth_subgroups_2006}. Indeed,
$K^{\mathbb N}$ contains the uncountable abelian Polish group
$A^{\mathbb N}$ as a closed subgroup, and Solecki stages computed in a
closed subgroup containing $H$ agree with those computed in the
ambient group, by Theorem~\ref{Theorem:characterize-solecki}.
Moreover, if a Polish subgroup $H$ of a Polish group $G$ has Solecki
rank $\rho$ and $\alpha\leq\rho$, then $s_\alpha^H(G)$ has Solecki
rank exactly $\alpha$, again by the minimality in
Theorem~\ref{Theorem:characterize-solecki}. What
Theorem~\ref{Theorem:main} adds is the exact complexity class at
each rank, that is, every index in $\mathcal I_{\mathbb Z}$. Farah and
Solecki asked whether every uncountable Polish group contains
Polish subgroups of every prescribed rank, and noted that Hjorth's
theorem gives a positive answer for abelian groups
\cite[Question~6.1, p.~540]{farah_borel_2006}. The groups
$K^{\mathbb N}$ contain uncountable closed abelian subgroups, so
Theorem~\ref{Theorem:main} does not bear on that question beyond the
abelian case.
\end{remark}

The use of an abelian base is needed for the bounded-variation
subgroup. The following example explains why it cannot simply be
dropped.

\begin{example}[Bounded variation in a nonabelian group]
Let $G=\prod_{i\geq1}\mathfrak S_3$, let $a=(12)$ and $b=(23)$,
and let $\ell_i$ be the weighted word length on $\mathfrak S_3$
assigning weights $\varepsilon_i=2^{-3i}$ to $a$ and
$\delta_i=2^{-i}$ to $b$. Then
\[
 L_G((g_i))=\sum_{i\geq1}\ell_i(g_i)
\]
is a compatible complete length on this compact group: each
$\ell_i$ is positive off the identity and bounded by $3\delta_i$.
Write $a_i,b_i$ for the elements supported in coordinate $i$.
Start each block at $(x,y)=(1_G,1_G)$, change $y$ to $b_i$,
and alternate $x$ between $1_G$ and $a_i$ for $2^{i+1}$
transitions while keeping $y=b_i$. Then $x=1_G$ again;
return $y$ to $1_G$ and proceed to the next block.
Both sequences tend to $1_G$, and
their variations are bounded by
\[
 \sum_i2^{i+1}\varepsilon_i<\infty,
 \qquad 2\sum_i\delta_i<\infty,
\]
respectively. At each internal transition, the increment of $xy$
is $b_i^{-1}a_i b_i$, whose length is at least $\delta_i$, because
any word representing $b^{-1}ab$ must use the letter $b$.
Thus the variation of $xy$ is at least
$\sum_i2^{i+1}\delta_i=\infty$.

\end{example}

Recall that a (pseudo-)ultralength function on a group $H$ is a
(pseudo-)length function $L$ such that $L\left( hh^{\prime }\right) \leq
\max \left\{ L\left( h\right) ,L\left( h^{\prime }\right) \right\} $ for $%
h,h^{\prime }\in H$. A Polish group $G$ is non-Archimedean if and only if it
admits a compatible ultralength function \cite[Theorem 2.4.1]%
{gao_invariant_2009}. We now give the non-Archimedean realization;
compare \cite[Section~5]{hjorth_borel_1998}.

\begin{theorem}\label{thm:nonarch-cli-realization}
Let $\xi\in\omega_1[1/2]$ and let $G$ be a nontrivial countable
discrete group. There exists a nonopen non-Archimedean CLI Polish subgroup
$N$ of $G^{\mathbb N}$ with
$\ell_{\boldsymbol\Delta}(G^{\mathbb N}/N)=\xi$.
\end{theorem}

For this construction we may replace $G$ by any of its nontrivial
cyclic subgroups: the resulting ambient countable power is a closed
subgroup of the original target. Thus we may assume that $G$ is
abelian. Define $H:=G^{\mathbb N}$. This is a non-Archimedean CLI group. The
topology on $H$ is induced by the ultralength function%
\begin{equation*}
L_{H}\left( \left( g_{n}\right) _{n\in \mathbb{N}}\right) =\exp \left( -\min
\left\{ n\in \mathbb{N}:g_{n}\neq 1_{G}\right\} \right) \text{.}
\end{equation*}%
Notice that the subgroup $\mathrm{c}\left( H\right) $ of $H^{\mathbb{N}}$ of
convergent sequences is a non-Archimedean CLI Polish subgroup of $H^{%
\mathbb{N}}$ of complexity class $\boldsymbol{\Pi }_{3}^{0}$ with topology
induced by the ultralength function 
\begin{equation*}
L_{\infty }\left( \left( h_{n}\right) _{n\in \mathbb{N}}\right) =\max
\left\{ L_{H}\left( h_{n}\right) :n\in \mathbb{N}\right\} \text{.}
\end{equation*}%
The subgroup $\sigma \left( H\right) $ of $H^{\mathbb{N}}$ consisting of
sequences $\left( h_{n}\right) _{n\in \mathbb{N}}$ such that the sequence $%
\left( h_{n}\left( 1\right) \right) _{n\in \mathbb{N}}$ in $G$ is eventually
equal to $1_{G}$ is a non-Archimedean CLI Polish subgroup of $H^{\mathbb{%
N}}$ of complexity class $\boldsymbol{\Sigma }_{2}^{0}$.
Its Polish topology is the product topology under the identification
\[
 \sigma(H)\cong G^{(\mathbb N)}\times
 G^{\mathbb N\times(\mathbb N\setminus\{1\})},
\]
where the countable group $G^{(\mathbb N)}$ of finite-support first
columns is discrete, and the second factor has its product topology.
The inclusion into $H^{\mathbb N}$ is continuous. This description
also proves completeness and the non-Archimedean and CLI properties.

Fix $\alpha <\omega _{1}$. We define by recursion on $\gamma \leq \alpha $ a
decreasing sequence $\left( F_{\gamma }\right) _{\gamma \leq\alpha }$ of
non-Archimedean Polish subgroups of $H^{I_{0}^{\alpha }}$. We also
recursively define, for $x\in F_{\gamma }$ and $\left( n;\beta \right) \in
I_{\leq\gamma}^{\alpha }$, the values $x\left( n;\beta \right) \in H$. We set $%
F_{0}=H^{I_{0}^{\alpha }}$. If $F_{\delta }$ has been defined for every $%
\delta <\gamma $, define $F_{<\gamma }=\bigcap_{\delta <\gamma }F_{\delta }$%
. Let $F_{\gamma }$ contain those $x\in F_{<\gamma }$ such that, for every $%
\left( n;\beta \right) \in I_{\gamma }^{\alpha }$, the sequence $\boldsymbol{%
x}\left( n;\beta \right) :=\left( x\left( k,n;\gamma _{k},\beta \right)
\right) _{k\in \omega }$ is convergent in $H$. For $x\in F_{\gamma }$ and $%
\left( n;\beta \right) \in I_{\gamma }^{\alpha }$, we define $x\left(
n;\beta \right) $ to be the limit of $\boldsymbol{x}\left( n;\beta \right) $%
. Then we have that the non-Archimedean Polish group topology on $F_{\gamma
} $ is induced by the restriction of the continuous pseudo-ultralength
functions on $F_{\delta }$ for $\delta <\gamma $ together with the
pseudo-ultralength function%
\begin{equation*}
L_{\gamma }^{\left( n;\beta \right) }\left( x\right) =L_{\infty }\left( 
\boldsymbol{x}\left( n;\beta \right) \right)
\end{equation*}%
for $\left( n;\beta \right) \in I_{\gamma }^{\alpha }$. This concludes the
recursive definition of the non-Archimedean Polish subgroups $F_{\gamma
} $ of $H^{I_{0}^{\alpha }}$ for $\gamma \leq \alpha $. Notice that in
particular $F_{\alpha }$ contains the elements $x\in F_{<\alpha }$ such that 
$\boldsymbol{x}\left( \alpha \right) :=\left( x\left( n,\alpha _{n}\right)
\right) _{n\in \mathbb{N}}$ belongs to $c\left( H\right) $. Define $%
Z_{\alpha }$ to contain those elements $x\in F_{<\alpha }$ such that $%
\boldsymbol{x}\left( \alpha \right) $ belongs to $\sigma \left( H\right) $.
Give $Z_\alpha$ its graph topology: it is identified with the
closed subgroup
\[
 \{(x,z)\in F_{<\alpha}\times\sigma(H):
                 \boldsymbol x(\alpha)=z\}.
\]
The top-row map is continuous into $H^{\mathbb N}$, so this is
indeed closed. The same closed-graph construction at each earlier
stage, using $\mathrm c(H)$ with its supremum ultralength, proves
inductively that all $F_\gamma$, $F_{<\gamma}$, and $Z_\alpha$
are non-Archimedean CLI Polish groups. Their Borel structures
are inherited from the leaf-coordinate space, by continuous
injectivity and the Lusin--Souslin theorem.
We obtain the following exact complexity statements.

\begin{theorem}
\label{Theorem:complexitynAP}Adopt the notations above. Suppose that $\alpha
=1+\lambda +n$ where $\lambda <\omega _{1}$ is either limit or zero and $%
n<\omega $.

\begin{enumerate}
\item If $n=0$ and $\lambda $ is limit, then $F_{<\lambda }$ has Solecki
rank $\lambda $ and complexity class $\boldsymbol{\Pi }_{\lambda }^{0}$ in $%
H^{I_{0}^{\alpha }}$.

\item If $n=0$, then $Z_{1+\lambda }$ and $F_{1+\lambda }$ have Solecki rank 
$\lambda +1$ in $H^{I_{0}^{\alpha }}$, and complexity class $\boldsymbol{%
\Sigma }_{1+\lambda +1}^{0}$ and $\boldsymbol{\Pi }_{1+\lambda +2}^{0}$,
respectively;

\item if $n\geq 1$, then $Z_{1+\lambda+n }$ and $F_{1+\lambda+n }$ have Solecki
rank $\lambda +n+1$ in $H^{I_{0}^{\alpha }}$, and complexity class $D(%
\boldsymbol{\Pi }_{1+\lambda +n+1}^{0})$ and $\boldsymbol{\Pi }_{1+\lambda
+n+2}^{0}$, respectively.
\end{enumerate}
\end{theorem}

\begin{proof}
Use the summable construction $S_\alpha$ over the abelian base
$H$ and the same length $L_H$. Summability forces all but
finitely many first coordinates of the root row to be the identity,
since each nonidentity first coordinate contributes $e^{-1}$.
Consequently
\[
 S_\alpha\subseteq Z_\alpha\subseteq F_{<\alpha},
 \qquad F_\gamma=P_\gamma\quad(\gamma\leq\alpha).
\]
The latter equalities hold also topologically. Monotonicity in
the generating subgroup and
Proposition~\ref{Proposition:solecki-subgroups-P0} give
\[
 F_{<(1+\xi)}\subseteq s_\xi^{Z_\alpha}(F_0)
                  \quad(1+\xi\leq\alpha).
\]
For the reverse inclusion, argue by induction. At a successor step
put $\gamma=1+\eta<\alpha$. The subgroup $F_\gamma$ is
$\boldsymbol\Pi^0_3$ in $F_{<\gamma}$, by the Cauchy criterion
for its countably many rows, and contains $Z_\alpha$. Minimality
of the first Solecki subgroup therefore puts the next stage in
$F_\gamma=F_{<(\gamma+1)}$. At limits take intersections;
at stage zero use the density of $S_\alpha$ in $F_0$. We have proved
\[
 s_\xi^{Z_\alpha}(F_0)=F_{<(1+\xi)}
                  \quad(1+\xi\leq\alpha).
\]
Thus the required density and interior stage identities follow from
the finite-child closure argument for $S_\alpha$.
The identical sandwich argument applies to $F_\alpha$ and, at
a positive limit $\mu\leq\alpha$, to $F_{<\mu}$.

The root insertion map of Lemma~\ref{Lemma:proper} is a continuous
right inverse of the top-row map into $H^{\mathbb N}$. It
therefore transfers the relative complexity classes of
$\sigma(H)$ and $\mathrm c(H)$ to $Z_\alpha$ and $F_\alpha$
inside $F_{<\alpha}$. The former is a dense proper
$\boldsymbol\Sigma^0_2$ subgroup: a union over $N$ of the closed
conditions $h_i(1)=1_G$ for $i\geq N$; it is not
$\boldsymbol\Pi^0_2$, since a dense $G_\delta$ subgroup would
be the whole group. The latter has exact class
$\boldsymbol\Pi^0_3$ by Lemma~\ref{Lemma:complexity-sequences}.
Writing $\alpha=1+\rho$, the displayed stage identity and these
proper relative classes give exact rank $\rho+1$. The
non-Archimedean cases of Theorem~\ref{Theorem:complexity}
give Items~(2)--(3). At a positive limit, the preceding stage
identity and properness give exact rank $\lambda$ and class
$\boldsymbol\Pi^0_\lambda$, proving Item~(1).
This also proves Theorem~\ref{thm:nonarch-cli-realization}, since
a countable power of $H$ is isomorphic to $G^{\mathbb N}$.
\end{proof}

\section{Fr\'echet subspaces\label{section:Frechetable}}

In this and the following section, we assume all the vector spaces to be
over the reals. Similar considerations apply to complex vector spaces.
Recall that a\emph{\ Fr\'{e}chet space} is a locally convex topological
vector space whose topology is induced by a complete, translation-invariant
metric. Thus, the additive group of a separable Fr\'{e}chet space is a
Polish group. In analogy with the notion of Polish subgroup of a Polish
group, we consider the notion of Fr\'echet subspace of a separable Fr%
\'{e}chet space.

\begin{definition}
\label{Definition:Frechetable}Suppose that $X$ is a separable Fr\'{e}chet
space, and $Y$ is a subspace of $X$.\ Then we call $Y$ a \emph{Fr\'echet subspace} if there exists a separable Fr\'{e}chet space topology on $Y$
whose open sets are Borel in $X$.
\end{definition}

This notion was considered by Saint Raymond in \cite%
{saint-raymond_espaces_1976}: a subspace $Y$ of $X$ is a Fr\'echet subspace if
and only if \emph{it has a separable model }according to \cite[Definition 1]%
{saint-raymond_espaces_1976}. Notice that a Fr\'echet subspace of $X$
is, in particular, a Polish subgroup of the additive group of $X$. Thus,
if it exists, the separable Fr\'{e}chet space topology on $Y$ as in
Definition \ref{Definition:Frechetable}, is unique; see also \cite[Corollary
4.38 and Example~5, p.~117]{osborne_locally_2014}. A subspace $Y$ of a\ separable Fr\'{e}chet
space $X$ is a Fr\'echet subspace if and only if there exists a separable Fr\'{e}%
chet space $Z$ and a continuous linear map $\varphi :Z\rightarrow X$ with
image equal to $Y$ \cite[Proposition 4]{saint-raymond_espaces_1976}. If $Y$
is a Fr\'echet subspace of $X$, then the separable Fr\'{e}chet space
topology on $Y$ is the Mackey topology $\tau(Y,Y^\sharp)$,
where $Y^\sharp$ is the space of Borel linear functionals for the
inherited Borel structure. Thus it is the finest locally convex
topology whose continuous dual is exactly $Y^\sharp$
\cite[Th\'eor\`eme~9, p.~215]{saint-raymond_espaces_1976}. Furthermore, we
have that a subspace $Y$ of $X$ is a Fr\'echet subspace if and only if it is a
Polish subgroup of the additive group of $X$, and the Polish topology on 
$Y$ has a basis of zero neighborhoods consisting of absolutely convex sets.
Here is the step that gives a topological vector space in the reverse
direction. If $U$ is such a neighborhood, $L=\operatorname{span}_{\mathbb R}U$
is an open subgroup of $Y$, hence has countable index by separability.
The real vector space $Y/L$ is therefore zero, since any nonzero real
vector space is uncountable. Thus
$Y=L=\bigcup_{n\geq1}nU$, so $U$ is absorbing.
Its Minkowski gauge $p_U$ is consequently a finite seminorm.
Choose an absolutely convex neighborhood $W$ with $W+W\subseteq U$;
then $W\subseteq\{p_U<1\}\subseteq U$.
The gauge balls are open in the group topology: for each
$\varepsilon>0$ choose an integer $m$ with $1/m<\varepsilon$;
the neighborhood $m^{-1}U=\{y:my\in U\}$ is contained in
$\{p_U<\varepsilon\}$, and use subadditivity at each point of a ball.
The seminorm topology is thus exactly the given group topology,
so scalar multiplication is jointly continuous and $Y$ is locally
convex. A Polish abelian group is complete in its group uniformity;
\cite[Proposition~3.33 and Corollary~3.36]{osborne_locally_2014}
then give the Fr\'echet property.

\begin{lemma}
\label{Lemma:frechetable-solecki}Suppose that $X$ is a separable Fr\'{e}chet
space, and $Y$ a Fr\'echet subspace of $X$. The first Solecki subgroup 
$s_{1}^{Y}\left( X\right) $ of $X$ relative to $Y$, where $X$ is regarded as
an additive Polish group and $Y$ as a Polish subgroup of $X$, is a Fr\'echet subspace of $X$.
\end{lemma}

\begin{proof}
By the first-stage formula, for $x\in X$, $x\in s_1^Y(X)$ if and only if for every
open neighborhood $V$ of zero in $Y$ there exists $z\in Y$ such that $x+z\in 
\overline{V}^{X}$. If $x\in s_1^Y(X)$, $\lambda \in \mathbb{R}$ is nonzero, and $V$
is an open neighborhood of zero in $Y$, then there exists $z\in Y$ such that 
$x+z\in \overline{\lambda ^{-1}V}^{X}$, whence $\lambda x+\lambda z\in 
\overline{V}^{X}$. This shows that $\lambda x\in s_{1}^{Y}\left( X\right) $,
whence $s_{1}^{Y}\left( X\right) $ is a subspace of $X$.

We now show that $s_{1}^{Y}\left( X\right) $ is a Fr\'echet subspace of $X$. Since $Y$
is a separable Fr\'{e}chet space, by the remarks above it has a basis $%
\left( V_{n}\right) _{n\in \omega }$ of zero neighborhoods consisting of
absolutely convex sets. Thus, $(\overline{V}_{n}^{X}\cap s_{1}^{Y}\left(
X\right) )_{n\in \omega }$ is a basis of zero neighborhoods in $%
s_{1}^{Y}\left( X\right) $ consisting of absolutely convex sets. Thus, $%
s_{1}^{Y}\left( X\right) $ is a Fr\'echet subspace of $X$ by the
remarks above again.
\end{proof}

As an immediate consequence of Lemma \ref{Lemma:frechetable-solecki} and
Theorem \ref{Theorem:characterize-solecki} by induction on $\alpha <\omega
_{1}$ we have the following.

\begin{theorem}
\label{Theorem:characterize-Solecki-Frechet}Suppose that $X$ is a separable
Fr\'{e}chet space, $Y$ is a Fr\'echet subspace of $X$, and $\alpha
<\omega _{1}$. Then the $\alpha $-th Solecki subgroup $s_{\alpha }^{Y}\left(
X\right) $ of $X$ relative to $Y$, where $X$ and $Y$ are regarded as
additive groups, is the smallest $\boldsymbol{\Pi }_{1+\alpha +1}^{0}$ Fr\'echet subspace of $X$ containing $Y$.
\end{theorem}

This linear minimality statement is Saint Raymond's
\cite[Corollaire~53]{saint-raymond_espaces_1976}; the proof here
identifies its stages with the group-theoretic Solecki stages.
Explicitly, if $A:E\to F$ is a continuous linear injection and
$v_{0,\alpha}:N_\alpha\to F$ is his canonical model, then
$v_{0,\alpha}(N_\alpha)=s_\alpha^{A(E)}(F)$: both are the smallest
$\boldsymbol\Pi^0_{1+\alpha+1}$ linear subspace containing $A(E)$,
by his Corollaire~53 and Theorem~\ref{Theorem:characterize-Solecki-Frechet}.

The group construction of Section~\ref{Section:existence}, carried
out over the base group $(\mathbb R,+)$, gives Fr\'echet realizations.

\begin{theorem}\label{Theorem:complexity-Frechetable}
Let $\xi\in\mathcal I_{\mathbb Z}$ and let $X$ be a nontrivial
separable Fr\'echet space. There exists a proper Fr\'echet subspace $Y$ of
$X^{\mathbb N}$ with $\ell_{\boldsymbol\Delta}(Y;X^{\mathbb N})=\xi$.
\end{theorem}

\begin{proof}[Proof of Theorem~\ref{Theorem:complexity-Frechetable}]
The closed class $\boldsymbol\Pi^0_1$ is realized by $\{0\}$ in
the nontrivial space $X^{\mathbb N}$. For the other classes,
first take the base group in Section~\ref{Section:existence} to be
$(\mathbb R,+)$ with its usual absolute value. The spaces $P_\gamma$
and their projective intersections are separable Fr\'echet spaces:
their topologies are generated by the coordinate and supremum
seminorms in the construction. At the terminal step the summability
and variation conditions give separable Banach spaces of rows,
so their inverse images $S_\alpha$ and $D_\alpha$, with the graph
topologies, are also Fr\'echet. The insertion and top-row maps are
linear. Theorem~\ref{Theorem:complexityP} therefore realizes every
class in the stated sublist in $\mathbb R^{\mathbb N}$.
For a nontrivial separable Fr\'echet space $X$, choose $z\ne0$.
The map $(t_n)\mapsto(t_nz)$ is a linear homeomorphism of
$\mathbb R^{\mathbb N}$ onto the closed subspace
$(\mathbb Rz)^{\mathbb N}$ of $X^{\mathbb N}$.
Extension from a closed subspace preserves all required upper
bounds, and restriction preserves the exact lower bounds.
\end{proof}

The additional successor additive classes in Theorem~\ref{Corollary:complexity-Frechetable} are supplied by Theorem~\ref{thm:additive-ranges}; its Banach spaces are already separable Fr\'echet spaces.

\section{Banach subspaces\label{Section:Banachable}}

This section proves the upper restriction for Banach subspaces in
Theorem~\ref{Theorem:complexity-Banachable}: among the indices of
$\omega_1^{\boldsymbol\Delta}$, exactly the positive limit ordinals do
not occur. The restriction is a classical consequence of Saint
Raymond's rank theory \cite{saint-raymond_espaces_1976}, in the form of
\cite[Theorem~1.1]{BDP}, which Banakh, Dobrowolski, and Plichko
attribute to him, and of \cite[Theorems~1.2--1.3]{BDP}; we derive it
here from the subgroup classification. Successor stages of a Banach
subspace are again Banach subspaces
(Proposition~\ref{prop:banachable-successors}), whereas a limit stage
can be a Banach space only if the rank is smaller
(Proposition~\ref{Proposition:Banachable-limit}); this excludes the
exact classes $\boldsymbol\Pi^0_\lambda$, $\lambda$ a positive limit
(Corollary~\ref{cor:banachable-limit-collapse}). The existence
statements are also stated in
\cite[Theorems~2.6--2.8 and Proposition~2.10]{BDP}, except for
$D(\boldsymbol\Pi^0_{1+\lambda+1})$ with $\lambda$ zero or a countable
limit, the case $\lambda=0$ being noted there as known; the closed
case is elementary. Sections~\ref{sec:additive-ranges}
and~\ref{sec:banach-realizations} supply the additive and difference
realizations, and Appendix~\ref{app:tree-spaces} gives explicit
models of the multiplicative classes, with all their Solecki stages
computed.

Let $V$ be a separable Fr\'{e}chet space. A subspace $X\subseteq V$ is called a 
\emph{Banach subspace} if it is the image of a continuous linear map $%
T:Z\rightarrow V$ for some separable Banach space $Z$. Equivalently, $X$
admits a separable Banach space topology that makes the inclusion map $X\rightarrow V$
continuous. We have that the Solecki subgroups whose index is a successor
ordinal associated with a Banach subspace of a separable Fr\'{e}chet
space are also Banach subspaces. Recall that a subset $A$ of a topological vector
space $Z$ is bounded if for every zero neighborhood $U$ in $Z$ there exists $%
t>0$ such that $A\subseteq tU$. A Fr\'{e}chet space is a Banach space if and
only if it has a bounded zero neighborhood.

\begin{proposition}\label{prop:banachable-successors}
Suppose that $V$ is a separable Fr\'{e}chet space, $X\subseteq V$ is
a Banach subspace. Then $s_{\alpha +1}^{X}\left( V\right) \subseteq V$ is
a Banach subspace for every $\alpha <\omega _{1}$.
\end{proposition}

In canonical-model form this is
\cite[Th\'eor\`eme~38]{saint-raymond_espaces_1976}, via the identification
following Theorem~\ref{Theorem:characterize-Solecki-Frechet}.

\begin{proof}
It suffices to consider the case $\alpha =0$. Let $B$ be a bounded zero
neighborhood in $X$. Define $C:=\overline{B}^{V}\cap s_{1}^{X}\left(
V\right) $. As $B$ is bounded in $X$, $\left( 2^{-n}B\right) _{n\in \omega }$
is a basis of zero neighborhoods in $X$. Thus, $\left( 2^{-n}C\right) _{n\in
\omega }$ is a basis of zero neighborhoods in $s_{1}^{X}\left( V\right) $.
Thus, $C$ is a bounded zero neighborhood in $s_{1}^{X}\left( V\right) $.
Whence, $s_{1}^{X}\left( V\right) $ is a Banach space.
\end{proof}

\begin{theorem}\label{Theorem:complexity-Banachable}
Let $V$ be a separable Fr\'echet space and let $H$ be a proper
Banach subspace of $V$. Then
$\ell_{\boldsymbol\Delta}(H;V)\in\mathcal I_{\mathrm B}$.
Conversely, for every $\xi\in\mathcal I_{\mathrm B}$ and every
nontrivial separable Banach space $Z$, there is a proper Banach
subspace $H$ of $c_0(\mathbb N;Z)$ with
$\ell_{\boldsymbol\Delta}(H;c_0(\mathbb N;Z))=\xi$.
It can be chosen as the range of a compact injective linear operator
from a separable Banach space. Here $c_0(\mathbb N;Z)$ is the
supremum-norm space of norm-null sequences in $Z$.
\end{theorem}

The upper restriction follows from
Corollary~\ref{cor:banachable-limit-collapse} and
Theorem~\ref{Theorem:complexity}. The realizations are obtained in
Corollaries~\ref{cor:banach-realizations}
and~\ref{cor:banach-fixed-target}, and the proof is completed in
Subsection~\ref{sec:final-proofs}.

We begin with showing that a Banach subspace of a separable Fr\'{e}chet
space cannot have complexity class $\boldsymbol{\Pi }_{\lambda }^{0}$ for
some countable limit ordinal $\lambda $.

\begin{proposition}
\label{Proposition:Banachable-limit}Suppose that $V$ is a separable Fr\'{e}%
chet space, and $X\subseteq V$ is a Fr\'echet subspace. Suppose that $%
s_{\alpha }^{X}\left( V\right) $ is a Banach space for some positive limit ordinal $%
\alpha $. Then $X$ has Solecki rank\emph{\ less than} $\alpha $.
\end{proposition}

The corresponding canonical-model stabilization is contained in
\cite[Th\'eor\`eme~38]{saint-raymond_espaces_1976}.

\begin{proof}
Put $S=s_\alpha^X(V)$ and choose a bounded zero neighborhood $B$ in
the Banach space $S$. The topology on $S$ is the projective
intersection topology. Hence there are $\beta<\alpha$ and a zero
neighborhood $C$ in $s_\beta^X(V)$ such that $C\cap S\subseteq B$.
Because $B$ is bounded, its positive scalar dilates form a
neighborhood basis at zero in $S$. The same is therefore true of
the intersections with $S$ of the corresponding dilates of $C$.
Thus $S$ carries the subspace topology inherited from
$s_\beta^X(V)$. Completeness makes $S$ a closed linear subspace
there. Since $X\subseteq S$ is dense in $s_\beta^X(V)$, we obtain
$S=s_\beta^X(V)$.

The tower is decreasing, so $s_\beta^X(V)=s_{\beta+1}^X(V)$.
Its defining recursion then keeps all subsequent stages equal to
this group. Since the tower eventually reaches $X$, it follows
that $X=s_\beta^X(V)$, and its rank is at most $\beta<\alpha$.
\end{proof}

\begin{corollary}\label{cor:banachable-limit-collapse}
Suppose that $V$ is a separable Fr\'{e}chet space, and $X\subseteq V$ is a
Banach subspace.\ If $X$ is $\boldsymbol{\Pi }_{\lambda }^{0}$ for some
limit ordinal $\lambda <\omega _{1}$, then $X$ is $\boldsymbol{\Pi }_{\beta
}^{0}$ for some $\beta <\lambda $.
\end{corollary}

\begin{proof}
By Theorem \ref{Theorem:characterize-solecki} we have that $X=s_{\lambda }^{X}\left(
V\right) $ is a Banach space. Thus, by Proposition \ref%
{Proposition:Banachable-limit} we have that $X$ has Solecki rank $\beta $
for some $\beta <\lambda $, and hence $X$ is $\boldsymbol{\Pi }_{1+\beta
+1}^{0}$ by Theorem \ref{Theorem:characterize-solecki} again.
\end{proof}

\section{Operator ranges of every successor additive complexity}
\label{sec:additive-ranges}

We construct compact operator ranges in every countable successor
additive class $\boldsymbol{\Sigma}^0_\beta$, $\beta\geq2$, using
domains with separable biduals. The existence assertion was already
stated in \cite[Theorem~2.7]{BDP}. We give a complete proof, based on
quotient norms and saturated fibers; Remark~\ref{add:rem:BDP}
compares it with the argument in \cite{BDP}. The initial space is
$\ell_2(J)$, where $J$ is James space; higher weak-star derived ranks
provide the transfinite indices. After recalling the classical inputs
(Subsection~\ref{sec:classical-inputs}), we identify the successor
canonical models with quotient-norm completions, which gives control
of their biduals (Subsection~\ref{add:sec:quotient}), compute the
images of closed bounded sets, extend a multiplicative operator range
to an additive one, and construct the initial operators.

\begin{theorem}\label{thm:additive-ranges}
Let $\lambda$ be zero or a countable limit ordinal, let $1\leq n<\omega$,
and put
\[
 \alpha=\lambda+n,\qquad \beta=1+\lambda+n=1+\alpha.
\]
There exist a Banach space $X_\beta$ with separable bidual and a compact
injective linear operator
\[
 T_\beta:X_\beta\longrightarrow\ell_2
\]
such that
\[
 T_\beta(X_\beta)\in\Sig^0_\beta(\ell_2)\setminus\PiB^0_\beta(\ell_2).
\]
Equivalently, writing $H=T_\beta(X_\beta)$, its Borel index is
\[
 \ell_{\boldsymbol\Delta}(H;\ell_2)=
 \begin{cases}
  \lambda+\tfrac12,&n=1,\\
  \lambda+n-1+\tfrac12-\varepsilon,&n\geq2.
 \end{cases}
\]
The Solecki rank of $H$ in $\ell_2$ is exactly $\alpha$, and
\[
 H\in\Sig^0_2\bigl(s_{\lambda+n-1}^H(\ell_2)\bigr).
\]
Thus, for $n\geq2$, the additive alternative in
Theorem~\ref{Theorem:complexity}\textup{(2d)} is realized.
\end{theorem}

\begin{remark}
The classification of operator images in the additive Borel classes is
developed systematically in \cite{LOI}, which takes the construction of
this section as one of its inputs.
\end{remark}

\subsection{Notation and classical inputs}
\label{sec:classical-inputs}

We work over $\R$. The letter $B_E$ denotes the closed unit ball of a
Banach space $E$. For a Fr\'echet space $F$, its dual $F^*$
and bidual $F^{**}$ carry the strong dual topologies unless a weak-star
topology is specified. Original spaces are identified with their canonical
copies in their biduals. The canonical map of a Fr\'echet space
into its strong bidual is a topological embedding: the polars of
zero neighborhoods are strongly bounded in the dual and recover
the original continuous seminorms. See also \cite[p.~38]{BDP}.
For historical comparison with
\cite{BDP,saint-raymond_espaces_1976}, the notation for multiplicative and additive classes
translates as
\[
 \mathcal{M}_\xi(Z)=\PiB^0_{1+\xi}(Z),\qquad
 \mathcal{A}_\xi(Z)=\Sig^0_{1+\xi}(Z),\qquad
 \mathcal{M}^2_\xi(Z)=D(\PiB^0_{1+\xi})(Z)
 \qquad(\xi<\omega_1).
\]
The fractional indices below refer to Table~\ref{tab:fractional-classes}.
The order of ordinal addition matters:
$1+\xi=\xi$ when $\xi$ is an infinite limit ordinal.

For a linear subspace $L\subseteq E^*$, with $E$ separable Banach, write
\[
 L_{(0)}=L,\qquad
 L_{(1)}=\bigcup_{m\geq1}\overline{L\cap mB_{E^*}}^{\,w^*},\qquad
 L_{(\xi)}=\bigcup_{\zeta<\xi}(L_{(\zeta)})_{(1)}
 \quad(\xi>0).
\]
The first derived set is the set of weak-star limits of sequences from
$L$. Norm boundedness of such sequences follows from uniform boundedness;
conversely bounded weak-star closures are sequential because $E$ is
separable. At limit ordinals the displayed definition agrees with
taking the union of the preceding stages.

We use the following classical facts.
For a continuous injective linear map $A:E\to F$ between separable
Fr\'echet spaces, Theorems~\ref{thm:saint-raymond-rank}
and~\ref{Theorem:characterize-solecki} give
\begin{equation}\label{add:eq:SRcriterion}
 \begin{split}
 A(E)\in\PiB^0_{1+\xi+1}(F)
 &\ \Longleftrightarrow\
 A^{-1}:A(E)\to E\text{ is Borel of class at most }\xi.
 \end{split}
\end{equation}
Borel class at most $\xi$ means that inverse images of closed sets
are in $\PiB^0_{1+\xi}$, relative to the domain. If $\xi$ is zero or a limit
ordinal, these conditions also imply $A(E)\in\PiB^0_{1+\xi}(F)$.
When $E$ is Banach, these conditions are also equivalent to
\[
 (A^*F^*)_{(\xi)}=E^*,
\]
with the derived sets defined above.
If $E$ is finite-dimensional, $A(E)$ is closed in $F$ and $A^{-1}$
is continuous. Every $e^*\in E^*$ gives a continuous functional
$e^*\circ A^{-1}$ on $A(E)$, which extends to $F$ by the Hahn--Banach
theorem. Thus $A^*(F^*)=E^*$, and every derived set is $E^*$.
For an infinite-dimensional Banach domain, injectivity forces $F$ to
be infinite-dimensional as well, and the adjoint criterion is Saint
Raymond's theorem, in the form \cite[Theorem~1.1, p.~39]{BDP}.

For the classical formulations of the range--inverse equivalence, see
\cite[Corollaire~45 and Th\'eor\`eme~54, pp.~242, 249]{saint-raymond_espaces_1976}.
Saint Raymond's degree is the least $\xi$ at which his canonical
model is $E$ (Corollaire~45). Stabilization and
$v_{0,\xi}(N_\xi)=s_\xi^{A(E)}(F)$ give the required bound.
The multiplicative condition is stated in \cite{BDP} for the absolute
Borel class of $A(E)$; for the classes occurring here this agrees with
membership in $\PiB^0_{1+\xi+1}(F)$, by Lavrentiev's theorem.
We use the adjoint formulation only for a Banach domain; the
range--inverse equivalence in \eqref{add:eq:SRcriterion} also applies
to Fr\'echet domains.

If $E$ is Banach and the inverse of $A$ is Borel of class at most $1$,
then
\begin{equation}\label{add:eq:closed-sum}
 E+\ker A^{**}\text{ is norm closed in }E^{**}.
\end{equation}
This is the classical closed-sum characterization recorded in
\cite[Remark~1.6(4), p.~42]{BDP}, which is stated there for a
separable Banach domain and a Fr\'echet target.
Here is a proof. Put $M=A^*(F^*)$ and $K=\ker A^{**}=M^\perp$.
The criterion gives $M_{(1)}=E^*$. By Baire category, one of the
absolutely convex weak-star compact sets
$\overline{M\cap mB_{E^*}}^{\,w^*}$ contains a norm neighborhood
of zero in $E^*$ (take its difference with itself if necessary).
Taking polars, there is $c>0$ such that
\[
 c\|x\|_E\leq\sup_{f\in M\cap B_{E^*}}|f(x)|
 \leq\inf_{k\in K}\|x+k\|_{E^{**}}\qquad(x\in E).
\]
The quotient map $E\to E^{**}/K$ is thus an isomorphic embedding
of a Banach space, with closed image. Its inverse image under
$E^{**}\to E^{**}/K$ is $E+K$, proving
\eqref{add:eq:closed-sum}.

For a continuous injection $A:E\to F$, with $E$ separable Banach
and $F$ separable Fr\'echet, put $L=A^*(F^*)$.
Saint Raymond's canonical models give, for every countable $\xi$,
a separable Fr\'echet space $N_\xi$ and continuous injections
\[
 r_\xi:E\longrightarrow N_\xi,\qquad
 v_\xi:N_\xi\longrightarrow F,\qquad
 A=v_\xi r_\xi,
\]
such that $r_\xi(E)$ is dense in $N_\xi$ and
\begin{equation}\label{add:eq:models}
 r_\xi^*(N_\xi^*)=L_{(\xi)},\qquad
 v_\xi^{-1}\text{ is Borel of class at most }\xi.
\end{equation}
When $E$ is Banach and $\xi$ is a positive successor, $N_\xi$ is
Banach. At $\xi=0$, $N_0$ is the closure of $A(E)$ in $F$.
These statements are \cite[Th\'eor\`emes~31, 36, and~38]{saint-raymond_espaces_1976},
on pp.~233--237.

Finally, Ostrovskii's theorem says that, for a separable
nonquasireflexive Banach space $E$ and any countable $\eta$, there is
a linear subspace $\Gamma\subseteq E^*$ with
\begin{equation}\label{add:eq:ostrov}
 \Gamma_{(\eta)}\ne\Gamma_{(\eta+1)}=E^*;
\end{equation}
see \cite[unnumbered theorem, English translation p.~1]{O}. Such a subspace is necessarily total.

\subsection{Successor canonical models as quotient-norm completions}
\label{add:sec:quotient}

We compare the canonical successor models with completions for a
norm defined by a total dual subspace. The quotient representation
of this completion gives the separable-bidual control needed for
the extension.

\begin{lemma}\label{add:lem:quotient}
Let $E$ be separable Banach and let $M\subseteq E^*$ be a total linear
subspace, not necessarily norm closed. Define
\[
 p_M(x)=\sup\{\abs{f(x)}:f\in M,\ \norm{f}\leq1\},
 \qquad K=M^\perp\subseteq E^{**}.
\]
Let $E_M$ be the completion of $(E,p_M)$, with canonical dense
injection $j_M:E\to E_M$. If $q:E^{**}\to E^{**}/K$ is the quotient map,
then
\begin{equation}\label{add:eq:quotient-model}
 E_M\cong\overline{q(E)}\subseteq E^{**}/K,\qquad
 j_M^*(E_M^*)=M_{(1)}.
\end{equation}
In particular, $E^{****}$ separable implies $E_M^{**}$ separable.
\end{lemma}

\begin{proof}
Restriction to the norm closure of $M$ is a surjective norm-one
quotient map
\[
 E^{**}\longrightarrow(\overline M^{\,\norm{\cdot}})^*
\]
by Hahn--Banach. Its kernel is $K$, and its value on $x\in E$ has
norm $p_M(x)$. Consequently,
\[
 p_M(x)=\norm{x+K}_{E^{**}/K}.
\]
This proves the first assertion of \eqref{add:eq:quotient-model}.

The bipolar theorem gives the precise unit-ball identity
\[
 j_M^*(B_{E_M^*})=\overline{M\cap B_{E^*}}^{\,w^*}.
\]
Taking positive integer dilates gives the second assertion of
\eqref{add:eq:quotient-model}.

Put $Q=E^{**}/K$. The second adjoint of the quotient map is a
surjection $E^{****}\to Q^{**}$. The closed-subspace inclusion
$E_M\subseteq Q$ induces an isometric embedding $E_M^{**}\to Q^{**}$.
Thus separability of $E^{****}$ implies separability of $E_M^{**}$.
\end{proof}

\begin{lemma}\label{add:lem:uniqueness}
Suppose that $j_i:E\to F_i$, $i=1,2$, are continuous dense injections
into Banach spaces and
\[
 j_1^*(F_1^*)=j_2^*(F_2^*)\subseteq E^*.
\]
Then $\norm{j_1x}$ and $\norm{j_2x}$ are equivalent norms on $E$.
In particular, $F_1$ and $F_2$ are isomorphic through the map
identifying their dense copies of $E$.
\end{lemma}

\begin{proof}
The bijection
\[
 U=(j_2^*)^{-1}j_1^*:F_1^*\longrightarrow F_2^*
\]
has closed graph: norm convergence in the two dual spaces gives norm
convergence of both adjoint images in $E^*$. The closed graph theorem,
also applied to $U^{-1}$, makes $U$ a Banach-space isomorphism.
Taking suprema over the dual unit balls yields
\[
 \norm{j_1x}\leq\norm{U}\,\norm{j_2x},\qquad
 \norm{j_2x}\leq\norm{U^{-1}}\,\norm{j_1x}.
\]
The equivalent norms have isomorphic completions.
\end{proof}

\begin{corollary}\label{add:cor:successor-bidual}
If $E^{****}$ is separable, every positive successor canonical model
$N_\xi$ of an injective operator with domain $E$ has separable bidual.
\end{corollary}

\begin{proof}
Write $\xi=\zeta+1$, and take $M=L_{(\zeta)}$, where $L=A^*(F^*)$.
The model $N_\xi$ is Banach, and \eqref{add:eq:models} gives
\[
 r_\xi^*(N_\xi^*)=L_{(\xi)}=M_{(1)}=j_M^*(E_M^*).
\]
The subspace $M$ is total because it contains the total subspace $L$.
Lemma~\ref{add:lem:uniqueness} identifies $N_\xi$ with $E_M$.
Now apply Lemma~\ref{add:lem:quotient}.
\end{proof}

Corollary~\ref{add:cor:successor-bidual} supplies the separable
bidual of the canonical successor model $N_\xi$ under the stated
$E^{****}$ hypothesis. The separable bidual of the extended domain is
justified in the proof of Proposition~\ref{add:prop:extension} below.

\subsection{Images of closed bounded sets}

The extension argument requires $G_\delta$ images of closed bounded
sets. Saturating such a set by the bidual kernel gives the required
image formula.

\begin{lemma}\label{add:lem:gdelta}
Let $E$ be a Banach space with separable bidual, let $F$ be Fr\'echet,
and let $R:E\to F$ be a continuous injection. If
\[
 W=E+\ker R^{**}
\]
is norm closed in $E^{**}$, then $R(C)\in\PiB^0_2(F)$
for every norm-closed bounded subset $C$ of $E$.
\end{lemma}

\begin{proof}
Put $K=\ker R^{**}$. Since $E\cap K=\{0\}$, the decomposition
$W=E\oplus K$ gives a bounded projection $P:W\to E$.
The set
\[
 A=P^{-1}(C)=C+K
\]
is norm closed in $W$ and hence in $E^{**}$. Choose $r>0$ with
$C\subseteq rB_E$ and put $D=rB_{E^{**}}$.

The complement $E^{**}\setminus A$, being norm open in a separable
Banach space, is a countable union of closed norm balls.
Their intersections with $D$ are weak-star compact. The map
$R^{**}$ is weak-star continuous, so $R^{**}(D)$ is weak-star compact,
and $R^{**}(D\setminus A)$ is a countable union of weak-star compact
sets in $F^{**}$.

The set $A=C+K$ is saturated for the fibers of $R^{**}$.
Since $C\subseteq D$, this proves the exact identity
\begin{equation}\label{add:eq:correct-ball}
 R(C)=R^{**}(D)\setminus R^{**}(D\setminus A).
\end{equation}
Indeed, a fiber has a representative in $C$ if and only if all its
representatives belong to $A$; in that case the fiber meets $D$.
Otherwise every representative of that fiber in $D$ lies in
$D\setminus A$.
Weak-star compact sets in $F^{**}$ are closed for its Hausdorff
weak-star topology. Pulling \eqref{add:eq:correct-ball} back along the
continuous canonical map $F\to(F^{**},w^*)$ writes $R(C)$ as a closed
set minus a $\Sig^0_2$ set in $F$. It therefore belongs to $\PiB^0_2(F)$.
\end{proof}

\begin{lemma}\label{add:lem:bounded-images}
Let $E$ be Banach with separable bidual, let $F$ be separable Fr\'echet,
and let $S:E\to F$ be a continuous injection. If
\[
 S(E)\in\PiB^0_{1+\theta+1}(F),\qquad 1\leq\theta<\omega_1,
\]
then $S(C)\in\PiB^0_{1+\theta}(F)$ for every norm-closed bounded $C\subseteq E$.
\end{lemma}

\begin{proof}
If $\theta$ is a nonzero limit ordinal, \eqref{add:eq:SRcriterion} gives
$S(E)\in\PiB^0_{1+\theta}(F)$ and $S^{-1}$ of Borel class at most $\theta$.
Thus $S(C)$ is relatively $\PiB^0_{1+\theta}$ in $S(E)$, and consequently
belongs to $\PiB^0_{1+\theta}(F)$.

Suppose $\theta=\delta+1$. Factor $S=v_\delta r_\delta$ through its
canonical model. Writing $L=S^*(F^*)$, we have
\[
 (r_\delta^*N_\delta^*)_{(1)}
 =(L_{(\delta)})_{(1)}
 =L_{(\theta)}=E^*.
\]
Therefore $r_\delta^{-1}$ is Borel of class at most $1$, by
\eqref{add:eq:SRcriterion}. Formula~\eqref{add:eq:closed-sum} and
Lemma~\ref{add:lem:gdelta} imply
\[
 r_\delta(C)\in\PiB^0_2(N_\delta).
\]
The inverse of $v_\delta$ has Borel class at most $\delta$.
Write $r_\delta(C)$ as a countable intersection of open subsets of
$N_\delta$. Their images are relatively $\Sig^0_{1+\delta}$ in
$v_\delta(N_\delta)$, so $S(C)$ is relatively
$\PiB^0_{1+\delta+1}=\PiB^0_{1+\theta}$ there.
By \eqref{add:eq:SRcriterion}, the range $v_\delta(N_\delta)$ itself belongs
to $\PiB^0_{1+\delta+1}(F)$. Hence $S(C)\in\PiB^0_{1+\theta}(F)$.
This includes $\delta=0$, when $v_0$ is the inclusion of the closed
range closure into $F$.
\end{proof}

\subsection{The additive extension}

We now enlarge a multiplicative operator range to obtain the
preceding additive class. The closed-sum criterion and bounded-set
estimates control the enlarged range and its intrinsic Banach topology.

\begin{proposition}\label{add:prop:extension}
Let $\eta\geq1$ be countable. Suppose that $E$ is Banach with separable
fourth dual, $Y$ is separable Banach, and $T:E\to Y$ is an injective
compact operator with
\[
 T(E)\in\PiB^0_{1+\eta+2}(Y)\setminus\PiB^0_{1+\eta+1}(Y).
\]
Then there is a Banach space $\widetilde E$ with separable bidual,
containing a closed isomorphic copy $i(E)$ of $E$, and an injective
compact operator $\widetilde T:\widetilde E\to Y$ such that
\[
 \widetilde T i=T,\qquad
 \widetilde T(\widetilde E)\in
 \Sig^0_{1+\eta+1}(Y)\setminus\PiB^0_{1+\eta+1}(Y).
\]
\end{proposition}

\begin{proof}
Use the factorization
\[
 T=v_\eta r_\eta,\qquad r=r_\eta:E\to N_\eta.
\]
The criterion \eqref{add:eq:SRcriterion} gives
\[
 (r^*N_\eta^*)_{(1)}
 =(T^*Y^*)_{(\eta+1)}=E^*.
\]
Thus $r^{-1}$ is Borel of class at most $1$, and
$E+K$ is norm closed in $E^{**}$, where $K=\ker r^{**}$.

The canonical image of a complete Fr\'echet space in its strong
bidual is closed, and carries its original topology. Consequently
\[
 Z=(r^{**})^{-1}(N_\eta)\subseteq E^{**}
\]
is a norm-closed subspace. Define
\begin{equation}\label{add:eq:extension}
 \widetilde E=Z/K,\qquad
 \widetilde T(w+K)=T^{**}w\quad(w\in Z).
\end{equation}
Here $K\subseteq\ker T^{**}$ because $T=v_\eta r$, so the definition
is independent of the representative. Compactness of $T$ implies
that $T^{**}$ takes values in the canonical copy of $Y$ and is
compact as an operator $E^{**}\to Y$. It follows, by lifting elements
of a quotient unit ball to representatives of norm at most $2$,
that $\widetilde T$ is compact.

The map $r^{**}$ induces a continuous injection
\[
 \widetilde r:\widetilde E\to N_\eta,\qquad
 \widetilde r(w+K)=r^{**}w,
\]
and $\widetilde T=v_\eta\widetilde r$.
Thus $\widetilde T$ is injective.
The map $i:E\to\widetilde E$, $i(x)=x+K$, is an isomorphic embedding
with closed range: its range is $(E+K)/K$, which is closed because
$E+K$ is norm closed, and $E\cap K=\{0\}$.
Moreover, $\widetilde T i=T$.
Since $Z$ is a closed subspace of $E^{**}$ and $\widetilde E$ is a
quotient of $Z$, its bidual is a quotient of $Z^{**}$, which embeds
into $E^{****}$. Therefore $\widetilde E^{**}$ is separable.

To prove the additive upper bound, put
\[
 F_m=r^{**}(mB_{E^{**}})\cap N_\eta\qquad(m\geq1).
\]
Each $F_m$ is closed in $N_\eta$: it is the inverse image of a
weak-star compact subset of $N_\eta^{**}$ under the canonical map.
It is also bounded in $N_\eta$. Indeed, the continuous map
$r^{**}:E^{**}\to N_\eta^{**}$ sends $mB_{E^{**}}$ to a bounded set
in the strong bidual, and the canonical copy of $N_\eta$ carries the
induced topology.
We have
\begin{equation}\label{add:eq:union}
 \begin{split}
 \widetilde r(\widetilde E)
 &=r^{**}(E^{**})\cap N_\eta=\bigcup_{m\geq1}F_m,\\
 \widetilde T(\widetilde E)&=\bigcup_{m\geq1}v_\eta(F_m).
 \end{split}
\end{equation}

If $\eta$ is a nonzero limit ordinal, \eqref{add:eq:SRcriterion} gives
$v_\eta(N_\eta)\in\PiB^0_{1+\eta}(Y)$. Since $v_\eta^{-1}$ has Borel class
at most $\eta$ and $F_m$ is closed, we obtain
$v_\eta(F_m)\in\PiB^0_{1+\eta}(Y)$.

If $\eta$ is a successor, $N_\eta$ is Banach and its bidual is separable
by Corollary~\ref{add:cor:successor-bidual}. Moreover,
$v_\eta(N_\eta)\in\PiB^0_{1+\eta+1}(Y)$ by \eqref{add:eq:SRcriterion}.
Apply Lemma~\ref{add:lem:bounded-images} to $v_\eta$ and the closed bounded
set $F_m$ to obtain the same conclusion
$v_\eta(F_m)\in\PiB^0_{1+\eta}(Y)$.
In both cases, \eqref{add:eq:union} proves
\[
 \widetilde T(\widetilde E)\in\Sig^0_{1+\eta+1}(Y).
\]

Finally, suppose that $\widetilde T(\widetilde E)$ belonged to
$\PiB^0_{1+\eta+1}(Y)$. Its inverse would be Borel of class at most $\eta$
by \eqref{add:eq:SRcriterion}. Since $i(E)$ is closed in $\widetilde E$,
its image $T(E)=\widetilde T(i(E))$ would then be relatively
$\PiB^0_{1+\eta}$ in $\widetilde T(\widetilde E)$.
Intersecting a $\PiB^0_{1+\eta}$ representative in $Y$ with a
$\PiB^0_{1+\eta+1}$ range gives $T(E)\in\PiB^0_{1+\eta+1}(Y)$, contradicting
the starting hypothesis.
\end{proof}

\subsection{Constructing the input operators}

To apply the extension, we need compact injective operators of
prescribed multiplicative rank whose domains have separable fourth
dual. We obtain them from weak-star derived subspaces of $\ell_2(J)^*$.

\begin{lemma}\label{add:lem:seed}
Let $J$ be separable James space, and put
\[
 E=\bigl(\bigoplus_{n\geq1}J\bigr)_{\ell_2}.
\]
For every countable $\eta\geq1$, there is an injective nuclear operator
$T:E\to\ell_2$ with
\[
 T(E)\in\PiB^0_{1+\eta+2}(\ell_2)\setminus\PiB^0_{1+\eta+1}(\ell_2).
\]
The space $E^{****}$ is separable.
\end{lemma}

This is a variant of the compact-range construction in
\cite[Theorem~2.6]{BDP}; we include the dual estimates needed here.

\begin{proof}
James space has separable dual and is quasi-reflexive of order one
\cite[Example, item~(iii), p.~523]{jamesBasesReflexivityBanach1950}. Since its bidual is a finite-dimensional extension
of $J$, repeated dualization shows that its duals through the fourth are
separable. Hence the fourth dual of $\ell_2(J)$ is separable. The space
$E$ is nonquasireflexive because
$E^{**}/E$ contains the infinitely many independent bidual defects of
its summands.

Choose $\Gamma\subseteq E^*$ satisfying \eqref{add:eq:ostrov}, and put
$M=\overline\Gamma^{\,\norm{\cdot}}$.
Since $E^*$ is separable, choose a sequence $(f_k)$ norm dense in
$B_M$. Define
\begin{equation}\label{add:eq:seed}
 Tx=(2^{-k}f_k(x))_{k\geq1}\in\ell_2.
\end{equation}
The subspace $M$ is total, so $T$ is injective.
The formula is a nuclear representation since
$\sum_k2^{-k}\norm{f_k}<\infty$.
For $L=T^*(\ell_2)$ we have
\[
 L\subseteq M,\qquad \overline L^{\,\norm{\cdot}}=M,
\]
because $T^*e_k=2^{-k}f_k$.

We record why taking norm closures causes no difficulty.
If $A$ is a linear subspace of $E^*$, then
\[
 A_{(1)}=(\overline A^{\,\norm{\cdot}})_{(1)}.
\]
For the nontrivial inclusion, approximate the $j$th member of any
weak-star convergent sequence in $\overline A$ by a member of $A$
within norm distance $1/j$. The approximating sequence is bounded
and has the same weak-star limit. Applying this identity to
$L\subseteq M$ and to $\Gamma\subseteq M$ gives
\[
 L_{(1)}=M_{(1)}=\Gamma_{(1)}.
\]
Transfinite induction therefore gives
$L_{(\xi)}=\Gamma_{(\xi)}$ for every $\xi\geq1$.
In particular,
\[
 L_{(\eta)}\ne L_{(\eta+1)}=E^*.
\]
The two applications of \eqref{add:eq:SRcriterion}, at $\eta+1$ and
at $\eta$, give the asserted multiplicative bounds.
\end{proof}

\begin{proof}[Proof of Theorem~\ref{thm:additive-ranges}]
Put $\eta=\lambda+(n-1)$, so $\alpha=\eta+1$ and
$\beta=1+\eta+1$.
First suppose $\beta=2$. The compact diagonal operator
\[
 T_2:\ell_2\to\ell_2,\qquad T_2(x)=(2^{-k}x_k)_k,
\]
has a $\Sig^0_2$ range: the image of every closed ball is compact.
Its range is dense and proper, so it is not a $\PiB^0_2$ subgroup.
Indeed, every $\PiB^0_2$ subgroup of a Polish group is closed.
Thus the complexity class of $T_2(\ell_2)$ in $\ell_2$ is $\Sig^0_2$.

For all other indices in the theorem, $\eta\geq1$. Apply Lemma~\ref{add:lem:seed} and
Proposition~\ref{add:prop:extension}. The resulting operator has range in
\[
 \Sig^0_{1+\eta+1}\setminus\PiB^0_{1+\eta+1}
 =\Sig^0_\beta\setminus\PiB^0_\beta,
\]
as required.

The range $H$ is Polish by transporting the Banach topology of
the domain; the Lusin--Souslin theorem identifies its Borel structure
with the inherited one.
Since $H\in\Sig^0_\beta\subseteq\PiB^0_{\beta+1}$, minimality of
the Solecki subgroups Theorem~\ref{Theorem:characterize-solecki} gives
$s_\alpha^H(\ell_2)=H$.
If its rank were $\rho<\alpha=\eta+1$, then
\[
 H\in\PiB^0_{1+\rho+1}(\ell_2)
 \subseteq\PiB^0_{1+\eta+1}(\ell_2)=\PiB^0_\beta(\ell_2),
\]
a contradiction. Thus its rank is exactly $\alpha$.
Finally, Corollary~\ref{Corollary:relativize}, applied at stage $\eta$ to
the additive upper bound, gives
\[
 H\in\Sig^0_2(s_\eta^H(\ell_2)).
\]
For $\eta=0$ this also follows immediately by restricting to the
closed subgroup $\overline H$.
\end{proof}

\begin{corollary}[A successor additive example]
\label{cor:additive-counterexample}
There is a Banach subspace $H$ of $\ell_2$ with
\[
 H\in\boldsymbol{\Sigma}^0_3(\ell_2)
       \setminus\boldsymbol{\Pi}^0_3(\ell_2).
\]
It has Solecki rank $2$, belongs to
$\boldsymbol{\Sigma}^0_2(s_1^H(\ell_2))$, and is not an intersection
of a $\boldsymbol{\Sigma}^0_2$ set and a $\boldsymbol{\Pi}^0_2$ set.
\end{corollary}
\begin{proof}
Apply Theorem~\ref{thm:additive-ranges} with $\lambda=0$ and $n=2$.
An intersection of a $\boldsymbol{\Sigma}^0_2$ set and a
$\boldsymbol{\Pi}^0_2$ set belongs to $\boldsymbol{\Delta}^0_3$.
\end{proof}

Replacing the target by the closed subspace
$\overline{T_\beta(X_\beta)}$ of $\ell_2$ makes the range in
Theorem~\ref{thm:additive-ranges} dense without changing its Borel
complexity.

\begin{remark}[The arguments in \cite{BDP}]\label{add:rem:BDP}
The proof of \cite[Theorem~2.7]{BDP} combines the compact-range
construction of \cite[Theorem~2.6]{BDP} with the extension theorem
\cite[Theorem~1.10]{BDP}. Three steps of these proofs require a
different argument, which the present proof supplies.
\begin{enumerate}
\item Case~2 of the proof of \cite[Theorem~1.10, p.~44]{BDP}, which
gives the classes $\Sig^0_{1+\lambda+n+1}$ with $n\geq1$, asserts
that the derived set $L_{(\xi)}$ is norm closed. A norm-closed first
derived set of a total subspace $M\subseteq E^*$ must equal $E^*$.
To see this, put $D=M_{(1)}$. If $D$ is norm closed, Baire category
in $D$, applied to
$D=\bigcup_m\overline{M\cap mB_{E^*}}^{\,w^*}$, shows that $B_D$ is
contained in a multiple of one of these weak-star compact sets. Its
weak-star closure therefore lies in $D$, and also in $B_{E^*}$, so
$B_D$ is weak-star closed. The Krein--\v{S}mulian theorem
\cite[Corollary~6.8]{osborne_locally_2014} makes $D$ weak-star
closed, and totality of $M$ forces $D=E^*$. In the setting of
Case~2, however, $L_{(\xi)}\ne E^*$ by the assumed inverse Borel rank
and \eqref{add:eq:SRcriterion}. The step serves to show that the
canonical model $N_\xi$ has separable bidual;
Corollary~\ref{add:cor:successor-bidual} proves this when $E^{****}$
is separable. That hypothesis holds in the application to
\cite[Theorem~2.7]{BDP}, by Lemma~\ref{add:lem:seed}, so
Theorem~\ref{thm:additive-ranges} recovers that existence statement
in full. The present argument does not prove
\cite[Theorem~1.10(2)]{BDP} at a double successor index when only
$E^{**}$ is assumed separable.
\item Claim~D in the same proof \cite[p.~43]{BDP} uses an identity
for images of balls. With $R$ as in Lemma~\ref{add:lem:gdelta}, $C$
the unit ball of $E$, $D$ that of $E^{**}$, and $K=\ker R^{**}$, it
reads
\[
 R^{**}(D)\setminus R(C)=R^{**}\bigl(D\setminus(E+K)\bigr).
\]
The inclusion $\supseteq$ holds, but $\subseteq$ fails whenever some
$x\in E$ with $\norm{x}>1$ satisfies $\norm{x+k}_{E^{**}}\leq1$ for
some $k\in K$: the point $R(x)$ then lies on the left, while every
representative of its fiber lies in $E+K$. Equivalently, the
projection onto $E$ along $K$ need not preserve the norm ball. The
conclusion of Claim~D, that the image of a closed bounded set is a
$G_\delta$ set, is Lemma~\ref{add:lem:gdelta}, whose formula
\eqref{add:eq:correct-ball} saturates the given set to $C+K$.
\item In Claim~A of the proof of \cite[Theorem~2.6, p.~46]{BDP}, the
bipolar of $M\cap B_{E^*}$ is identified with $M_{(1)}\cap B_{E^*}$.
The bipolar theorem gives the weak-star closure of $M\cap B_{E^*}$,
which may be smaller. The conclusion used there survives, because it
only requires the union of the dilates: this is the second identity
in \eqref{add:eq:quotient-model}.
\end{enumerate}
\end{remark}

\section{Limit-successor difference classes and the complete Banach list}
\label{sec:banach-realizations}
\label{sec:limit-difference}
This section completes the Banach list of
Theorem~\ref{Corollary:complexity-Banachable}. Its main result,
Theorem~\ref{thm:limit-difference}, constructs for every positive
countable limit $\lambda$ a compact injective operator whose range has
complexity class $D(\PiB^0_{\lambda+1})$; this answers
\cite[Question~2.11]{BDP}. The idea is to assemble, in an $\ell_1$
sum, operators whose Solecki ranks are cofinal in $\lambda$, and to cut
the resulting range by one hyperplane. The hyperplane does not change
the Solecki stages up to stage $\lambda$ (Lemma~\ref{lim:lem:towers}), but
makes the range a relative difference set that is not relatively
$\Sig^0_2$ in the $\lambda$-th stage
(Proposition~\ref{lim:prop:terminal}). Subsection~\ref{lim:sec:gauges}
records the successor gauges and the closed-ball test for relative
$\Sig^0_2$ membership; Subsections~\ref{lim:sec:assembly}
to~\ref{lim:sec:terminal} carry out the construction; and
Subsection~\ref{sec:consequences} combines it with the classical
realizations and transfers all classes of the Banach list into an
arbitrary separable infinite-dimensional Banach target.

We work over $\R$. From Subsection~\ref{lim:sec:assembly} on, $G$ is
the separable Hilbert space and $A$ the Banach subspace of $G$ defined
in \eqref{lim:eq:assembly}, and $H_\gamma=s_\gamma^A(G)$ denotes the
Solecki subspaces of $A$ with their canonical topologies.

\begin{theorem}\label{thm:limit-difference}
For every nonzero countable limit ordinal $\lambda$ there exist a separable
Banach space $X_\lambda$ and a compact injective linear operator
\[
 T_\lambda:X_\lambda\longrightarrow\ell_2
\]
whose range $N$ has Solecki rank exactly $\lambda+1$ and satisfies
\[
 N\in D(\PiB^0_2)\bigl(s_\lambda^N(\ell_2)\bigr)
       \setminus\Sig^0_2\bigl(s_\lambda^N(\ell_2)\bigr).
\]
Consequently,
$\ell_{\boldsymbol\Delta}(N;\ell_2)=\lambda+1/2+\varepsilon$;
its exact nonselfdual complexity class is $D(\PiB^0_{\lambda+1})$.
\end{theorem}

The final implication is Theorem~\ref{Theorem:complexity}\textup{(2b)}, with $n=1$
and $1+\lambda=\lambda$. The only realization input needed is the
existence of compact injective operators between separable Banach
spaces whose Solecki ranks are cofinal below $\lambda$.
Theorem~\ref{thm:additive-ranges} supplies these.

The independent basic Solecki properties used below are:
$A$ is dense in each $H_\gamma$; inclusions of later stages
into earlier ones are continuous; limit stages carry the projective
intersection topology; and the construction commutes with finite
products of pairs, including the Polish group topologies on the
corresponding Solecki subgroups
(Section~\ref{Section:solecki} and Lemma~\ref{Lemma:characterize-solecki-product}).
For linear subspaces the first-stage formula is
\begin{equation}\label{lim:eq:solecki}
 s_1^A(H)=\bigcap_{\varepsilon>0}
            \bigl(A+\cl{\varepsilon B_A}{H}\bigr),
\end{equation}
where $B_A$ is the closed unit ball of $A$ for its Banach norm.

\subsection{Successor gauges and the closed-ball test}\label{lim:sec:gauges}
We prove the successor-gauge calculation and closed-ball criterion
for Banach pairs with the stated hypotheses on the containing
Fr\'echet space $H$.

\begin{lemma}[Successor gauges]\label{lim:lem:gauge}
Let $A$ be a separable Banach space continuously and injectively included
in a separable Fr\'echet space $H$. Put
\[
 C=\cl{B_A}{H},\qquad
 V=\bigcup_{m\ge1}mC,\qquad
 p(x)=\inf\{r>0:x\in rC\}.
\]
Then $V$ is Banach for $p$, the inclusion $V\to H$ is continuous, and
\[
 s_1^A(H)=\cl{A}{(V,p)}
\]
with the canonical topology. In particular, this stage is a Banach space.
\end{lemma}
\begin{proof}
Every continuous seminorm $q$ on $H$ satisfies
$q(a)\le c_q\norm{a}_A$ on $A$, so $q(x)\le c_q r$ on $rC$.
These inequalities make $p$ a norm: if $p(x)=0$, every
continuous seminorm of the Hausdorff space $H$ vanishes at $x$, so $x=0$. The set $C$ is closed, balanced,
and convex. Its dilates are precisely the closed balls of $p$:
if $p(x)\le r$, then $x\in(r+\varepsilon)C$ for every $\varepsilon>0$;
scaling and closedness give $x\in rC$.

A $p$-Cauchy sequence is Cauchy in every continuous seminorm on $H$
and has a limit there. Closedness of the dilates of $C$ shows that
this limit lies in $V$ and that convergence holds in $p$.
Thus $V$ is Banach. Formula \eqref{lim:eq:solecki} identifies $s_1^A(H)$
with the $p$-closure of $A$. This closure is separable Banach, its
inclusion into $H$ is continuous, and its topology is the canonical
Polish topology by the neighborhood description of the first
stage, or by uniqueness of that topology.
\end{proof}

\begin{lemma}[Closed-ball test]\label{lim:lem:ball}
Let $N$ be a separable Banach space continuously included in a separable
Fr\'echet space $H$. Then
\[
 N\in\Sig^0_2(H)\quad\Longleftrightarrow\quad
 \cl{B_N}{H}\subseteq N.
\]
\end{lemma}
\begin{proof}
If the containment holds, then $N=\bigcup_{m\ge1}m\cl{B_N}{H}$.
Conversely, if $N=\bigcup_j F_j$ with each $F_j$ closed in $H$,
Baire category in $N$ gives $a\in N$
and $r>0$ with $a+rB_N\subseteq F_j$ for some $j$.
Taking closure in $H$, translating, and scaling gives the containment.
\end{proof}

\subsection{Small vectors in a cofinal assembly}\label{lim:sec:assembly}
We assemble operators of cofinal Solecki ranks into one operator
whose range has all its stages below $\lambda$ under control.
Choose a strictly increasing cofinal sequence of successor ordinals
\[
 \gamma_1<\gamma_2<\cdots<\lambda
\]
and successor ordinals $r_n$ with $\gamma_n<r_n<\lambda$.
By the realization input choose compact injective operators
\[
 R_n:E_n\longrightarrow Y_n,\qquad \norm{R_n}\le1,
\]
where $E_n$ is separable Banach, $Y_n$ is separable infinite-dimensional
Hilbert, and $R_n(E_n)$ has Solecki rank $r_n$ in $Y_n$.
We may take $R_n(E_n)$ dense by replacing the target by its closed
linear span and identifying that span with a Hilbert space.

Put
\begin{equation}\label{lim:eq:assembly}
 E=\bigl(\bigoplus_{n\ge1}E_n\bigr)_{\ell_1},\qquad
 G=\bigl(\bigoplus_{n\ge1}Y_n\bigr)_{\ell_2},\qquad
 J(x_n)=\bigl(2^{-n}R_nx_n\bigr)_n.
\end{equation}
The map $J$ is injective and compact: its finite-coordinate truncations
are compact and converge to $J$ in operator norm.
Its range $A=J(E)$ is dense in $G$, since it contains all finite-coordinate
vectors with coordinates in the dense block ranges.
Identify $A$ with $E$ through $J$, and write $H_\gamma=s_\gamma^A(G)$.

\begin{lemma}[Small vectors]\label{lim:lem:small}
There exist vectors $v_n\in A$, supported in coordinate $n$, such that
\[
 \norm{v_n}_A=1,\qquad v_n\longrightarrow0
 \quad\text{in every }H_\gamma,\ \gamma<\lambda.
\]
The same convergence holds in $H_\lambda$.
\end{lemma}
\begin{proof}
Each $H_{\gamma_j}$ is a Banach space by Lemma~\ref{lim:lem:gauge}, since
$\gamma_j$ is a successor. Fix a compatible Banach norm on it.
For each single coordinate $n$, the decomposition
\[
 (G,A)\cong
 (Y_n,2^{-n}R_nE_n)\times(G_{\ne n},A_{\ne n})
\]
is a finite product decomposition of pairs.
Thus, for every $\gamma$, the coordinate maps between $H_\gamma$
and $s_\gamma^{2^{-n}R_nE_n}(Y_n)$ are continuous.

The block subspace $R_nE_n$ is proper and dense in
$s_{\gamma_n}^{R_nE_n}(Y_n)$, since its rank exceeds $\gamma_n$.
The Banach norm on $R_nE_n$ transported from $E_n$ cannot be controlled
by a finite family of continuous seminorms on
$s_{\gamma_n}^{R_nE_n}(Y_n)$. Otherwise its inclusion would
be a topological embedding. Completeness would make the image closed
in the Fr\'echet stage, contradicting properness and density.

For $z\in R_nE_n$, write
$\widehat z=(0,\ldots,2^{-n}z,\ldots)\in A$ and transport the norm by
$\norm{R_nx}_{R_nE_n}=\norm{x}_{E_n}$; then
$\norm{\widehat z}_A=\norm{z}_{R_nE_n}$. Pull back the finitely many
norms of the global stages $H_{\gamma_j}$, $j\leq n$, to this block
stage, using the continuous coordinate inclusion and the continuous
inclusions of later stages into earlier ones. If
\[
 \norm{z}_{R_nE_n}\leq
 2^n\max_{j\leq n}\norm{\widehat z}_{H_{\gamma_j}}
 \qquad(z\in R_nE_n),
\]
then these finitely many continuous seminorms would control the block
norm, contrary to the preceding observation. Hence, after
normalization, there is $z_n\in R_nE_n$ such that
$v_n=\widehat z_n$ has $\norm{v_n}_A=1$ and
\[
 \norm{v_n}_{H_{\gamma_j}}<2^{-n}\qquad(j\le n).
\]
For fixed $j$ this gives convergence in $H_{\gamma_j}$.
Cofinality and continuity give convergence in every $H_\gamma$ with
$\gamma<\lambda$, including $G=H_0$.
The projective topology at $\lambda$ gives the last assertion.
\end{proof}

\subsection{A hyperplane preserving the earlier stages}
Write $v_n=J(\iota_nu_n)$, where $u_n\in E_n$ has norm one and
$\iota_n$ is the coordinate inclusion. By Hahn--Banach choose
$f_n\in E_n^*$ with $\norm{f_n}=1$ and $f_n(u_n)=1$.
The series
\[
 f(x)=\sum_{n\ge1}f_n(x_n)\qquad(x\in E)
\]
defines a bounded linear functional of norm one. Set
\begin{equation}\label{lim:eq:kernel}
 X_\lambda=\ker f,\qquad N=J(X_\lambda),\qquad
 T_\lambda=J|_{X_\lambda}.
\end{equation}
Then $X_\lambda$ is a closed hyperplane in a separable Banach space,
and $T_\lambda$ is compact and injective.
We also regard $f:A\to\R$ as a continuous linear functional,
so that $f(v_n)=1$.

\begin{lemma}[Equality of the towers]\label{lim:lem:towers}
For every $\gamma\le\lambda$,
\[
 s_\gamma^N(G)=H_\gamma=s_\gamma^A(G)
\]
with their canonical topologies.
\end{lemma}
\begin{proof}
For $a\in A$ put $P_n(a)=a-f(a)v_n$.
Then $P_n(a)\in N$, $\norm{P_n(a)}_A\le2\norm{a}_A$, and
$P_n(a)\to a$ in every $H_\gamma$ with $\gamma\le\lambda$.
At stage zero this makes $N$ dense in the dense subspace $A$ of $G$,
so $s_0^N(G)=G=H_0$.

Suppose equality is proved at $\gamma<\lambda$. In the common
space $H_\gamma$ let
\[
 C_A=\cl{B_A}{H_\gamma},\qquad C_N=\cl{B_N}{H_\gamma},
\]
where $N$ has the norm inherited from $A$. The correction maps give
\begin{equation}\label{lim:eq:ballcomparison}
 C_N\subseteq C_A\subseteq2C_N.
\end{equation}
Indeed, for $a\in B_A$, the vectors $P_n(a)$ belong to $2B_N$ and
converge to $a$ in $H_\gamma$; now take closure.
Thus the two finite-gauge spaces coincide and their norms satisfy
$p_A\le p_N\le2p_A$.

The $p_A$-closure of $A$ is $H_{\gamma+1}$, by Lemma~\ref{lim:lem:gauge}.
Since $\lambda$ is a limit ordinal, $\gamma+1<\lambda$.
Hence $v_n\to0$ also in $H_{\gamma+1}$,
equivalently in its $p_A$ norm. Therefore $P_n(a)\to a$ in that norm
for every $a\in A$.
It follows that $N$ and $A$ have the same closure in the common
finite-gauge space. Equivalent gauges identify their canonical
topologies, proving equality at $\gamma+1$.

At a limit ordinal $\mu\le\lambda$, take projective intersections
of the identical predecessor stages. This completes the induction.
\end{proof}

\begin{remark}
The density argument in the successor gauge is essential.
Comparison \eqref{lim:eq:ballcomparison} alone would not identify the next
stages. Convergence in $H_{\gamma+1}$, in addition to convergence in
$H_\gamma$, supplies the required density. The fact that $\lambda$
is a limit makes this available at every successor step of the induction.
\end{remark}

\subsection{The terminal difference class}\label{lim:sec:terminal}
Identify coordinates of $x\in H_\lambda$ with vectors $x_n\in E_n$,
so its coordinates in $G$ are $2^{-n}R_nx_n\in Y_n$.
This is well defined and each map $x\mapsto x_n$ is continuous:
the single-coordinate product formula shows that projection lands
in the block's $\lambda$-stage, which equals the Polish subgroup
$R_nE_n$ because $r_n<\lambda$. The inverse
$R_n^{-1}:R_nE_n\to E_n$ is continuous.

\begin{proposition}\label{lim:prop:terminal}
In $H_\lambda$,
\[
 N\in D(\PiB^0_2)(H_\lambda)\setminus\Sig^0_2(H_\lambda).
\]
\end{proposition}
\begin{proof}
The coordinate characterization of the $\ell_1$ sum gives
\[
 A=\left\{x\in H_\lambda:\sum_{n\ge1}\norm{x_n}_{E_n}<\infty\right\}.
\]
For each $m\ge1$ the set
\[
 F_m=\bigcap_{k\ge1}
 \left\{x\in H_\lambda:\sum_{n\le k}\norm{x_n}_{E_n}\le m\right\}
\]
is closed, and $A=\bigcup_m F_m$. Thus $A\in\Sig^0_2(H_\lambda)$.
The partial sums $q_k(x)=\sum_{n\le k}f_n(x_n)$ are continuous on
$H_\lambda$. Put
\[
 Q=\bigcap_{j,K\ge1}\ \bigcup_{k\ge K}
       \{x\in H_\lambda:|q_k(x)|<1/j\}.
\]
This is a $\PiB^0_2$ set. On $A$ the series converges absolutely;
for a convergent scalar sequence, arbitrarily late values arbitrarily
close to zero are equivalent to having limit zero. Therefore
\[
 N=A\cap Q\in D(\PiB^0_2)(H_\lambda).
\]
For every $n\geq2$, however,
\[
 v_1-v_n\in2B_N,\qquad
 v_1-v_n\longrightarrow v_1\text{ in }H_\lambda,\qquad f(v_1)=1.
\]
Thus $\cl{2B_N}{H_\lambda}$ contains the point $v_1\notin N$.
Scalar multiplication by $2$ is a homeomorphism of $H_\lambda$, so
$\cl{2B_N}{H_\lambda}=2\cl{B_N}{H_\lambda}$ and
$v_1/2\in\cl{B_N}{H_\lambda}\setminus N$. Lemma~\ref{lim:lem:ball}
gives $N\notin\Sig^0_2(H_\lambda)$.
\end{proof}

\begin{proof}[Proof of Theorem~\ref{thm:limit-difference}]
By Lemma~\ref{lim:lem:towers}, $s_\lambda^N(G)=H_\lambda$, and
$H_\lambda\ne N$ by Proposition~\ref{lim:prop:terminal}, which also
gives $N\in D(\PiB^0_2)(H_\lambda)\subseteq\PiB^0_3(H_\lambda)$.
By Theorem~\ref{Theorem:characterize-solecki} applied inside $H_\lambda$ and the successor recursion,
$s_{\lambda+1}^N(G)=s_1^N(H_\lambda)=N$. No earlier stage can equal $N$, since
then all subsequent stages would equal $N$. Hence the rank is exactly
$\lambda+1$. Theorem~\ref{Theorem:complexity}\textup{(2b)} shows that the complexity
class of $N$ in $G$ is $D(\PiB^0_{\lambda+1})$.
Identify the separable infinite-dimensional Hilbert space $G$ with $\ell_2$.
\end{proof}

\subsection{Consequences and scope}\label{sec:consequences}
We combine these constructions with the classical realizations to
obtain every class in the Banach list. A subsequent transfer argument
places them in any prescribed separable infinite-dimensional Banach target.
The multiplicative classes are realized by Saint Raymond's operators,
whose compactness we verify first.

\begin{lemma}\label{lim:lem:SR-compact}
For every countable ordinal $\alpha$, the injective operator
$u_{\alpha+1}$ of \cite[Th\'eor\`eme~21]{saint-raymond_espaces_1976}
has norm at most one and is compact.
\end{lemma}

\begin{proof}
We recall the recursion of
\cite[pp.~222--225]{saint-raymond_espaces_1976}; all norms are
supremum norms. For $\alpha=0$, $K_1=\N\cup\{a_1\}$ is the one-point
compactification of $\N$, and
\[
 u_1(x)(n)=\frac{x(n)}{n},\qquad u_1(x)(a_1)=0.
\]
If $\alpha=\beta+1$, then $K_{\alpha+1}$ is the disjoint sum of
$K_\alpha\times\N$ and $\N$, the point $a_{\alpha+1}$ is the point at
infinity of $\N$, and, for $\theta\in K_\alpha$ and $n\in\N$,
\[
 u_{\alpha+1}(x)(\theta,n)=\frac{u_\alpha(x(\cdot,n))(\theta)}{n},\qquad
 u_{\alpha+1}(x)(n)=\frac12\Bigl(\frac{x(n)}{n^2}-\frac{x(a_\alpha,n)}{n}\Bigr),
\]
and $u_{\alpha+1}(x)(a_{\alpha+1})=0$. If $\alpha$ is a limit, let
$(\beta_k)_{k\in\N}$ be the increasing sequence cofinal in $\alpha$ used
there, and put $H_k=K_{\beta_k+1}$, $v_k=u_{\beta_k+1}$ and
$b_k=a_{\beta_k+1}$. Then $K_{\alpha+1}$ is the disjoint sum of the
spaces $H_k$ and $\N$, $a_{\alpha+1}$ is the point at infinity of $\N$,
and
\[
 u_{\alpha+1}(x)(\theta)=\frac{v_k(x|_{H_k})(\theta)}{k}\quad(\theta\in H_k),
 \qquad
 u_{\alpha+1}(x)(k)=\frac12\Bigl(\frac{x(k)}{k^2}-\frac{x(b_k)}{k}\Bigr),
\]
with $u_{\alpha+1}(x)(a_{\alpha+1})=0$. These operators have norm at most
one \cite[pp.~222, 223, and~226]{saint-raymond_espaces_1976}; this also
follows from the formulas by induction.

We prove compactness by induction on $\alpha$. The operator $u_1$ is the
limit in operator norm of its truncations to the coordinates
$n\leq N$, which have finite rank and are at distance at most $1/(N+1)$
from $u_1$. Let $\alpha\geq1$, and suppose that the operators of the
earlier stages are compact. For $N\in\N$, let $t_N(x)$ agree with
$u_{\alpha+1}(x)$ on the blocks and scalar coordinates of index $n\leq N$
and vanish elsewhere. If $\norm{x}\leq1$, each omitted block coordinate
has absolute value at most $1/n$, and each omitted scalar coordinate at
most $\frac12(1/n^2+1/n)\leq1/n$, where $n>N$. Hence
$\norm{u_{\alpha+1}-t_N}\leq1/(N+1)$. Each block of $t_N$ is the
composition of a restriction map of norm one, a compact operator of an
earlier stage, and an isometric inclusion, and the retained scalar
coordinates have finite rank. Thus each $t_N$ is compact, and so is
$u_{\alpha+1}$.
\end{proof}

\begin{corollary}\label{cor:banach-realizations}
For every $\xi\in\mathcal I_{\mathrm B}$ there are separable
Banach spaces $X,Y$ and a compact injective operator $T:X\to Y$
with proper range and $\ell_{\boldsymbol\Delta}(T(X);Y)=\xi$. Thus every index
in the Banach classification of
Theorem~\ref{Corollary:complexity-Banachable} has a compact realization.
\end{corollary}
\begin{proof}
A one-dimensional subspace of $\ell_2$ gives the closed case.
The additive cases are Theorem~\ref{thm:additive-ranges}. The
multiplicative cases follow from Saint Raymond's construction
\cite[Th\'eor\`eme~21 and Corollaire~24, pp.~221--227]{saint-raymond_espaces_1976},
as formulated in \cite[Theorem~2.8]{BDP}: the exact lower bounds come
from the continuous reductions of
\cite[Th\'eor\`eme~23 and Corollaire~24]{saint-raymond_espaces_1976},
and the operators are compact by Lemma~\ref{lim:lem:SR-compact}. The
multiplicative indices of \cite{BDP} are translated as explained in
Subsection~\ref{sec:classical-inputs}.

For $\beta=1+\lambda+n+2$, take compact injective operators
$T_\Sigma:X_\Sigma\to Y_\Sigma$ and $T_\Pi:X_\Pi\to Y_\Pi$ whose
ranges have complexity classes $\Sig^0_\beta$ and $\PiB^0_\beta$,
respectively. The product operator
\[
 T_\Sigma\times T_\Pi:X_\Sigma\times X_\Pi
   \longrightarrow Y_\Sigma\oplus_{\ell_2}Y_\Pi
\]
is compact and injective, the \(\ell_2\)-sum being identified
with the product Banach space. Its range is the product of the two
ranges. Writing $R_\Sigma=T_\Sigma(X_\Sigma)$ and
$R_\Pi=T_\Pi(X_\Pi)$, one has
\[
 R_\Sigma\times R_\Pi
   =(R_\Sigma\times Y_\Pi)\cap(Y_\Sigma\times R_\Pi),
\]
so the product is the intersection of a $\Sig^0_\beta$ set and a
$\PiB^0_\beta$ set. Equivalently,
\[
 R_\Sigma\times R_\Pi
   =(Y_\Sigma\times R_\Pi)\setminus
    \bigl((Y_\Sigma\setminus R_\Sigma)\times R_\Pi\bigr)
\]
is a difference of two $\PiB^0_\beta$ sets. If the product
belonged to $\PiB^0_\beta$, a section over a point of $R_\Pi$
would place $R_\Sigma$ in $\PiB^0_\beta$. If it belonged to
$\Sig^0_\beta$, a section over a point of $R_\Sigma$ would place
$R_\Pi$ in $\Sig^0_\beta$. Thus the product lies in neither class.
Lemma~\ref{Lemma:difference} then excludes the dual difference class, so
the complexity class of the product range is $D(\PiB^0_\beta)$.
This is the product construction of \cite[Proposition~2.10]{BDP},
with the additive factor supplied by
Theorem~\ref{thm:additive-ranges}.

The base class $D(\PiB^0_2)$ is realized by the bounded-variation
space in Example~\ref{ex:bv}. For a Banach target, take
\[
 T:\ell_1\to\ell_2,\qquad (Ta)_n=2^{-n}\sum_{k\ge n}a_k,
\]
whose range is the weighted version of $\mathrm{bv}_0$.
Its finite-coordinate truncations converge in operator norm,
since the $\ell_2$ norm of the output tail is bounded by
$(\sum_{n>m}4^{-n})^{1/2}\|a\|_1$; hence $T$ is compact.
For $y\in\ell_2$, put $z_n=2^ny_n$. The conditions
\[
 \sum_n|z_n-z_{n+1}|<\infty,\qquad \liminf_n|z_n|=0
\]
are respectively $\Sig^0_2$ and $\PiB^0_2$ in $\ell_2$ and
together describe the range. Since $Te_k$ converges in $\ell_2$
to $(2^{-n})_n$, which is outside the range, the closed-ball test
excludes $\Sig^0_2$. Density and properness exclude $\PiB^0_2$;
Lemma~\ref{Lemma:difference} then shows that the complexity class of the range
in $\ell_2$ is $D(\PiB^0_2)$.
The remaining difference cases are Theorem~\ref{thm:limit-difference}.
\end{proof}

\begin{corollary}[Every infinite-dimensional Banach target]
\label{cor:banach-fixed-target}
Let $Y$ be a separable infinite-dimensional Banach space.
For every $\xi\in\mathcal I_{\mathrm B}$ there are a separable
Banach space $X$ and a compact injective operator $S:X\to Y$ with
$\ell_{\boldsymbol\Delta}(S(X);Y)=\xi$.
For $\xi>0$ the range can be chosen dense.
In particular, every index in $\mathcal I_{\mathrm B}$ occurs in
$c_0(\mathbb N;Z)$ for every nontrivial separable Banach space $Z$.
\end{corollary}
\begin{proof}
The complexity class of a one-dimensional subspace of $Y$ in $Y$
is $\PiB^0_1$.
For a nonclosed class, take a compact injective operator $T:X\to Y_0$
from Corollary~\ref{cor:banach-realizations}, and replace $Y_0$ by
$\overline{T(X)}$. This does not change the complexity class of $T(X)$.
The space $Y_0$ is separable and infinite-dimensional.
By \cite[Proposition~2.1, p.~44]{BDP}, applied with domain $Y_0$,
target $Y$, and the norming subspace $N=Y_0^*$, there is a compact
injective operator $P:Y_0\to Y$ with dense range. Then $S=P\circ T$ is compact, injective,
and dense, with the same domain $X$.

We prove directly that this composition preserves the complexity class of the range.
Put
\[
 A=T(X),\qquad K_m=\overline{T(mB_X)}^{\,Y_0},
 \qquad U=\bigcup_{m\ge1}K_m.
\]
Each $K_m$ is compact and $A\subseteq U$.
For any $B\subseteq U$ belonging to $\Sig^0_\beta(Y_0)$, $\beta\ge2$,
the identity
\[
 P(B)=\bigcup_{m\ge1}P(B\cap K_m)
\]
shows that $P(B)\in\Sig^0_\beta(Y)$: the restriction of $P$
to $K_m$ is a homeomorphism onto the compact set $P(K_m)$,
and $\Sig^0_\beta$ is closed under countable unions and passage
from the closed subspace $P(K_m)$ to $Y$.

If $B\subseteq U$ belongs to $\PiB^0_\beta(Y_0)$ with $\beta\ge3$,
then $U\setminus B\in\Sig^0_\beta(Y_0)$, so the preceding
paragraph gives $P(U\setminus B)\in\Sig^0_\beta(Y)$.
The set $P(U)$ is $\sigma$-compact, hence $\Sig^0_2$, and
$Y\setminus P(U)\in\PiB^0_2\subseteq\Sig^0_\beta$.
Injectivity of $P$ yields the complement identity
\[
 Y\setminus P(B)
   =(Y\setminus P(U))\cup P(U\setminus B)\in\Sig^0_\beta(Y),
\]
and therefore $P(B)\in\PiB^0_\beta(Y)$.

For $D(\PiB^0_\beta)$ with $\beta\ge3$, write
$A=B_1\setminus B_2$ with $B_i\in\PiB^0_\beta(Y_0)$.
We bound the images of the intersections with $U$ before taking their
difference. Since $A\subseteq U$, injectivity gives
\[
 P(A)=P(U\cap B_1)\setminus P(U\cap B_1\cap B_2).
\]
For $\beta\ge3$ one has $U\in\Sig^0_2(Y_0)\subseteq\Sig^0_\beta(Y_0)$,
so $U\setminus B_1\in\Sig^0_\beta(Y_0)$. The additive transfer
already proved yields $P(U\setminus B_1)\in\Sig^0_\beta(Y)$.
The complement identity
\[
 Y\setminus P(U\cap B_1)
   =(Y\setminus P(U))\cup P(U\setminus B_1)
   \in\Sig^0_\beta(Y)
\]
therefore gives $P(U\cap B_1)\in\PiB^0_\beta(Y)$.
The same argument with $B_1\cap B_2$ in place of $B_1$ puts
$P(U\cap B_1\cap B_2)$ in $\PiB^0_\beta(Y)$. Hence
$P(A)\in D(\PiB^0_\beta)(Y)$.

For $D(\PiB^0_2)$, write $A=F\cap Q$ with $F\in\Sig^0_2(Y_0)$
and $Q\in\PiB^0_2(Y_0)$. The sets $F\cap U$ and
$(F\cap U)\setminus A=F\cap U\cap(Y_0\setminus Q)$ are
$\sigma$-compact. Their continuous images are $\sigma$-compact,
and injectivity gives
\[
 P(A)=P(F\cap U)\setminus P\bigl((F\cap U)\setminus A\bigr).
\]
This is the intersection of a $\Sig^0_2$ set and a $\PiB^0_2$ set.
We have therefore transferred every required upper bound.
All exact lower bounds, including exclusion of the dual difference
class, transfer by the continuous-preimage identity
$A=P^{-1}(P(A))$.

Finally, $c_0(\N;Z)$ is separable and infinite-dimensional for every
nontrivial separable Banach space $Z$. For this target one can also
use the explicit compact injection
\[
 P(y)=\bigl(2^{-n}h_n(y)z_0\bigr)_{n\ge1},
\]
where $\norm{z_0}=1$ and $(h_n)\subseteq B_{Y_0^*}$ is a countable
norming family. This formula suffices for existence in $c_0(\N;Z)$;
the preceding general construction additionally gives density.
\end{proof}

The domain $X_\lambda$ contains a closed copy of $\ell_1$.
Indeed, the closed linear span of $(\iota_nu_n)$ in $E$ is an isometric copy of
$\ell_1$, and its intersection with $\ker f$ is the zero-sum hyperplane,
again isomorphic to $\ell_1$. Consequently, $X_\lambda^*$ and
$X_\lambda^{**}$ are nonseparable.
Theorem~\ref{thm:limit-difference} answers
\cite[Question~2.11, p.~47]{BDP} with a Banach domain and a compact
operator: at a positive limit
$\lambda$, its multiplicative index $\lambda+1$ corresponds to
$\PiB^0_{\lambda+1}$ under the historical convention recorded in
Subsection~\ref{sec:classical-inputs}. The question itself permits
a Fr\'echet domain; Theorem~\ref{Theorem:complexity-Frechetable}
already answers it in that generality. The present construction
provides a Banach domain and a compact operator.
The realizations in $c_0(\N;Z)$, for every nontrivial separable Banach
space $Z$, are obtained from Corollary~\ref{cor:banach-fixed-target}.
The remaining questions of \cite{BDP} on the Borel and topological
classification of operator images are treated in \cite{LOI}.

\subsection{Completion of the classification theorems}\label{sec:final-proofs}
We finish by combining the upper restrictions with the realization
theorems to prove the four classifications stated in the Introduction.

\begin{proof}[Proofs of Theorems~\ref{Corollary:complexity}--\ref{Corollary:complexity-Banachable} and~\ref{Theorem:complexity-Banachable}]
Translate between indices and their exact pointclasses using
Table~\ref{tab:fractional-classes}.
For a nonopen Polish subgroup, Theorem~\ref{Theorem:complexity},
together with the discussion following it when $\rho-1$ is a
successor, gives Items~(1) and~(2) of
Theorem~\ref{Corollary:complexity}. By these items the complexity
class of a nonopen Polish subgroup determines its Solecki rank, so for
the converse it suffices to realize each class. The abelian
constructions of Theorem~\ref{Theorem:complexityP}, with a countable
discrete abelian base replaced first by its countable power, realize
every index in $\mathcal I_{\mathbb Z}$ in $\mathbb Z^{\mathbb N}$.
Closed nonopen subgroups give the first closed level.
Theorem~\ref{thm:additive-ranges} supplies every remaining successor
additive class in a Hilbert target. This proves
Theorem~\ref{Corollary:complexity}.

The non-Archimedean conclusion of Theorem~\ref{Theorem:complexity}
excludes the additional additive classes and the first difference
class after each limit. The non-Archimedean construction in
Theorem~\ref{Theorem:complexitynAP}, again with base $\mathbb Z$,
realizes every remaining class. This proves
Theorem~\ref{Corollary:complexity-nonArchimedean}.

A proper range of a linear map between separable Fr\'echet spaces
is a nonopen Polish subgroup. Thus the general upper list
applies. Theorem~\ref{Theorem:complexity-Frechetable} realizes
every index in $\mathcal I_{\mathbb Z}$ by linear maps; Theorem~\ref{thm:additive-ranges}
realizes every additional class by maps between separable Banach
spaces, hence also between separable Fr\'echet spaces.
This proves Theorem~\ref{Corollary:complexity-Frechetable}.

For a Banach subspace, the general upper list applies, and
Corollary~\ref{cor:banachable-limit-collapse} excludes the exact
classes $\boldsymbol\Pi^0_\lambda$ at positive limits $\lambda$.
The classes left are precisely the Banach list.
Corollary~\ref{cor:banach-realizations} realizes every one by a
compact injective operator between separable Banach spaces.
Corollary~\ref{cor:banach-fixed-target} places them in every
prescribed separable infinite-dimensional Banach target, in
particular in $c_0(\mathbb N;Z)$ for any nontrivial separable
Banach space $Z$. These results prove
Theorems~\ref{Corollary:complexity-Banachable}
and~\ref{Theorem:complexity-Banachable}.
\end{proof}
\appendix

\section{Banach tree spaces}\label{app:tree-spaces}

This appendix gives explicit Banach-space models of the
multiplicative classes, with all their Solecki stages computed.
Saint Raymond constructed Banach operator images of these classes
\cite[Th\'eor\`eme~21 and Corollaire~24]{saint-raymond_espaces_1976},
and his operators are the ones used in
Corollary~\ref{cor:banach-realizations}; the main theorems do not
depend on this appendix. Subsection~\ref{sec:tree-banach-spaces}
defines the spaces, Subsection~\ref{sec:tree-approximation}
constructs norm-controlled finite truncations,
Subsection~\ref{sec:relative-multiplicative} computes the relative
class of each stage, and Subsection~\ref{sec:tree-solecki-stages}
identifies the Solecki stages and the complexity classes.
Subsection~\ref{sec:banach-limit} shows that at a limit stage the
condition on all lower rows cannot be replaced by a condition on the
top row. For $\alpha\geq2$ the inclusion $X_\alpha\to X_0$ is not
compact. Indeed, fix a node of rank~$2$ and a vector $z\in Z$ of norm
one, and for each child $t$ of that node let $x_t$ have the leaf value
$z$ at the first child of $t$ and all other leaf values zero. These
elements have norm one in $X_\alpha$ and are at distance one from one
another in $X_0$.

\subsection{The spaces}
\label{sec:tree-banach-spaces}
Fix $0<\alpha<\omega_1$ and a nontrivial separable Banach
space $Z$. For readability, write $T=I^\alpha$ for the rooted tree
of node indices, put $r(v)=\gamma$ for $v\in I_\gamma^\alpha$,
and write $v[k]$ for the child
$(k,\boldsymbol n;\gamma_k,\boldsymbol\beta)$ of
$v=(\boldsymbol n;\boldsymbol\beta)$. Thus the root has rank
$\alpha$, and $r(v[k])=r(v)_k$.
Write $T_{\leq\gamma}$ and $T_{<\gamma}$ for the corresponding
sets of nodes. For every $v$ and every $\gamma<r(v)$, the set
\[
 \{k:r(v[k])<\gamma\}
\]
is finite. Each node has finitely many ancestors. Every
$\gamma\leq\alpha$ occurs as a node rank: from a node of larger
rank, choose a child of rank at least $\gamma$ and continue;
the descent terminates because ordinals admit no infinite
strictly decreasing sequence.

The row at a nonleaf $v$ is
$\boldsymbol{x}(v)=(kx(v[k]))_{k\in\N}$, and $x(v)$ is its limit
when this exists. At a leaf, the row is constantly $x(v)$.
Set
\[
 C_\gamma=\{(v,k):r(v[k])<\gamma<r(v)\}.
\]
These are the coordinates crossing level $\gamma$. With $X_0=c_0(T_0;Z)$ and
$X_{<\gamma}=\bigcap_{\xi<\gamma}X_\xi$, the recursive definition is
\begin{equation}\label{eq:tree-full-definition}
\begin{split}
 X_\gamma=\{x\in X_{<\gamma}:\;&
   (\boldsymbol{x}(v))_{v\in T_{\leq\gamma}}
            \in c_0(T_{\leq\gamma};c(\N;Z)),\\
 & (kx(v[k]))_{(v,k)\in C_\gamma}\in c_0(C_\gamma;Z)\},
  \qquad 0<\gamma\leq\alpha.
\end{split}
\end{equation}
The first condition requires the rows up to level $\gamma$ to
converge and to be uniformly small off finite sets of nodes; the
second controls the coordinates crossing level $\gamma$, and makes
the norms increase with $\gamma$ (Lemma~\ref{lem:tree-banach}).
Both conditions are imposed at every stage, including limit stages.
Denote by
$\|x\|_\gamma$ the maximum of the two displayed $c_0$ norms;
$\|x\|_0$ is the usual norm on $c_0(T_0;Z)$.

\begin{lemma}\label{lem:tree-banach}
For every $\gamma\leq\alpha$, $X_\gamma$ is a separable Banach
space. For $\xi<\gamma$ and $x\in X_\gamma$,
\[
 \|x\|_\xi\leq\|x\|_\gamma.
\]
Every node value of rank at most $\gamma$ is continuous on
$X_\gamma$, and $X_{<\gamma}$ has its projective Fr\'echet topology.
\end{lemma}
\begin{proof}
The row coordinates contributing to $\|x\|_\xi$ also occur at
level $\gamma$. For a crossing coordinate $(v,k)\in C_\xi$,
either $r(v)\leq\gamma$, in which case
$\|kx(v[k])\|\leq\|\boldsymbol{x}(v)\|_\infty$, or
$r(v)>\gamma$, in which case $(v,k)\in C_\gamma$.
This proves the norm inequality, as in \cite[Lemma~10.7]{L}.
Node evaluation is bounded by the corresponding row norm.

Proceed by induction on $\gamma$. A Cauchy sequence in $X_\gamma$
is Cauchy in every earlier $X_\xi$ by the inequality just proved.
Completeness at earlier stages gives compatible limits, hence a
limit $x\in X_{<\gamma}$. Its two families in
\eqref{eq:tree-full-definition} are the norm limits of the
corresponding families of the Cauchy sequence: each individual
coordinate is continuous on $X_{<\gamma}$, since it is formed
from node values of rank strictly below $\gamma$.
The two $c_0$ spaces are complete, so both limiting families
belong to those spaces and the convergence is in $\|\cdot\|_\gamma$.
Thus $X_\gamma$ is complete. The map to the product of the two
$c_0$ spaces is an isometry. Both factors are separable, so
$X_\gamma$ is separable. Finally, a countable projective
intersection of the earlier Banach spaces is Fr\'echet.
\end{proof}

\subsection{Finite-subtree truncations and norm-controlled approximation}
\label{sec:tree-approximation}
For a finite set $F\subseteq T$, put
$\downarrow F=\{u:u\leq v\text{ for some }v\in F\}$,
where $u\leq v$ means that $u$ is a descendant of $v$ or equals $v$.
Given leaf coordinates $x$, let $P_Fx$ agree with $x$ on the
leaves in $\downarrow F$ and vanish on the other leaves.

\begin{lemma}\label{lem:tree-truncation}
Let $0\leq\gamma\leq\alpha$, $x\in X_\gamma$, and
$F\subseteq T_{\leq\gamma}$ be finite. Then $P_Fx\in X_\alpha$.
At every node its value is $x(v)$ if $v\in\downarrow F$, and
zero otherwise. For every $\xi\leq\gamma$,
\[
 \|P_Fx\|_\xi\leq\|x\|_\xi.
\]
If $\gamma>0$ and $F\subseteq T_{<\gamma}$, then also
\begin{equation}\label{eq:tree-truncation-bound}
 \|P_Fx\|_\alpha\leq\|x\|_\gamma.
\end{equation}
\end{lemma}
\begin{proof}
Put $D=\downarrow F$ and $y=P_Fx$. Induction on node rank proves
the asserted formula for the values. At a node in $D$, the entire
subtree is retained, so its row agrees with the original row.
At a node outside $D$, its children can have nonzero values only
when they belong to $D$. There are only finitely many such edges:
the child must be a maximal node of $D$, hence a maximal member
of $F$. Thus a row outside $D$ is finitely supported and has
limit zero. Only finitely many of these rows can be nonzero.

We check membership at every stage, without discarding either
$c_0$ condition. If $\xi\leq\gamma$, each row of $y$ at rank
at most $\xi$ is obtained by setting some entries of the
corresponding row of $x$ to zero. It is convergent by the preceding
paragraph, and its norm is no larger. Each crossing coordinate
at level $\xi$ is likewise either unchanged or zero. Both
families therefore satisfy their $c_0$ conditions, and
$\|y\|_\xi\leq\|x\|_\xi$.
If $\xi>\gamma$, the rows inside $D$ form a subfamily of the
$c_0$ family of rows of $x$ at level $\gamma$; the remaining
nonzero rows are the finitely many boundary rows just described.
A nonzero crossing coordinate at level $\xi$ must cross from
outside $D$ to a maximal node of $D$, so there are only finitely
many. This proves $y\in X_\xi$ at all remaining stages.

Suppose now that $F\subseteq T_{<\gamma}$. Rows inside $D$,
and boundary rows whose parent has rank at most $\gamma$, are
bounded by $\|x\|_\gamma$. For a nonzero entry in a boundary
row of rank greater than $\gamma$, the retained child has rank
strictly below $\gamma$. The entry is therefore a crossing
coordinate in $C_\gamma$, and has the same bound. Since
$C_\alpha$ is empty, this proves
\eqref{eq:tree-truncation-bound}.
\end{proof}

\begin{lemma}\label{lem:tree-density-balls}
The space $X_\alpha$ is dense in $X_\gamma$ for every
$\gamma<\alpha$. For $0<\gamma<\alpha$ and every $r>0$,
\begin{equation}\label{eq:tree-ball-closure}
 B_\gamma(r)\subseteq
       \overline{B_\alpha(r)}^{\,X_{<\gamma}},
 \qquad B_\xi(r)=\{x\in X_\xi:\|x\|_\xi\leq r\}.
\end{equation}
\end{lemma}
\begin{proof}
Fix $x\in X_\gamma$ and $\varepsilon>0$. Choose a finite set
$E\subseteq T_{\leq\gamma}$ outside which the row norms of $x$
are at most $\varepsilon$, and a finite set $A\subseteq C_\gamma$
outside which the crossing-coordinate norms are at most
$\varepsilon$. For $\gamma=0$, take $A=\varnothing$ and let $E$ consist of the
finitely many leaf coordinates of norm greater than $\varepsilon$. Put
\[
 F=E\cup\{v[k]:(v,k)\in A\}.
\]
All rows indexed by $E$ are retained in full by $P_F$; all crossing
coordinates indexed by $A$ are also retained. In the other rows
and crossing coordinates, the error consists of entries of $x$
that have been set to zero, and is bounded by $\varepsilon$.
Thus $\|x-P_Fx\|_\gamma\leq\varepsilon$.
Lemma~\ref{lem:tree-truncation} gives $P_Fx\in X_\alpha$,
proving density.

For \eqref{eq:tree-ball-closure}, let $x\in B_\gamma(r)$.
A neighborhood of $x$ in $X_{<\gamma}$ contains a neighborhood
given by $\|x-y\|_\xi<\varepsilon$ for some $\xi<\gamma$;
this follows from monotonicity of the earlier norms. Carry out
the preceding truncation at level $\xi$, with error less than
$\varepsilon$. The resulting finite set $F$ is contained in
$T_{\leq\xi}\subseteq T_{<\gamma}$. Applying
\eqref{eq:tree-truncation-bound} to $x\in X_\gamma$ gives
$P_Fx\in B_\alpha(r)$, proving the closure assertion.
\end{proof}

\subsection{Relative multiplicative complexity at every stage}
\label{sec:relative-multiplicative}

The defining conditions give a relative multiplicative upper bound
at each stage. An insertion map from a weighted sequence space
provides the matching lower bound, through the following fact.

\begin{lemma}\label{lem:c0-cube-trace}
The set $c_0\cap[-1,1]^{\mathbb N}$ is not
$\boldsymbol\Sigma^0_3$ in the product cube.
\end{lemma}
\begin{proof}
Indeed, if this trace were $\boldsymbol{\Sigma}^0_3$ in the cube,
it would be $\boldsymbol{\Sigma}^0_3$ in $\R^{\N}$, since the cube is closed. Taking its countably many positive
integer dilates would make $c_0$ a $\boldsymbol{\Sigma}^0_3$ subgroup of
$\R^{\N}$. Since $c_0$ is $\boldsymbol{\Pi}^0_3$ in $\R^{\N}$,
Proposition~\ref{prop:pi3-local}, applied to the supremum-norm balls
$V_m=c_0\cap[-1/m,1/m]^{\N}$, would make some $V_m$ a $G_\delta$
subset of $\R^{\N}$. But $V_m$ is dense and meager in the compact
cube $[-1/m,1/m]^{\N}$: it is contained in the union over $k$ of
the closed nowhere dense sets on which $|x_n|\leq1/(2m)$ for all
$n\geq k$. Thus $V_m$ cannot be $G_\delta$ in that cube, a
contradiction.
\end{proof}

\begin{lemma}\label{lem:tree-relative-pi3}
For every $0<\gamma\leq\alpha$, the complexity class of
$X_\gamma$ in $X_{<\gamma}$ is $\boldsymbol{\Pi}^0_3$.
\end{lemma}
\begin{proof}
First, $X_\gamma\in\boldsymbol{\Pi}^0_3(X_{<\gamma})$ follows
directly from \eqref{eq:tree-full-definition}.
For a countable set $I$, a family $(z_{i,k})$ belongs to
$c_0(I;c(\N;Z))$ precisely when each row converges and
\[
 \forall m\in\N\ \exists F\subseteq I\text{ finite}\quad
 \forall i\notin F\ \forall k\in\N\quad
 \|z_{i,k}\|\leq2^{-m}.
\]
Row convergence is the Cauchy condition, hence is
$\boldsymbol{\Pi}^0_3$ in the product topology. For fixed $m,F$,
the remaining conditions are closed, so the displayed condition
is also $\boldsymbol{\Pi}^0_3$. Likewise $c_0(C_\gamma;Z)$ is
$\boldsymbol{\Pi}^0_3$ in the coordinatewise topology.
The coordinate maps in \eqref{eq:tree-full-definition} are
continuous on $X_{<\gamma}$ by Lemma~\ref{lem:tree-banach}.

For the lower bound we give a full insertion construction.
If $u\leq w$, let $p(w,u)$ be the product of the child indices
along the path from $w$ to $u$, with $p(w,w)=1$.
For every fixed $w$ and $R\geq1$,
\begin{equation}\label{eq:tree-finite-products}
 \{u\leq w:p(w,u)\leq R\}\text{ is finite}.
\end{equation}
Indeed, induction on $r(w)$ applies: apart from $w$, the set is
a union over the finitely many children $w[k]$ with $k\leq R$
of the corresponding finite sets with bound $R/k$.

Choose a node $v$ of rank $\gamma$. For $t=(t_k)\in c_0(\N;Z)$,
define leaf coordinates $\Phi_v(t)$ by
\[
 \Phi_v(t)(u)=
 \begin{cases}
  t_k/p(v[k],u),&u\leq v[k]\text{ for some }k,\\
  0,&u\not\leq v.
 \end{cases}.
\]
The same formula gives the values at all proper descendants of
$v$. At each such node the row is constant with that value;
all node values outside the subtree of $v$ that are defined
below rank $\gamma$ are zero. Formula
\eqref{eq:tree-finite-products} and $t_k\to0$ show that the
family of norms of these constant rows is a $c_0$ family.

Fix $\xi<\gamma$. A crossing coordinate whose parent is a
proper descendant of $v$ equals the value at that parent.
Each parent has only finitely many crossing children, so this
family is again $c_0$. The only other possibly nonzero crossing
coordinates have parent $v$; there are finitely many, namely
those with $r(v[k])<\xi$, and their values are $kt_k$.
Consequently $\Phi_v(t)\in X_\xi$ and
\[
 \|\Phi_v(t)\|_\xi\leq
 \max\bigl\{\|t\|_\infty,
          \max_{r(v[k])<\xi}k\|t_k\|\bigr\},
\]
where a maximum over an empty set is zero. This proves that
$\Phi_v:c_0(\N;Z)\to X_{<\gamma}$ is continuous and linear.
At level $\gamma$, all defining conditions are already satisfied
except convergence of the row at $v$. That row is $(kt_k)_k$.
Thus
\begin{equation}\label{eq:tree-insertion}
 \Phi_v(t)\in X_\gamma
       \quad\Longleftrightarrow\quad (kt_k)_k\text{ converges in }Z.
\end{equation}

Choose $z\in Z$ with $\|z\|=1$. The map
$[-1,1]^\N\to c_0(\N;Z)$ defined by
\[
 t_{2j}=\frac{a_jz}{2j},\qquad t_{2j-1}=0
\]
is continuous for the product topology on the cube and the norm
topology on $c_0$. By \eqref{eq:tree-insertion}, the preimage
of $X_\gamma$ under its composition with $\Phi_v$ is
$c_0\cap[-1,1]^\N$. Lemma~\ref{lem:c0-cube-trace} shows that this trace
is not $\boldsymbol{\Sigma}^0_3$. Hence
$X_\gamma\notin\boldsymbol{\Sigma}^0_3(X_{<\gamma})$,
which proves the assertion at every stage.
\end{proof}

\subsection{The Solecki stages and the multiplicative realizations}
\label{sec:tree-solecki-stages}

The norm-controlled approximations and relative complexity estimates
now identify each Solecki stage. These identities determine the
ranks of the multiplicative models.

\begin{theorem}\label{thm:tree-solecki-stages}
For the spaces defined by \eqref{eq:tree-full-definition},
and every countable ordinal $\beta$ with $1+\beta\leq\alpha$,
\begin{equation}\label{eq:tree-solecki-stages}
 s_\beta^{X_\alpha}(X_0)
   =s_\beta^{X_{<\alpha}}(X_0)
   =X_{<(1+\beta)}.
\end{equation}
These are equalities of Fr\'echet subspaces with their canonical
topologies. If $\alpha=1+\rho$, then $X_\alpha$ and $X_{<\alpha}$
have Solecki ranks $\rho+1$ and $\rho$, respectively, in $X_0$.
\end{theorem}
\begin{proof}
We first show, for $0<\gamma<\alpha$, that
\begin{equation}\label{eq:tree-one-step}
 s_1^{X_\alpha}(X_{<\gamma})
    =s_1^{X_{<\alpha}}(X_{<\gamma})=X_\gamma.
\end{equation}
Apply Lemma~\ref{Lemma:characterize-solecki}. Its hypotheses for $N=X_\alpha$ and
$H=X_\gamma\subseteq X_{<\gamma}$ hold by
Lemmas~\ref{lem:tree-density-balls} and
\ref{lem:tree-relative-pi3}: $N$ is dense in $H$, $H$ is
$\boldsymbol{\Pi}^0_3$, and the closure in $X_{<\gamma}$ of
every zero neighborhood of $N$ contains a zero neighborhood of
$H$, by \eqref{eq:tree-ball-closure}.
For $N=X_{<\alpha}$, density follows from the density of
$X_\alpha$. A zero neighborhood of $X_{<\alpha}$ pulls back to
a zero neighborhood of $X_\alpha$, since the inclusion is
continuous. Its closure therefore has the same required property.
This proves \eqref{eq:tree-one-step} for both subspaces.

Now let $N$ be either $X_\alpha$ or $X_{<\alpha}$ and induct on
$\beta$ in \eqref{eq:tree-solecki-stages}.
For $\beta=0$, density gives $s_0^N(X_0)=X_0=X_{<1}$.
If $\beta=\delta+1$ and $1+\beta\leq\alpha$, then
$0<1+\delta<\alpha$, and
\[
 s_{\delta+1}^N(X_0)
  =s_1^N(X_{<(1+\delta)})
  =X_{1+\delta}=X_{<(1+(\delta+1))}
\]
by \eqref{eq:tree-one-step}. At a nonzero limit $\beta$,
\[
 s_\beta^N(X_0)
   =\bigcap_{\delta<\beta}X_{<(1+\delta)}
   =X_{<\beta}=X_{<(1+\beta)}.
\]
The topologies at limits are the projective topologies; at other
stages they agree by uniqueness of the compatible Polish group
topology on a Polish subgroup. This proves the stage identities,
including the endpoint $\beta=\alpha$ when $\alpha$ is infinite.

Write $\alpha=1+\rho$. Lemma~\ref{lem:tree-relative-pi3} gives
$X_\alpha\ne X_{<\alpha}$ and
$X_\alpha\in\boldsymbol{\Pi}^0_3(X_{<\alpha})$.
By Theorem~\ref{Theorem:characterize-solecki}, applied in $X_{<\alpha}$, this implies
$s_1^{X_\alpha}(X_{<\alpha})=X_\alpha$. Thus
\eqref{eq:tree-solecki-stages} gives
\[
 s_\rho^{X_\alpha}(X_0)=X_{<\alpha}\ne X_\alpha,
 \qquad s_{\rho+1}^{X_\alpha}(X_0)=X_\alpha.
\]
Its rank is $\rho+1$. The stage $\rho$ for $X_{<\alpha}$ is
itself. For $\beta<\rho$, the inclusion
$X_{<\alpha}\subseteq X_{1+\beta}$ and properness of
$X_{1+\beta}$ in $X_{<(1+\beta)}$ show that the earlier stage
is strictly larger. Hence its rank is $\rho$.
\end{proof}

\begin{corollary}\label{cor:tree-multiplicative}
Let $\lambda$ be zero or a countable limit ordinal and let
$n<\omega$. For $\alpha=1+\lambda+n$, the complexity class of
$X_\alpha$ in $X_0$ is
$\boldsymbol{\Pi}^0_{1+\lambda+n+2}$, and its Solecki rank is
$\lambda+n+1$. For a positive countable limit $\lambda$, in the
construction with $\alpha=\lambda$, the complexity class of
$X_{<\lambda}$ in $X_0$ is $\boldsymbol{\Pi}^0_\lambda$, and
its Solecki rank is $\lambda$.
\end{corollary}
\begin{proof}
For the first assertion the preceding Solecki subgroup is
$X_{<\alpha}$, by Theorem~\ref{thm:tree-solecki-stages} with
$\rho=\lambda+n$. Lemma~\ref{lem:tree-relative-pi3} implies
$X_\alpha\notin D(\boldsymbol{\Pi}^0_2)(X_{<\alpha})$,
since $D(\boldsymbol{\Pi}^0_2)\subseteq\boldsymbol{\Sigma}^0_3$.
Theorem~\ref{Theorem:complexity}(2a) now gives the stated
complexity class. At a positive limit, the rank assertion for
$X_{<\lambda}$ follows from the same theorem, and
Theorem~\ref{Theorem:complexity}(1) gives its complexity class.
\end{proof}

\subsection{The condition on all lower rows}
\label{sec:banach-limit}

The weighted top row is
\[
 q_\alpha(x)=\bigl(kx(k;\alpha_k)\bigr)_{k\in\mathbb N},
 \qquad x\in X_{<\alpha}.
\]
At a limit stage, the first condition in
\eqref{eq:tree-full-definition} concerns the rows at all nodes of
rank at most $\alpha$, not only the top row. The following example
shows that convergence of the top row does not imply this condition,
and that neither does a uniform bound on the lower norms.

\begin{example}\label{ex:banach-limit}
Take $Z=\R$, $\alpha=\omega$, and $\omega_k=k$ in the notation
of the present tree construction. There is a continuous map
\[
 [-1,1]^{\N}\longrightarrow X_{<\omega},\qquad a\longmapsto x(a),
\]
such that
\[
 \sup_{j<\omega}\|x(a)\|_{X_j}=\sup_k|a_k|,
 \qquad \boldsymbol{x(a)}(\omega)=0,
 \qquad x(a)\in X_\omega\ \Longleftrightarrow\ a\in c_0.
\]
In particular, for $a=(1,1,\ldots)$, the point $x(a)$ of $X_{<\omega}$
has bounded norms and zero weighted top row, but does not belong to
$X_\omega$.
\end{example}

\begin{proof}
Put $M_k=2^k$. Define $u^k$ on the leaf coordinates $I_0^\omega$,
supported only on the branch with top child $(k;k)$.
Its nonzero coordinates are
\[
 u^k(j_1,\ldots,j_{k-1},M_k,k;0,1,\ldots,k)
       =\frac{1}{M_kj_1\cdots j_{k-1}}
       \qquad(j_1,\ldots,j_{k-1}\in\N).
\]
For $k=1$ the product and the list of $j$'s are empty. All other
leaf coordinates are zero. Below the node $(M_k,k;k-1,k)$,
the weighted rows are constant with the prescribed value of their
parent. At the node $(k;k)$, the weighted row has just one nonzero
entry, equal to $1$ at index $M_k$, and has limit zero.
Every top-child value $u^k(\ell;\ell)$ is therefore zero.

For finite $j$, the crossing coordinates are zero: the only possible
crossing coordinates come from the root of degree $\omega$,
and its child values vanish. The norms of the rows of rank at most $j$
give directly
\[
 \|u^k\|_{X_j}=
 \begin{cases}
  2^{-k},&j<k,\\
   1,&j\geq k.
  \end{cases}.
\]
The required $c_0$ conditions hold within each branch: the product
lengths are bounded by its finite depth, so only finitely many row
coordinates can exceed any fixed positive threshold.

For $a\in[-1,1]^{\N}$, define $x(a)$ by the disjoint branch sum
$\sum_k a_ku^k$. For every finite $j$ this series converges in
$X_j$: after branch $K\geq j$, its tail norm is at most $2^{-(K+1)}$.
More precisely,
\begin{equation}\label{eq:limit-spike-norm}
 \|x(a)\|_{X_j}
 =\max\left\{\max_{1\leq k\leq j}|a_k|,
             \sup_{k>j}2^{-k}|a_k|\right\},
\end{equation}
with the first maximum interpreted as zero for $j=0$.
This also proves continuity from the product cube to every $X_j$,
hence to $X_{<\omega}$, and the asserted uniform norm equality.

The weighted top row is zero. At degree $\omega$, however, the
family of all lower-node rows of $x(a)$ contains the degree-$k$ rows at
$(k;k)$, whose norms are $|a_k|$. Thus its $c_0$ condition forces
$a_k\to0$. Conversely, if $a_k\to0$, the branch contributions
converge in the $X_\omega$ norm: each individual branch satisfies
the defining $c_0$ condition, and the tail norm is
$\sup_{k>K}|a_k|\to0$. Hence $x(a)\in X_\omega$.
\end{proof}

\section*{Acknowledgments and declarations}
We are grateful to Su Gao, Alexander Kechris, Andr\'e Nies,
and S\l awomir Solecki for remarks on a preliminary version of
\cite{L}, and to the anonymous referee of \cite{L} for a careful
reading. The work on \cite{L} was partially supported by Marsden Fund
Fast-Start Grant VUW1816, Rutherford Discovery Fellowship VUW2002,
GNSAGA of INDAM, and ERC Starting Grant 101077154, ``Definable
Algebraic Topology''.
The author was also supported by the University of Bologna.
Funded by the European Union. Views and opinions expressed are those
of the author and do not necessarily reflect those of the European
Union or the European Research Council Executive Agency; neither the
European Union nor the granting authority can be held responsible for them.

ChatGPT by OpenAI, OpenAI Codex, and Claude by Anthropic were used
during the preparation and revision of this manuscript for mathematical
exploration, proof development, source retrieval, editorial and structural
revision, and to assist in checking proofs, notation, terminology,
internal consistency, cross-references, and bibliographic citations.

\raggedbottom

\end{document}